\documentclass{scrartcl}

\usepackage{lmodern}
\usepackage[utf8]{inputenc}
\usepackage[T1]{fontenc}
\usepackage[english]{babel}
\usepackage{csquotes}
\usepackage{subcaption}
\usepackage{stmaryrd}
\usepackage[backend=biber,style=alphabetic,url=false,giveninits=true,sorting=anyt]{biblatex}
\usepackage{amsmath,amssymb,amsthm}
\usepackage{mathtools,thmtools,esint}
\usepackage{mathrsfs}

\usepackage[hypertexnames=false]{hyperref}
\hypersetup{colorlinks=true}

\usepackage[nameinlink]{cleveref}
\crefname{equation}{}{f}
\crefname{enumi}{}{}

\usepackage[svgnames]{xcolor}
\usepackage{graphicx}
\usepackage[shortlabels]{enumitem}
\usepackage{tikz}
\usepackage{bm}
\usepackage{bbold}

\numberwithin{equation}{section}

\declaretheorem[style=plain,numberwithin=section]{theorem}
\declaretheorem[style=plain,numberlike=theorem]{lemma, proposition, corollary}
\declaretheorem[style=definition,numberwithin=section]{definition}
\declaretheorem[style=remark,numberlike=definition]{remark,example}

\declaretheorem[name=Theorem A,numbered=no,style=plain]{theoremA}

\newtheorem*{theorem*}{Theorem}

\newcommand{\N}{\mathbb{N}}
\newcommand{\Z}{\mathbb{Z}}

\newcommand{\R}{\mathbb{R}}
\newcommand{\T}{\mathbb{T}}

\newcommand{\ball}{\mathbb{B}}
\newcommand{\Hn}[1]{H_{#1}E}

\newcommand{\OT}{\mathsf{OT}}

\newcommand{\mres}{\mathbin{\vrule height 1.6ex depth 0pt width
0.13ex\vrule height 0.13ex depth 0pt width 1.3ex}}
\newcommand{\eps}{\varepsilon}

\newcommand{\dd}{\mathop{}\mathopen{}\mathrm{d}}

\newcommand{\defeq}{\stackrel{\text{def.}}{=}}

\DeclareMathOperator{\spn}{span}

\DeclareMathOperator{\img}{Img}

\DeclareMathOperator{\vol}{vol}

\DeclareMathOperator{\dist}{dist}
\DeclareMathOperator{\id}{Id}

\DeclareMathOperator{\rank}{rank}

\DeclareMathOperator{\spt}{spt}
\DeclareMathOperator{\entropy}{Ent}

\DeclareMathOperator{\interior}{int}

\DeclareMathOperator{\x}{\textbf{x}}
\DeclareMathOperator{\proj}{\mathrm{proj}}

\DeclareMathOperator{\gibbs}{\mathbb{P}_\varepsilon}
\DeclareMathOperator{\unnormalizedgibbs}{\widetilde{\mathbb{P}}_\varepsilon}
\DeclareMathOperator{\flow}{\mathfrak m_{t,\varepsilon}}
\DeclareMathOperator{\hess}{Hess}

\DeclarePairedDelimiterX\set[2]{\{}{\}}{#1\,\delimsize\vert\,#2}
\DeclarePairedDelimiterX\inp[2]{\langle}{\rangle}%
{#1\,\delimsize\vert\,\mathopen{}#2}

\title{First-order \texorpdfstring{\(\Gamma\)-}{Γ-}expansion for entropic optimal transport with non-degenerate costs}
\author{Luca Nenna\thanks{Universit\'e Paris-Saclay, CNRS, Laboratoire de math\'ematiques d'Orsay, ParMA, Inria Saclay, 91405, Orsay, France. 
email: \texttt{luca.nenna@universite-paris-saclay.fr}}\;\thanks{Institut Universitaire de France, Paris, France.},\;
Brendan Pass\thanks{Department of Mathematical and Statistical Sciences, 632 CAB, University of Alberta, Edmonton, Alberta, Canada, T6G 2G1.
email: \texttt{pass@ualberta.ca}},\;
Paul Pegon\thanks{CEREMADE, CNRS UMR 7534, Université Paris-Dauphine – PSL, Paris, France \& MOKAPLAN, Inria Paris, Paris, France. Email: \texttt{pegon@ceremade.dauphine.fr}} \; and Louis Tocquec\thanks{Universit\'e Paris-Saclay, CNRS, Laboratoire de math\'ematiques d'Orsay \& MOKAPLAN, Inria Paris, 75016 Paris, France. Email: \texttt{louis.tocquec@universite-paris-saclay.fr}}}
\date\today

\begin{document}

\maketitle

\begin{abstract}
Entropy-regularized optimal transport (EOT) has become a key tool in optimal transport, providing both theoretical insight and efficient computation via the Sinkhorn algorithm. We study the first-order \(\Gamma\)-expansion of EOT as the regularization parameter \(\varepsilon\) vanishes, by rewriting EOT as the minimization of a relative entropy with respect to a Gibbs measure \(\gibbs\) associated with the duality gap energy. Using a mild extension of the Laplace method, we prove the convergence of the Gibbs measures and derive this \(\Gamma\)-expansion for non-degenerate ground costs, under suitable assumptions on the duality gap, the contact set and the marginals. This yields at the same time an entropic selection principle and an expansion of the EOT cost up to \(o(\varepsilon)\) for this class of ground costs, generalizing existing results.


\end{abstract}
% arXiv categories: math.OC (Optimization and Control, primary); cross-lists: math.AP (Analysis of PDEs), math.PR (Probability)
% HAL domains: Mathematics [math] / Optimization and Control [math.math-oc] (primary); Analysis of PDEs [math.math-ap]; Probability [math.math-pr]\tableofcontents

\section{Introduction}
\label{sec: intro}
Optimal transport (OT) has emerged as a central research topic at the intersection of probability theory, partial differential equations, and optimization, with applications in fields such as economics, machine learning, and image processing. The monographs of Villani \cite{villani2008optimal} and Santambrogio \cite{santambrogioOptimalTransportApplied2015} provide comprehensive accounts of its mathematical foundations. Applications to economics
are developed in \cite{galichon2016optimal}, while \cite{peyre2019computational} surveys computational methods and applications to data science. Large discrete transport problems can
nevertheless be expensive to solve, and high-dimensional statistical
estimation presents additional difficulties.

These computational challenges have renewed interest in entropy-regularized optimal transport (EOT). Entropic regularization leads to a strictly convex problem whose discrete solution can be computed by iterative matrix scaling, following the work of Sinkhorn \cite{sinkhornRelationshipArbitraryPositive1964,sinkhorn1967diagonal}. Its use for large-scale transport was popularized by Cuturi \cite{cuturi2013sinkhorn} and developed through iterative Bregman projections in \cite{benamou2015iterative}. Convergence results for the Sinkhorn algorithm range from the matrix-scaling analysis of Franklin and Lorenz \cite{franklin1989scaling} to continuous and multi-marginal settings \cite{marinoOptimalTransportApproach2020,eckstein2022quantitative}, and complexity guarantees for approximate discrete OT are established in \cite{altschuler2017near}. Statistical benefits at fixed
regularization have also been quantified for Sinkhorn divergences \cite{genevay2019sample} and entropic transport costs \cite{mena2019statistical}, with bounds that depend on the
regularization parameter. We refer to \cite{peyre2019computational,nutz2021introduction} for broader perspectives on EOT and its applications.

EOT also has a probabilistic interpretation through the Schrödinger problem and variational problems for stochastic processes \cite{mikami1990variational,leonardSurveySchrodingerProblem2014}, where \(\varepsilon\) plays the role of a noise or temperature parameter. It is therefore natural to investigate the behavior of the optimal plans and values as \(\varepsilon\) tends to zero. The zero-noise connection with quadratic OT was established by Mikami \cite{mikami2004monge}, and variational convergence was subsequently developed in \cite{leonard2012schrodinger,carlier2017convergence} in the two-marginal case, and in \cite{gerolin2020multimarginal} in the multi-marginal case, including singular repulsive costs. Convergence of the associated potentials is studied in \cite{nutz2022entropic}, and a complementary geometric approach is given by Bernton, Ghosal, and Nutz \cite{bernton2022geometry}, who describe concentration of the entropic optimizers through large-deviation estimates and relate the rate function to Kantorovich potentials.

Finer asymptotics of the optimal value have been studied in several
settings. For the quadratic cost, a next-order variational convergence of the entropic cost, viewed as a function of one marginal, arises in the connection between large deviations and Wasserstein gradient flows \cite{adams2011large,duong2013wasserstein,erbar2015large}. Pal \cite{pal2024difference} identifies the next asymptotic term for smooth strictly convex displacement costs under suitable regularity and elliptic assumptions on the potentials, while Conforti and Tamanini \cite{conforti2021formula} obtain differentiation formulas and a second-order expansion in the Schrödinger setting. For general \(C^2\) non-degenerate ground costs \(c\) (i.e.~such that the mixed Hessian \(D^2_{xy}c\) is invertible), Carlier, Pegon, and Tamanini \cite{carlier2023convergence} establish the leading correction \(\tfrac d2\varepsilon\log(1/\varepsilon)+O(\varepsilon)\), which is extended to the multi-marginal setting by Nenna and Pegon \cite{nenna2025convergence}. A different approach based on quantization of measures yields further convergence bounds in \cite{eckstein2024convergence}. The separate behavior of the transport and entropy terms appearing in the total EOT cost, together with quantitative convergence of the plans, is studied by Malamut and Sylvestre
\cite{malamut2025convergence}.

When the unregularized problem admits several minimizers, whether the entropic plans converge and, if so, which optimal plan is selected is an open problem in general. For discrete problems, the work of Cominetti and San Martín \cite{cominetti1994asymptotic} on exponential penalties in linear programming establishes convergence to the entropy-minimizing optimal solution, with exponential convergence rates, and non-asymptotic quantitative bounds were later obtained by Weed \cite{weed2018explicit}. In the continuous setting, a direct entropy selection argument still applies when an optimal plan has finite relative entropy \cite{nutz2021introduction}, but the main difficulty arises when every optimal plan is singular with respect to the product of the marginals, so that the relative entropy is infinite on the entire optimal set.

For the Euclidean distance cost on the real line, which is neither $C^1$ nor non-degenerate, the set of admissible destinations of a typical source $x$ is a one-dimensional segment, called a transport ray. In this case, Di Marino and Louet \cite{di2018entropic} identify an entropic selection through a renormalized \(\Gamma\)-limit under suitable integrability assumptions. Ley \cite{ley2025entropic} obtains complementary convergence results for less regular marginals under a componentwise mutual-singularity condition. In higher dimensions, selection along transport rays is studied by Aryan and Ghosal \cite{aryan2025entropic} and more recently by Nutz and Zhong \cite{nutz2026monge}, who establish a variational characterization of the selected plan as a constrained relative-entropy minimizer with geometry-dependent reference measures on transport rays.

\paragraph{Main contributions and strategy.}

The main purpose of this paper is to derive an entropic selection principle, through a \(\Gamma\)-convergence approach, when the ground cost \(c\) is non-degenerate but not necessarily twisted. In this situation, the set of admissible destinations of a typical source \(x\) is a locally finite set but need not be a singleton. We first consider the case of two marginals \((\mu_1,\mu_2)\) on the flat torus \(\T^d\), to avoid boundary issues, then explain how to handle the case of \(\R^d\). We assume that the contact set \(\Sigma\), defined as the zero set of the duality gap energy \(E \defeq c-\overline\varphi_1\oplus\overline\varphi_2\), \((\overline\varphi_1,\overline\varphi_2)\) being a pair of Kantorovich potentials, is a \(d\)-dimensional \(C^{1,1}\) submanifold. This is a strong assumption: under non-degeneracy alone, McCann, Pass and Warren \cite{mccann2012rectifiability} only guarantee that \(\Sigma\) is contained in a \(d\)-dimensional locally Lipschitz submanifold. We impose it to keep technicalities under control, since the construction of a recovery sequence is already delicate in the \(C^{1,1}\) setting, leaving the general case for future work.

Under our assumptions, the normal Hessian \(\hess^\perp E\) of \(E\), defined on \(\Sigma\) as the Hessian of \(E\) restricted to the normal space of \(\Sigma\), is positive definite. Roughly speaking, the Gibbs measure \(\gibbs = Z_\eps^ {-1} e^{-E/\varepsilon} \mu_1\otimes \mu_2\) behaves like a Gaussian with covariance matrix \(\eps (\hess^\perp E)^{-1}\) in the normal fibers to the manifold, and concentrates on \(\Sigma\) as \(\varepsilon \to 0\). We extend the results of Hwang \cite{hwang1980laplace}, who proves convergence towards a probability measure \(\mathbb P_0\) with an explicit density with respect to the volume measure on \(\Sigma\), from \(C^2\) to \(C^{1,1}\) regularity, and to compactly supported marginals on \(\R^d\), in which case \(\Sigma\) need not be compact. We then prove that the relative entropy with respect to the concentrating Gibbs probabilities, which coincides with the entropy-regularized functional up to an additive constant and a factor \(\varepsilon\) (see \Cref{prop: EOT-Gibbs}), \(\Gamma\)-converges under the original marginal constraints to the relative entropy with respect to this limiting probability. Since there always exists an optimal plan with finite limiting entropy, the limit problem has a unique minimizer, which is the selected optimal plan, and the entropic plans converge to it.

The principal difficulty lies in the construction of a recovery sequence for the \(\Gamma\)-\(\limsup\), which must approximate the relevant Gaussian profile in the normal fibers while satisfying the marginal constraints exactly. We first perturb an arbitrary admissible plan with finite limiting entropy by a varying Gaussian kernel in the normal fibers, so as to obtain an approximate recovery sequence that converges in energy but only satisfies the marginal constraints in the limit. The most delicate step is then to estimate how much these constraints are violated along the sequence. Roughly speaking, the centering of the normal kernels makes the first-order contributions cancel, so that the squared Wasserstein distance between the marginals of the perturbed plans and the prescribed ones is \(o(\varepsilon)\), although the normal variance is of order \(\varepsilon\). Thanks to this estimate, the shadow construction of \cite{eckstein2022quantitative} finally allows us to correct the marginals while preserving convergence in energy.

To state the resulting \(\Gamma\)-expansion, in the sense of Braides and Truskinovsky \cite{braides2008asymptotic}, we consider the functionals defined on \(\mathcal P(\Omega)\), where \(\Omega\defeq\T^d\times\T^d\), by
\begin{align*}
\mathcal F_\varepsilon(\gamma)&\defeq\int_{\Omega}c\dd \gamma+\varepsilon\entropy(\gamma\,|\,\mu_1\otimes\mu_2)+\chi_{\Pi(\mu_1,\mu_2)}(\gamma)
\qquad(\varepsilon>0),
\\
\mathcal F^{(0)}(\gamma) &\defeq\int_{\Omega}c\dd \gamma+\chi_{\Pi(\mu_1,\mu_2)}(\gamma),
\end{align*}
where $\Pi(\mu_1,\mu_2)$ denotes the set of couplings $\gamma \in \mathcal P(\Omega)$ having $\mu_1$ and $\mu_2$ as marginals, and $\chi$ denotes the indicator function in the sense of convex analysis. Writing \(\mu_i=\rho_i\dd x\), we also define the first-order functional
\[
\mathcal F^{(1)}(\gamma)\defeq\entropy(\gamma\,|\,\mathbb P_0)+\chi_{\Pi(\mu_1,\mu_2)}(\gamma)-\log Z_0-\frac{d}{2}\log(2\pi),
\]
where \(\mathbb P_0\) is the probability measure on \(\Sigma\) given by
\[
\dd\mathbb P_0(x_1,x_2)\defeq\frac{1}{Z_0}\frac{\rho_1(x_1)\rho_2(x_2)}{\sqrt{\det\hess^\perp E(x_1,x_2)}}\dd\mathcal H^d_{|\Sigma}(x_1,x_2),
\]
where \(\mathcal H^d_{|\Sigma}\) denotes the \(d\)-dimensional Hausdorff measure restricted to \(\Sigma\) and \(Z_0\) is the normalizing constant.

\begin{theoremA}[\(\Gamma\)-expansion and entropic selection]
\label{thm: A}
Assume that \ref{H1}--\ref{H4} of \Cref{subsec: assumptions} hold. In particular, \(M_1=M_2=\T^d\), \(c\) is non-degenerate near \(\Sigma\), and \(\Sigma\) is a \(d\)-dimensional \(C^{1,1}\) submanifold. Then the following holds.
\begin{enumerate}[(i)]
\item (\(\Gamma\)-expansion) The functionals \(\mathcal F_\varepsilon\) admit the first-order \(\Gamma\)-expansion
\begin{equation}
\label{eq: gamma-expansion}
\boxed{
 \mathcal F_\varepsilon - \frac{d}{2}\varepsilon\log(1/\varepsilon)
 \stackrel{\Gamma}{=}
 \mathcal F^{(0)}+\varepsilon\mathcal F^{(1)}+o(\varepsilon),
}
\end{equation}
which means that \((\mathcal F_\varepsilon-\min\mathcal F^{(0)})/\varepsilon-\frac{d}{2}\log(1/\varepsilon)\) \(\Gamma\)-converges to \(\mathcal F^{(1)}\) for the weak convergence of probability measures.
\end{enumerate}
As a consequence:
\begin{enumerate}[(i),resume]
\item (Asymptotics of the cost) The minimum of \(\mathcal F^{(1)}\) is finite, and
\begin{equation}
    \label{eq: energy-expansion}
    \boxed{
    \min\mathcal{F}_\varepsilon=\min \mathcal F^{(0)}+\dfrac{d}{2}\varepsilon\log(1/\varepsilon)+\varepsilon \min\mathcal F^{(1)}+o(\varepsilon).
    }
\end{equation}
\item (Entropic selection) The unique minimizer \(\gamma_\varepsilon\) of \(\mathcal F_\varepsilon\) converges weakly, as \(\varepsilon\to0\), to the unique minimizer \(\gamma_0\) of \(\entropy(\cdot\,|\,\mathbb P_0)\) among all optimal transport plans, which is also the unique minimizer of \(\mathcal F^{(1)}\).
\end{enumerate}
\end{theoremA}

We then treat compactly supported marginals on \(\R^d\) and illustrate the selection principle for the non-twisted cost of \cite[Example~3.1]{mccann2012rectifiability}. The extension to \(N\) marginals is stated separately under an explicit normal-rank condition. Finally, we apply it to a repulsive harmonic problem with three marginals on the circle.

\paragraph{Organization of the paper.} 

\Cref{sec: framework} introduces the transport problems, our assumptions, and the geometry of the contact set. \Cref{sec: gibbs} deals with the contact set and its tubular neighborhood, and states the geometric Laplace method. \Cref{sec: selection} constructs the recovery sequence, proves the constrained \(\Gamma\)-convergence, and derives the \(\Gamma\)-expansion and its consequences. \Cref{sec: extensions} gives the proof of the geometric Laplace method in the case of the flat torus and the Euclidean space, and provides the multi-marginal extension of our results, together with some examples. 

\section{Entropic Optimal Transport, assumptions and geometry of the contact set}
\label{sec: framework}

\subsection{Preliminaries on Entropic Optimal Transport}
We consider two compact metric spaces \(M_1\) and \(M_2\), endowed with probability measures \(\mu_1\) and \(\mu_2\), and a continuous cost function \(c:\Omega\to\R\), where \(\Omega\defeq M_1\times M_2\). Typical examples are closed Riemannian manifolds, such as the flat torus of \Cref{subsec: assumptions}, or compact subsets of \(\R^d\), as in \Cref{subsec: euclidean-extension}. The Monge--Kantorovich optimal transport problem reads
\begin{equation}
\label{eq: OT}
\inf_\gamma
\left\{
\int_{\Omega} c(\x)\dd\gamma(\x)
\:
\Big|
\:
\gamma \in \Pi(\mu_1,\mu_2)
\right\},
\tag{\(\mathrm{OT}\)}
\end{equation}
where
\[
\Pi(\mu_1,\mu_2)
\defeq
\left\{
\gamma \in \mathcal{P}(\Omega)
\:
\Big|
\:
(\proj_{M_i})_\#\gamma=\mu_i
\text{ for every }i=1,2
\right\}
\]
denotes the set of probability measures on \(\Omega\) having \(\mu_1\) and \(\mu_2\) as marginals.

It is well known that the optimal transport problem admits a dual formulation, for which strong duality holds and maximizers exist as soon as \(M_1\), \(M_2\) are compact and \(c\) is continuous (see \cite[Chapter~1]{santambrogioOptimalTransportApplied2015}). The dual problem is given by
\begin{equation}
\label{eq: dual}
\sup_{(\varphi_1,\varphi_2)}
\left\{
\int_{M_1} \varphi_1\dd\mu_1+\int_{M_2} \varphi_2\dd\mu_2
\;
\Big|
\;
\varphi_1\oplus\varphi_2 \leq c,\ \varphi_i \in C(M_i) \text{ for } i=1,2
\right\},
\tag{\(\mathrm{Dual}\)}
\end{equation}
where \((\varphi_1\oplus\varphi_2)(x_1,x_2)\defeq\varphi_1(x_1)+\varphi_2(x_2)\). Any pair \((\overline\varphi_1,\overline\varphi_2)\) attaining the supremum in \eqref{eq: dual} is called a pair of Kantorovich potentials. Given such a pair, we define\footnote{$E$ may depend on the choice of Kantorovich potentials, but the $\Gamma$-expansion result actually will not.} the duality gap energy \(E:\Omega\to\R_+\) and the contact set respectively by
\[
E
\defeq
c-\overline\varphi_1\oplus\overline\varphi_2 \qquad\text{and} \qquad \Sigma \defeq \{\x \in \Omega : E(\x) = 0\}.
\]
Since \(E\) is continuous and nonnegative, strong duality yields the following characterization of optimality:
\begin{equation}
\label{eq: optimality-support}
\overline \gamma \in \Pi(\mu_1,\mu_2) \text{ is a solution of \eqref{eq: OT} } \iff \spt(\overline \gamma) \subseteq \Sigma.
\end{equation}

For any \(\varepsilon>0\), the entropy-regularized two-marginal optimal transport problem is defined by
\begin{equation}
\label{eq: EOT}
\inf_{\gamma}
\left\{
\int_{\Omega} c(\x)\dd\gamma(\x)
+
\varepsilon
\entropy(
\gamma
\,|\,
\mu_1\otimes\mu_2)
\:
\Big|
\:
\gamma\in\Pi(\mu_1,\mu_2)
\right\},
\tag{\(\mathrm{OT}_\varepsilon\)}
\end{equation}
where, for \(\gamma\in\mathcal P(\Omega)\) and a probability measure \(\nu\) on \(\Omega\), the relative entropy is given by
\[
\entropy(
\gamma
\,|\,
\nu)
\defeq
\begin{cases}
\displaystyle
\int_{\Omega}
\log\left(
\frac{\dd\gamma}
{\dd\nu}(\x)
\right)
\dd\gamma(\x)
&
\text{if }\displaystyle
\gamma\ll\nu,
\\[1.2em]
+\infty
&
\text{otherwise}.
\end{cases}
\]
We use the same notation when \(\nu\) is a nonnegative Borel measure which is not a probability (e.g.\ a Hausdorff measure), whenever the integral is well defined.

We start by making the following observation.

\begin{proposition}
\label{prop: EOT-Gibbs}
For \(\varepsilon>0\), let \(\gibbs\) be the Gibbs probability measure defined by
\[
\gibbs
\defeq
\frac{1}{Z_\varepsilon}
\exp\left(-\frac{E}{\varepsilon}\right)
\mu_1\otimes\mu_2,
\qquad
Z_\varepsilon
\defeq
\int_{\Omega}
\exp\left(-\frac{E(\x)}{\varepsilon}\right)
\dd(\mu_1\otimes\mu_2)(\x).
\]
Then, for every \(\gamma\in\Pi(\mu_1,\mu_2)\),
\begin{equation}
\label{eq: EOT-Gibbs identity}
\int_{\Omega} c\dd\gamma
+
\varepsilon\entropy(\gamma\,|\,\mu_1\otimes\mu_2)
=
\OT(\mu_1,\mu_2)
+
\varepsilon\entropy(\gamma\,|\,\gibbs)
-
\varepsilon\log Z_\varepsilon,
\end{equation}
where \(\OT(\mu_1,\mu_2)\) denotes the value of \eqref{eq: OT}. In particular, problem \eqref{eq: EOT} has the same minimizers as
\begin{equation}
\label{eq: EOT_2}
\inf_{\gamma}
\left\{
\entropy\left(
\gamma
\,\big|\,
\gibbs
\right)
\:
\Big|
\:
\gamma\in\Pi(\mu_1,\mu_2)
\right\},
\tag{\(\mathrm{EOT}_\varepsilon\)}
\end{equation}
\end{proposition}

\begin{proof}
Let \(\gamma\in\Pi(\mu_1,\mu_2)\). Since \(\gamma\) has marginals \(\mu_1\) and \(\mu_2\), strong duality gives \(\int_\Omega\overline\varphi_1\oplus\overline\varphi_2\dd\gamma=\OT(\mu_1,\mu_2)\), hence \(\int_\Omega c\dd\gamma=\OT(\mu_1,\mu_2)+\int_\Omega E\dd\gamma\). Moreover, since \(E\) is bounded, the measures \(\gibbs\) and \(\mu_1\otimes\mu_2\) are equivalent, with bounded log-density \(\log\frac{\dd\gibbs}{\dd(\mu_1\otimes\mu_2)}=-\frac{E}{\varepsilon}-\log Z_\varepsilon\), so that
\[
\varepsilon\entropy(\gamma\,|\,\mu_1\otimes\mu_2)
=
\varepsilon\entropy(\gamma\,|\,\gibbs)
-\int_\Omega E\dd\gamma
-\varepsilon\log Z_\varepsilon,
\]
both sides being possibly infinite. Summing the two identities gives \eqref{eq: EOT-Gibbs identity}.
\end{proof}

\subsection{Main assumptions and geometry of the contact set}
\label{subsec: assumptions}
We first consider two marginals on flat tori.  A Euclidean extension
is given in \Cref{subsec: euclidean-extension}.

\begin{enumerate}[(H1)]
\item\label{H1} For \(i=1,2\), \(M_i=\mathbb{T}^d\), with the flat product metric. The cost
\(c\) is continuous, and we fix optimal continuous Kantorovich
potentials \(\overline\varphi_1,\overline\varphi_2\).
The cost is \(C^2\)
and satisfies the non-degeneracy (or infinitesimal twist) condition on an open neighborhood
\(U\subset\Omega\) of the contact set \(\Sigma\):
\[
 D^2_{x_1x_2}c(z):T_{x_2}M_2\longrightarrow T_{x_1}M_1
 \quad\hbox{is invertible for every }z=(x_1,x_2)\in U.\]

\item\label{H2} The duality gap
\(E(x_1,x_2)=c(x_1,x_2)-\overline\varphi_1(x_1)-\overline\varphi_2(x_2)\geq0\)
belongs to \(C^{2,1}(\Omega)\).

\item\label{H3} The contact set \(\Sigma=\{E=0\}\) is a nonempty
\(C^{1,1}\) submanifold without boundary, of
dimension \(d\). We denote by \(k\defeq d\geq1\) its codimension in \(\Omega\).

\item\label{H4} The marginals have densities
\(\mu_i=\rho_i\,\dd\vol_{M_i}\), with \(\rho_i\) Lipschitz and
\(0<\alpha\leq\rho_i\leq\beta<\infty\) on \(M_i\), for \(i=1,2\).
\end{enumerate}

\begin{remark}[About assumption \ref{H3}]
\label{rem: about-H3}
Under \ref{H1}--\ref{H2}, the implicit function theorem shows that \(S\defeq\{z\in U:\nabla_{x_1}E(z)=0\}\) is a \(d\)-dimensional \(C^{1,1}\) submanifold, since \(D_{x_2}(\nabla_{x_1}E)=D^2_{x_1x_2}c\) is invertible. As \(\Sigma\subseteq S\), \ref{H3} amounts to \(\Sigma\) being relatively open in \(S\) (compare with \cite{mccann2012rectifiability}, where \(\Sigma\) is shown to lie in a Lipschitz submanifold). Moreover, \ref{H3} implies that \(\Sigma=\{\nabla E=0\}\) in a neighborhood of \(\Sigma\), by the positive definiteness of the normal Hessian (see \Cref{prop: twist-geometry}). Requiring this equality globally would however not be natural: \(\{\nabla E=0\}\) contains the maximum points of \(E\) and need not be a \(d\)-dimensional submanifold. For the bilinear cost \(c(x_1,x_2)=-\langle x_1,x_2\rangle\) in the Euclidean setting, \(E(x_1,x_2)=u(x_1)+u^*(x_2)-\langle x_1,x_2\rangle\) with \(u\) convex is convex in each variable, so that \(S=\Sigma\) and \ref{H2} implies \ref{H3}.
\end{remark}

We set \(\rho(x_1,x_2)=\rho_1(x_1)\rho_2(x_2)\) and
\(\mu^\otimes=\mu_1\otimes\mu_2=\rho\,\dd\vol_\Omega\). For \(z\in\Sigma\), the normal Hessian is
\[
 \Hn{z}\defeq\hess^\perp E(z)
 =\hess E(z)|_{N_z\Sigma},\qquad N_z\Sigma=(T_z\Sigma)^\perp.
\]

\begin{proposition}
\label{prop: twist-geometry}
Under \ref{H1}--\ref{H3},
\[
 \ker\hess E(z)=T_z\Sigma\qquad(z\in\Sigma),
\]
and both coordinate projections \(\proj^\Sigma_{M_i}:\Sigma\to M_i\)
are local \(C^{1,1}\) diffeomorphisms. Moreover, there exist constants
\(0<\lambda\leq\Lambda<\infty\) such that
\begin{equation}
\label{eq: normal-ellipticity}
 \lambda|v|^2\leq \Hn{z}(v,v)\leq\Lambda|v|^2
 \qquad(z\in\Sigma,\ v\in N_z\Sigma).
\end{equation}
\end{proposition}

\begin{proof}
Since \(E\geq0\) and \(E=0\) on \(\Sigma\), we have
\(\nabla E=0\) on \(\Sigma\), and \(\hess E(z)\) is positive
semidefinite there. Differentiating \(\nabla E=0\) along
\(\Sigma\) gives \(T_z\Sigma\subseteq\ker\hess E(z)\).
The separated form of the potentials implies
\(D^2_{x_1x_2}E=D^2_{x_1x_2}c\) near \(\Sigma\): this identity
also follows directly by taking mixed differences, without initially
assuming differentiability of the potentials.
Thus the mixed block of \(\hess E(z)\) has rank \(d\), so
\(\operatorname{rank}\hess E(z)\geq d\). The preceding kernel
inclusion and \(\dim T_z\Sigma=d\) give the reverse inequality.
Consequently, \(\ker\hess E(z)=T_z\Sigma\).

The differential of \(\proj^\Sigma_{M_1}\) at \(z\) is the restriction of \((u,v)\mapsto u\) to \(T_z\Sigma\), whose kernel consists of the vectors \((0,v)\in T_z\Sigma=\ker\hess E(z)\). For such a vector, the first component of \(\hess E(z)[(0,v)]=0\) reads \(D^2_{x_1x_2}c(z)v=0\), whence \(v=0\). Thus \(D\proj^\Sigma_{M_1}(z)\) is injective, hence an isomorphism since \(\dim T_z\Sigma=d=\dim M_1\), and the inverse function theorem shows that \(\proj^\Sigma_{M_1}\) is a local \(C^{1,1}\) diffeomorphism. The same argument applies to \(\proj^\Sigma_{M_2}\).
Finally, positive semidefiniteness and the kernel identity \(\ker\hess E(z)=T_z\Sigma\) make the restriction \(\Hn{z}\) of \(\hess E(z)\) to \(N_z\Sigma=(T_z\Sigma)^\perp\) positive definite. Continuity of the Hessian and the normal
spaces, together with compactness of \(\Sigma\), gives
\eqref{eq: normal-ellipticity}.
\end{proof}

\begin{remark}[Vertical and horizontal branches for the two-marginal problem]
    By \Cref{prop: twist-geometry}, both projections \(\proj^\Sigma_{M_1}\) and \(\proj^\Sigma_{M_2}\) are local diffeomorphisms. In particular, \(\Sigma\) has neither vertical nor horizontal branches, that is, points \(z\in\Sigma\) such that \(\rank(D\proj^\Sigma_{M_1}(z))<d\), respectively \(\rank(D\proj^\Sigma_{M_2}(z))<d\).
\end{remark}

\section{Convergence of the Gibbs measure}
\label{sec: gibbs}

\subsection{Tubular coordinates around the contact set}
\label{subsec:structure}

By \ref{H3}, \(\Sigma\) is a compact \(C^{1,1}\)-submanifold of \(\Omega\), and thus has positive reach in the sense of Federer \cite{federer1959curvature}.
Consequently, there exists \(r_0>0\) such that, setting, for every \(0<r\leq r_0\),
\[
N_r\Sigma
\defeq
\left\{
(z,v):z\in\Sigma,\;
v\in N_z\Sigma,\;
\|v\|<r
\right\},
\qquad
\Sigma_r\defeq\{\x\in\Omega:\dist(\x,\Sigma)<r\},
\]
the nearest-point projection \(\proj_\Sigma:\Sigma_{r_0}\to\Sigma\) is well defined and Lipschitz, and the map \((z,v)\mapsto z+v\) is a bi-Lipschitz bijection from \(N_r\Sigma\) onto \(\Sigma_r\), with inverse \(\x\mapsto(\proj_\Sigma(\x),\x-\proj_\Sigma(\x))\). Here the addition is
understood componentwise and modulo \(\mathbb{Z}^d\) on each torus
factor.

\begin{lemma}[Lipschitz normal frame]
\label{lemma: normal-frame}
Under \ref{H1}--\ref{H3}, there exists a Lipschitz map \(n=(n_1,\ldots,n_k):\Sigma\to\R^{2d\times k}\) such that, for every \(z\in\Sigma\), the columns \(n_1(z),\ldots,n_k(z)\) form an orthonormal basis of \(N_z\Sigma\).
\end{lemma}
\begin{proof}
By \Cref{prop: twist-geometry}, the projection \(\proj^\Sigma_{M_1}\) is a local \(C^{1,1}\) diffeomorphism, so that, locally, \(\Sigma\) is the graph of a \(C^{1,1}\) map \(x_1\mapsto x_2\). Identifying the tangent spaces of the flat tori with \(\R^d\), this means that, for every \(z\in\Sigma\), \(T_z\Sigma=\{(u,B(z)u):u\in\R^d\}\) for some matrix \(B(z)\in\R^{d\times d}\), where \(z\mapsto B(z)\) is locally Lipschitz, hence Lipschitz on \(\Sigma\) by compactness. Hence \(N_z\Sigma=\{(-B(z)^\top w,w):w\in\R^d\}\), that is, the columns of
\[
N(z)\defeq\begin{pmatrix}
-B(z)^\top\\
I_d
\end{pmatrix}\in\R^{2d\times k}
\]
form a basis of \(N_z\Sigma\). Since their Gram matrix \(N(z)^\top N(z)=I_d+B(z)B(z)^\top\) is symmetric positive definite, we obtain an orthonormal basis of \(N_z\Sigma\) by setting
\[
n(z)
\defeq
N(z)\bigl(I_d+B(z)B(z)^\top\bigr)^{-1/2},
\qquad\text{so that}\qquad
n(z)^\top n(z)=I_k .
\]
The map \(n\) is Lipschitz, as \(B\) is.
\end{proof}
We fix such a frame \(n\). For every \(z\in\Sigma\) and \(s\in\R^k\), we thus have
\[
n(z)s
=
\sum_{p=1}^{k}s_pn_p(z)
\in N_z\Sigma.
\]
Through the isometry \(n(z):\R^k\to N_z\Sigma\), we write \(\Hn{z}(s,s)\) for \(\Hn{z}(n(z)s,n(z)s)\), so that \eqref{eq: normal-ellipticity} holds for every \(s\in\R^k\), and \(\det \Hn{z}\) for the determinant of the \(k\times k\) matrix with entries \(\Hn{z}(n_p(z),n_q(z))\), which does not depend on the choice of the orthonormal frame.
Denoting by \(\ball_r\) the open ball of radius \(r\) centered at \(0\) in \(\R^k\), we define the tubular parametrization
\[
\Phi(z,s)
\defeq
z+n(z)s,
\qquad
z\in\Sigma,\quad s\in\ball_{r_0},
\]
which, by the above, is a bi-Lipschitz bijection from \(\Sigma\times\ball_r\) onto \(\Sigma_r\) for every \(0<r\leq r_0\). Moreover, since \(z=\proj_\Sigma(\Phi(z,s))\), we have
\begin{equation}
\label{eq: normal-distance}
\dist\bigl(\Phi(z,s),\Sigma\bigr)=|s|
\qquad\text{for every }(z,s)\in\Sigma\times\ball_{r_0},
\end{equation}
which will be used repeatedly in the sequel. Since \(n\) is Lipschitz, it is differentiable along \(\Sigma\) at \(\mathcal H^d\)-a.e. point, by Rademacher's theorem. For such \(z\in\Sigma\) and
\(p\in\{1,\ldots,k\}\), we denote by
\[
G_p(z):T_z\Sigma\longrightarrow T_z\Sigma
\]
the linear operator determined by the tangential component of
\(Dn_p(z)\), namely
\[
\proj_{T_z\Sigma}\bigl(Dn_p(z)[X]\bigr)
=
G_p(z)X,
\qquad
X\in T_z\Sigma,
\]
where \(\proj_{T_z\Sigma}\) denotes the orthogonal projection of \(\R^{2d}\) onto \(T_z\Sigma\).
For \(s\in\mathbb R^k\), we set
\[
G(z,s)
\defeq
\sum_{p=1}^{k}s_pG_p(z).
\]

\begin{lemma}
\label{lemma: determinant of Phi}
Under \ref{H1}--\ref{H3}, for \(\mathcal H^d\)-a.e. \(z\in\Sigma\) and every \(s\in\ball_{r_0}\), the map \(\Phi\) is differentiable at \((z,s)\), and its Jacobian \(\mathcal J_\Phi(z,s)\), that is, the absolute value of the determinant of \(D\Phi(z,s):T_z\Sigma\times\R^k\to\R^{2d}\) in orthonormal bases, is
\[
\mathcal J_\Phi(z,s)
=
\Bigl|\det\bigl(\operatorname{Id}_{T_z\Sigma}+G(z,s)\bigr)\Bigr|.
\]
Moreover, there exists \(C>0\) such that
\[
\|G(z,s)\|\le C|s|
\]
for \(\mathcal H^d\)-a.e. \(z\in\Sigma\) and every \(s\in\R^k\).
\end{lemma}

\begin{proof}
Fix \(z\in\Sigma\) at which \(n\) is differentiable along \(\Sigma\). Since \(\Phi(z,s)=z+n(z)s\) is affine in \(s\), the map \(\Phi\) is differentiable at \((z,s)\) for every \(s\in\ball_{r_0}\), and, for \(X\in T_z\Sigma\) and \(t\in\R^k\),
\[
D\Phi(z,s)[X,t]
=
X+\sum_{p=1}^{k}s_pDn_p(z)[X]+n(z)t
=
\bigl(X+G(z,s)X\bigr)+Q(z,s)X+n(z)t,
\]
where \(Q(z,s)X\defeq\proj_{N_z\Sigma}\bigl(\sum_{p=1}^{k}s_pDn_p(z)[X]\bigr)\in N_z\Sigma\). Hence, decomposing the target space as \(\R^{2d}=T_z\Sigma\oplus N_z\Sigma\), the operator \(D\Phi(z,s):T_z\Sigma\times\R^k\to T_z\Sigma\oplus N_z\Sigma\) has the block lower triangular form
\[
D\Phi(z,s)
=
\begin{pmatrix}
\operatorname{Id}_{T_z\Sigma}+G(z,s) & 0_{\R^k\to T_z\Sigma}\\
Q(z,s) & n(z)
\end{pmatrix}.
\]
Let \((e_1,\ldots,e_d)\) be an orthonormal basis of \(T_z\Sigma\). By \Cref{lemma: normal-frame}, the vectors \(e_1,\ldots,e_d\), \(n_1(z),\ldots,n_k(z)\) form an orthonormal basis of \(\R^{2d}\), in which the block \(n(z)\) has matrix \(I_k\). Computing \(\mathcal J_\Phi(z,s)\) in this basis and in the orthonormal basis of \(T_z\Sigma\times\R^k\) formed by \((e_1,\ldots,e_d)\) and the canonical basis of \(\R^k\), we obtain \(\mathcal J_\Phi(z,s)=\bigl|\det\bigl(\operatorname{Id}_{T_z\Sigma}+G(z,s)\bigr)\bigr|\). Finally, denoting by \(\operatorname{Lip}(n)\) the Lipschitz constant of \(n\), we have \(\|G_p(z)\|\leq\|Dn_p(z)\|\leq\operatorname{Lip}(n)\) for every \(p\), whence \(\|G(z,s)\|\leq\sum_{p=1}^{k}|s_p|\,\|G_p(z)\|\leq\sqrt k\operatorname{Lip}(n)|s|\).
\end{proof}

Since \(\Phi\) is a bi-Lipschitz bijection from \(\Sigma\times\ball_r\) onto \(\Sigma_r\), the area formula gives, for every \(0<r\leq r_0\) and every nonnegative Borel function \(f\) on \(\Sigma_r\),
\begin{equation}
\label{eq: tube-area-formula}
\int_{\Sigma_r}f\dd\vol_\Omega
=
\int_\Sigma\int_{\ball_r}f\bigl(\Phi(z,s)\bigr)\,\mathcal J_\Phi(z,s)\dd s\dd\mathcal H^d(z).
\end{equation}
After decreasing \(r_0\), we may assume that \(Cr_0\leq1/2\), so that \(\|G(z,s)\|\leq1/2\) for \(\mathcal H^d\)-a.e. \(z\in\Sigma\) and every \(s\in\ball_{r_0}\). In particular, \(\det\bigl(\operatorname{Id}_{T_z\Sigma}+G(z,s)\bigr)>0\), and expanding the determinant gives
\begin{equation}
\label{eq: tube-jacobian-expansion}
\mathcal J_\Phi(z,s)
=
1+L(z,s)+O(|s|^2),
\qquad
L(z,s)\defeq\operatorname{tr}G(z,s),
\end{equation}
where the remainder is uniform in \(z\). The function \(L(z,\cdot)\) is linear, hence odd, and \(|L(z,s)|\leq dC|s|\). In particular,
\begin{equation}
\label{eq: tube-jacobian-rescaled}
\mathcal J_\Phi(z,\sqrt\varepsilon s)
=
1+\sqrt\varepsilon L(z,s)+O(\varepsilon|s|^2)
\qquad\text{whenever }\sqrt\varepsilon|s|<r_0.
\end{equation}

\subsection{Geometric Laplace's method}
\label{subsec: laplace}
\begin{proposition}
\label{lemma: quadratic-detachment}
Under \ref{H1}--\ref{H3}, there exist constants \(r>0\) and
\(0<\underline{C}\leq\overline{C}<+\infty\) such that, for every
\(\x\in\Omega\) satisfying \(\dist(\x,\Sigma)\leq r\), one has
\[
\underline{C}\dist(\x,\Sigma)^2
\leq
E(\x)
\leq
\overline{C}\dist(\x,\Sigma)^2.
\]
\end{proposition}

\begin{proof}
Let \(r_0>0\) be as in \Cref{subsec:structure}, and let
\(\x\in\Sigma_{r_0}\). Then \(\x=\Phi(z,s)\) with \(z=\proj_\Sigma(\x)\) and a unique \(s\in\ball_{r_0}\), and \(|s|=\dist(\x,\Sigma)\) by \eqref{eq: normal-distance}. Since \(z\) is a minimum point of \(E\geq0\), one has \(E(z)=0\) and \(\nabla E(z)=0\), so that a Taylor expansion of \(E\) gives
\[
E(\x)
=
\frac{1}{2}\Hn{z}(s,s)
+
\mathcal R(z,s),
\qquad
\mathcal R(z,s)=O\bigl(|s|^3\bigr)
\quad\text{as }s\to0,
\]
uniformly with respect to \(z\in\Sigma\), since \(E\) belongs to \(C^{2,1}(\Omega)\).
By \eqref{eq: normal-ellipticity}, \(\lambda |s|^2 \leq \Hn{z}(s,s) \leq \Lambda |s|^2\).
Taking \(r<r_0\) small enough, we have \(|\mathcal R(z,s)|\leq\frac{\lambda}{4}|s|^2\) whenever \(z\in\Sigma\) and \(|s|\leq r\).
Consequently,
\[
\frac{\lambda}{4}\dist(\x,\Sigma)^2
\leq
E(\x)
\leq
\frac{3\Lambda}{4}\dist(\x,\Sigma)^2,
\]
and the conclusion follows by setting \(
\underline{C}\defeq\frac{\lambda}{4}\) and \(
\overline{C}\defeq\frac{3\Lambda}{4}\).
\end{proof}

\begin{theorem}
\label{thm: convergence-Gibbs-measure}
Under \ref{H1}--\ref{H4}, as \(\varepsilon\to0\), the Gibbs measure \(\gibbs\) converges weakly to a probability measure \(\mathbb{P}_0\), given by
\[
\dd\mathbb{P}_0(z)
=
\frac{1}{Z_0}
\left(\det \Hn{z}\right)^{-1/2}
\rho(z)
\dd\mathcal H^d_{|\Sigma}(z),
\qquad
Z_0
\defeq
\int_\Sigma
\left(\det \Hn{z}\right)^{-1/2}
\rho(z)
\dd\mathcal H^d_{|\Sigma}(z).
\]
Moreover, \(Z_\varepsilon=(2\pi\varepsilon)^{k/2}(Z_0+o(1))\).
\end{theorem}

\begin{remark}
  We postpone the proof to \Cref{subsec: Convergence of the Gibbs measure and the geometric Laplace's method}. It follows the general strategy of \cite{hwang1980laplace}, with the various steps of the geometric Laplace method made explicit, in a slightly different setting: the proof is formulated so as to cover the case of compactly supported marginals in \(\R^d\), and requires only \(C^{2,1}\)-regularity of the duality gap \(E\) and \(C^{1,1}\)-regularity of the contact set \(\Sigma\).
\end{remark}

\begin{remark}
A connection between the geometric Laplace method and optimal transport has already been established by Léger and Vialard in \cite{leger2023geometric}, under stronger regularity assumptions and assuming that the contact set is the graph of a \(C^3\)-diffeomorphism. These assumptions allow them to derive a first-order geometric Laplace expansion. In contrast, our purpose is only to identify the leading-order limit, under weaker regularity assumptions and without requiring the contact set to be globally a graph.
\end{remark}

\section{The \texorpdfstring{\(\Gamma\)-}{Γ-}expansion of Entropic Optimal Transport}
\label{sec: selection}
Throughout this section, we assume \ref{H1}--\ref{H4}, and we consider the functionals defined on \(\mathcal P(\Omega)\) by
\begin{equation}
\label{eq: G-functionals}
\mathcal G_\varepsilon(\gamma)\defeq\entropy(\gamma\,|\,\mathbb P_\varepsilon)+\chi_{\Pi(\mu_1,\mu_2)}(\gamma),
\qquad
\mathcal G_0(\gamma)\defeq\entropy(\gamma\,|\,\mathbb P_0)+\chi_{\Pi(\mu_1,\mu_2)}(\gamma).
\end{equation}
By \Cref{prop: EOT-Gibbs} and \Cref{thm: convergence-Gibbs-measure}, the functionals \(\mathcal F_\varepsilon\), \(\mathcal F^{(0)}\) and \(\mathcal F^{(1)}\) of \hyperref[thm: A]{Theorem~A} satisfy
\begin{equation}
\label{eq: G-vs-F}
\begin{aligned}
\frac{\mathcal F_\varepsilon-\min\mathcal F^{(0)}}{\varepsilon}-\frac k2\log\frac1\varepsilon
&=\mathcal G_\varepsilon-\log Z_0-\frac k2\log(2\pi)+o(1),
\\
\mathcal F^{(1)}&=\mathcal G_0-\log Z_0-\frac k2\log(2\pi),
\end{aligned}
\end{equation}
where the \(o(1)\) does not depend on \(\gamma\). Hence, the first-order \(\Gamma\)-expansion of \hyperref[thm: A]{Theorem~A} amounts to the \(\Gamma\)-convergence of \(\mathcal G_\varepsilon\) to \(\mathcal G_0\) as \(\varepsilon\to0\), for the weak convergence of probability measures, that is, to the following two inequalities:
\begin{itemize}
\item (\(\Gamma\)-\(\liminf\) inequality) for every \(\overline\gamma\in\mathcal P(\Omega)\) and every family \((\gamma_\varepsilon)_{\varepsilon>0}\) in \(\mathcal P(\Omega)\) such that \(\gamma_\varepsilon\rightharpoonup\overline\gamma\) as \(\varepsilon\to0\),
\begin{equation}
\label{eq: gamma-liminf}
\liminf_{\varepsilon\to0}\mathcal G_\varepsilon(\gamma_\varepsilon)\geq\mathcal G_0(\overline\gamma);
\end{equation}
\item (\(\Gamma\)-\(\limsup\) inequality) for every \(\overline\gamma\in\mathcal P(\Omega)\), there exists a family \((\gamma_\varepsilon)_{\varepsilon>0}\) in \(\mathcal P(\Omega)\), called a recovery sequence for \(\overline\gamma\), such that \(\gamma_\varepsilon\rightharpoonup\overline\gamma\) as \(\varepsilon\to0\) and
\begin{equation}
\label{eq: gamma-limsup}
\limsup_{\varepsilon\to0}\mathcal G_\varepsilon(\gamma_\varepsilon)\leq\mathcal G_0(\overline\gamma).
\end{equation}
\end{itemize}

The \(\Gamma\)-\(\liminf\) inequality \eqref{eq: gamma-liminf} follows directly from the convergence of the Gibbs measures. The core of the section is the \(\Gamma\)-\(\limsup\) inequality \eqref{eq: gamma-limsup}, that is, the construction of a recovery sequence: we first perturb a given plan in the normal directions to \(\Sigma\) (\Cref{subsec: normal-recovery}), and then correct its marginals (\Cref{subsec: correction}). Both inequalities, and then \hyperref[thm: A]{Theorem~A}, are proved in \Cref{subsec: selection-value-asymptotics}. Finally, \Cref{subsection: proof of the main result} contains the proof of the key estimate on the marginals of the perturbed plans.

\subsection{The recovery sequence: a normal perturbation approach}
\label{subsec: normal-recovery}
To prove the \(\Gamma\)-\(\limsup\) inequality \eqref{eq: gamma-limsup}, we fix \(\overline\gamma\) such that \(\mathcal G_0(\overline\gamma)<\infty\) (otherwise, \(\gamma_\varepsilon\defeq\overline\gamma\) is a recovery sequence), that is, \(\overline\gamma\in\Pi(\mu_1,\mu_2)\) with \(\entropy(\overline\gamma\,|\,\mathbb P_0)<\infty\), and we look for plans \(\gamma_\varepsilon\in\Pi(\mu_1,\mu_2)\) such that \(\gamma_\varepsilon\rightharpoonup\overline\gamma\) and \(\limsup_{\varepsilon\to0}\mathcal G_\varepsilon(\gamma_\varepsilon)\leq\mathcal G_0(\overline\gamma)\). We proceed in two steps: we first spread \(\overline\gamma\) in the normal directions according to the Gaussian profile of the Gibbs measure, which gives the right entropy but not the right marginals, and we then correct the marginals in \Cref{subsec: correction}.

Since \(\overline\gamma\ll\mathbb P_0\ll\mathcal H^d_{|\Sigma}\), we write \(\overline\gamma=\overline\rho\,\mathcal H^d_{|\Sigma}\). We fix a truncation radius \(R\geq1\) for the rescaled normal variable \(s\), the distance to \(\Sigma\) being \(\sqrt\varepsilon|s|\), and we take \(\varepsilon\) small enough so that \(\sqrt\varepsilon R<r_0\), with \(r_0\) as in \Cref{subsec:structure}; \(R\) will be sent to infinity at the end of the proof. Recall that \(\Hn{z}=\hess^\perp E(z)\) is the normal Hessian of \(E\) at \(z\in\Sigma\) (\Cref{subsec: assumptions}), seen as a quadratic form on \(\R^k\) through the normal frame \(n(z)\) (\Cref{subsec:structure}).

For \(z\in\Sigma\), let
\[
K_R(z,s)\defeq\frac{\mathbf1_{\ball_R}(s)\,e^{-\Hn{z}(s,s)/2}}{Z_R(z)},
\qquad
Z_R(z)\defeq\int_{\ball_R}e^{-\Hn{z}(s,s)/2}\dd s,
\]
be the density of the centered Gaussian measure on \(\R^k\) with covariance \(\Hn{z}^{-1}\), truncated to \(\ball_R\) and renormalized. By \eqref{eq: normal-ellipticity}, \(Z_R(z)\geq c\defeq\int_{\ball_1}e^{-\Lambda|s|^2/2}\dd s>0\) for \(R\geq1\), so that the moments of these kernels are bounded uniformly in \(z\in\Sigma\) and \(R\geq1\): for every \(p\geq0\),
\begin{equation}
\label{eq: gaussian-moments}
\int_{\ball_R}|s|^p\,K_R(z,s)\dd s\leq M_p\defeq\frac1c\int_{\R^k}|s|^pe^{-\lambda|s|^2/2}\dd s<\infty.
\end{equation} Setting \(\Phi_\varepsilon(z,s)\defeq\Phi(z,\sqrt\varepsilon s)\), we have \(\overline\gamma=(\Phi_\varepsilon)_\#(\overline\gamma\otimes\delta_0)\). Replacing the Dirac mass in the normal variable by this truncated Gaussian, we define the \emph{normal perturbation} of \(\overline\gamma\) by
\begin{equation}
\label{eq: recovery sequence with wrong marginals 1}
\widetilde\gamma_{\varepsilon,R}\defeq(\Phi_\varepsilon)_\#\bigl(K_R(z,s)\dd s\,\dd\overline\gamma(z)\bigr),
\end{equation}
that is, \(\int_\Omega f\dd\widetilde\gamma_{\varepsilon,R}=\int_\Sigma\int_{\ball_R}f(\Phi(z,\sqrt\varepsilon s))\,K_R(z,s)\dd s\dd\overline\gamma(z)\) for every nonnegative Borel function \(f\) on \(\Omega\). Since \(K_R(z,\cdot)\) is a probability density, \(\widetilde\gamma_{\varepsilon,R}\) is a probability measure, concentrated on \(\Sigma_{\sqrt\varepsilon R}\). In the original normal variable \(\eta=\sqrt\varepsilon s\), its profile in each normal fiber is the Gaussian with covariance \(\varepsilon \Hn{z}^{-1}\), truncated at distance \(\sqrt\varepsilon R\) from \(\Sigma\), which mimics the Gibbs measure \(\gibbs\) near \(\Sigma\), since \(E(\Phi(z,\eta))\simeq\frac12\Hn{z}(\eta,\eta)\). By \eqref{eq: tube-area-formula}, \(\widetilde\gamma_{\varepsilon,R}\) is absolutely continuous with respect to \(\vol_\Omega\), with a density \(\widetilde p_{\varepsilon,R}\) vanishing outside \(\Sigma_{\sqrt\varepsilon R}\) and given in tubular coordinates by
\begin{equation}
\label{eq: normal-recovery-density}
\widetilde p_{\varepsilon,R}\bigl(\Phi(z,\sqrt\varepsilon s)\bigr)
=\frac{\overline\rho(z)\,K_R(z,s)}{\varepsilon^{k/2}\,\mathcal J_\Phi(z,\sqrt\varepsilon s)}
\qquad\text{for \(\mathcal H^d\)-a.e.\ \(z\in\Sigma\) and every \(s\in\ball_R\).}
\end{equation}

Since \(\rho_i\geq\alpha>0\), the marginals \(\widetilde\mu_{i,\varepsilon,R}\), \(i=1,2\), of \(\widetilde\gamma_{\varepsilon,R}\) are absolutely continuous with respect to \(\mu_i\). They differ from \(\mu_i\) in general, and are corrected in \Cref{subsec: correction}. Finally, the image of \(K_R(z,s)\dd s\,\dd\overline\gamma(z)\) under the map \((z,s)\mapsto(z,\Phi_\varepsilon(z,s))\) is a coupling of \(\overline\gamma\) and \(\widetilde\gamma_{\varepsilon,R}\), and the distance between \(z\) and \(\Phi_\varepsilon(z,s)=z+\sqrt\varepsilon\,n(z)s\) is at most \(\sqrt\varepsilon|s|\leq\sqrt\varepsilon R\), so that
\begin{equation}
\label{eq: normal-perturbation-W2}
W_2(\widetilde\gamma_{\varepsilon,R},\overline\gamma)\leq\sqrt\varepsilon R\longrightarrow0
\qquad(\varepsilon\to0).
\end{equation}
In particular, since the marginals of \(\overline\gamma\) are \(\mu_i\) and the projections \(\Omega\to M_i\) are \(1\)-Lipschitz, we also have \(W_2(\widetilde\mu_{i,\varepsilon,R},\mu_i)\leq\sqrt\varepsilon R\) for \(i=1,2\).

\begin{lemma}[Entropy of the normal perturbation]
\label{lemma: convergence entropy}
For every fixed \(R\geq1\), as \(\varepsilon\to0\),
\begin{equation}
\label{eq: fixed-R-entropy} \entropy\left(\widetilde\gamma_{\varepsilon,R}\,|\,\gibbs\right)
 \longrightarrow
 \entropy(\overline\gamma\,|\,\mathbb P_0)+b_R(\overline\gamma),
 \qquad
b_R(\overline\gamma)=\int_\Sigma\log\frac{Z_\infty(z)}{Z_R(z)}\,\dd\overline\gamma(z),
\end{equation}
where \(Z_\infty(z)\defeq\int_{\R^k}e^{-\Hn{z}(s,s)/2}\dd s=(2\pi)^{k/2}(\det \Hn{z})^{-1/2}\).
Moreover, \(b_R(\overline\gamma)\geq0\) and
\(\sup_{\overline\gamma\in\mathcal P(\Sigma)}b_R(\overline\gamma)\to0\)
as \(R\to\infty\).
\end{lemma}

\begin{proof}
By \eqref{eq: normal-recovery-density}, and since \(\gibbs\) has density \(Z_\varepsilon^{-1}e^{-E/\varepsilon}\rho\) with respect to \(\vol_\Omega\), where \(\rho\) is the density of \(\mu_1\otimes\mu_2\), we have, for \(\mathcal H^d\)-a.e.\ \(z\in\Sigma\) and every \(s\in\ball_R\),
\begin{align*}
 \log\left(
 \frac{\dd\widetilde\gamma_{\varepsilon,R}}
      {\dd\gibbs}(\Phi(z,\sqrt\varepsilon s))\right)
 &=\log\overline\rho(z)
   +\log\left(\frac{Z_\varepsilon}
   {\varepsilon^{k/2}Z_R(z)}\right)
   -\log\rho(\Phi(z,\sqrt\varepsilon s))
\\
 &\qquad+\frac{E(\Phi(z,\sqrt\varepsilon s))-\frac12\Hn{z}(\sqrt\varepsilon s,\sqrt\varepsilon s)}
 {\varepsilon}
 -\log\mathcal J_\Phi(z,\sqrt\varepsilon s).
\end{align*}
As in the proof of \Cref{lemma: quadratic-detachment}, a Taylor expansion of \(E\) gives
\[
E(\Phi(z,\eta))=\frac12\Hn{z}(\eta,\eta)+\mathcal R(z,\eta),
\qquad
|\mathcal R(z,\eta)|\leq C|\eta|^3.
\]
Since \(\entropy(\widetilde\gamma_{\varepsilon,R}\,|\,\gibbs)=\int_\Omega\log\frac{\dd\widetilde\gamma_{\varepsilon,R}}{\dd\gibbs}\dd\widetilde\gamma_{\varepsilon,R}\), we may compute this integral in tubular coordinates by \eqref{eq: recovery sequence with wrong marginals 1}, that is, integrate the above identity against \(K_R(z,s)\dd s\,\dd\overline\gamma(z)\). We obtain
\begin{align*}
\entropy\left(\widetilde\gamma_{\varepsilon,R}\,|\,\gibbs \right)
 &=\int_\Sigma\log\overline\rho(z)\overline\gamma(\dd z)
 +\int_\Sigma\log\left(\frac{Z_\varepsilon}
 {\varepsilon^{k/2}Z_R(z)}\right)
 \overline\gamma(\dd z)
 \\
 &\qquad-\int_\Sigma\int_{\ball_R}
 \Bigl(\log\rho(\Phi(z,\sqrt\varepsilon s))
       -\frac{\mathcal R(z,\sqrt\varepsilon s)}{\varepsilon}
 \\
 &\qquad\qquad\qquad\qquad
       +\log\mathcal J_\Phi(z,\sqrt\varepsilon s)\Bigr)
 K_R(z,s)\dd s\,\overline\gamma(\dd z).
\end{align*}
These integrals are finite: since \(\entropy(\overline\gamma\,|\,\mathbb P_0)<\infty\) and the density of \(\mathbb P_0\) with respect to \(\mathcal H^d_{|\Sigma}\) is bounded from above and below by positive constants, we have
\(\int_\Sigma|\log\overline\rho|\,\dd\overline\gamma<\infty\).
Since \(|\mathcal R(z,\sqrt\varepsilon s)|\leq C\varepsilon^{3/2}|s|^3\), since \(\log\rho\) is Lipschitz (\(\rho\) being Lipschitz and bounded below by a positive constant), and since \(|\log\mathcal J_\Phi(z,\sqrt\varepsilon s)|\leq C\sqrt\varepsilon|s|\) by \eqref{eq: tube-jacobian-rescaled}, the double integral in the last line differs from \(\int_\Sigma\log\rho\dd\overline\gamma\) by at most
\[
C\sqrt\varepsilon\int_\Sigma\int_{\ball_R}\bigl(|s|+|s|^3\bigr)K_R(z,s)\dd s\,\overline\gamma(\dd z)\leq C\sqrt\varepsilon\longrightarrow0,
\]
by \eqref{eq: gaussian-moments}.
Moreover, by \Cref{thm: convergence-Gibbs-measure},
\(Z_\varepsilon=(2\pi\varepsilon)^{k/2}(Z_0+o(1))\), so that, uniformly in \(z\),
\[
 \log\left(\frac{Z_\varepsilon}
 {\varepsilon^{k/2}Z_R(z)}\right)
 =\log\frac{(2\pi)^{k/2}Z_0}{Z_R(z)}+o(1)
 =\log\left(Z_0\sqrt{\det \Hn{z}}\right)
  +\log\frac{Z_\infty(z)}{Z_R(z)}+o(1).
\]
Combining these identities yields
\begin{align*}
 \entropy\left(\widetilde\gamma_{\varepsilon,R}\,|\,\gibbs \right)
 &\longrightarrow
 \int_\Sigma\log\frac{\overline\rho(z)}{\rho(z)}\overline\gamma(\dd z)
 +\int_\Sigma\log\left(Z_0\sqrt{\det \Hn{z}}\right)
 \overline\gamma(\dd z)
 \\
 &\qquad+\int_\Sigma\log\frac{Z_\infty(z)}{Z_R(z)}
 \overline\gamma(\dd z)
 \\
 &=\entropy\left(\overline\gamma\,|\,\mathbb P_0\right)+b_R(\overline\gamma).
\end{align*}
Finally, the Gaussian tail estimate
\[
 0\leq Z_\infty(z)-Z_R(z)
 \leq\int_{|s|\geq R}\exp\left(-\frac{\lambda}{2}|s|^2\right)\,\dd s
 \longrightarrow0\qquad(R\to\infty)
\]
is uniform in \(z\). Together with \(Z_R(z)\geq c>0\), this shows
that \(b_R\geq0\) tends to zero uniformly over probability measures
on \(\Sigma\), and proves the claim.
\end{proof}

\subsection{Correction of the marginals}
\label{subsec: correction}
The normal perturbation \(\widetilde\gamma_{\varepsilon,R}\) has the expected entropy (\Cref{lemma: convergence entropy}), but not the prescribed marginals. To correct them, we use the shadow construction of \cite{eckstein2022quantitative}:
each perturbed marginal is transported to the prescribed one by an optimal transport map, and
the product of these maps is applied to the plan, as illustrated in the diagram below. In \cite{eckstein2022quantitative}, the construction is carried out more generally with optimal transport plans; here it is simpler, since the perturbed marginals are absolutely continuous, so that optimal transport maps exist.

\begin{center} 
\begin{tikzpicture}[scale=1,>=stealth]
\node[below left] at (0,-0.5) {\(\mu_1\)};
\node[above left] at (0,1) {\(\widetilde\mu_1\)};
\node[above right] at (3,1) {\(\widetilde\mu_2\)};
\node[below right] at (3,-0.5) {\(\mu_2\)};

\node[above] at (1.6,1.4) {\(\widetilde{\gamma}\)};
\node[left] at (-0.4,0.3) {\(T_1\)};
\node[right] at (3.4,0.3) {\(T_2\)};

\draw[<-, thick] (-0.3,-0.4) -- (-0.3,1);
\draw[<->, thick] (0.1,1.3) -- (2.9,1.3);
\draw[->, thick] (3.3,1) -- (3.3,-0.4);
\end{tikzpicture}
\end{center}

\begin{definition}[Shadow {\cite[Definition~3.1]{eckstein2022quantitative}}]
\label{def: shadow} 
Let \(\mu_i, \widetilde{\mu}_i \in \mathcal{P}(M_i)\) with \(\widetilde\mu_i\ll\vol_{M_i}\), and let \(T_i\) denote the optimal transport map from \(\widetilde \mu_i\) to \(\mu_i\) for the quadratic cost, \(i=1,2\). Then, for a given \(\widetilde{\gamma} \in \Pi(\widetilde{\mu}_1,\widetilde\mu_2)\), its \emph{shadow} \(\gamma \in \Pi(\mu_1,\mu_2)\) is defined by
\[
\gamma \defeq \big(T_1,T_2\big)_\# \widetilde{\gamma}.
\]
\end{definition}
The interest of this construction is that it turns a plan with good properties but mismatched marginals into a plan with the prescribed marginals, which remains close to the original one. More precisely, the distance between the shadow and the original plan depends only on the distance between the marginals, and its relative entropy is controlled by that of the original plan, as stated in the following lemma.

\begin{lemma}[Properties of the shadow {\cite[Lemma~3.2]{eckstein2022quantitative}}]
\label{prop: bound entropy shadow}
Given \(\widetilde{\gamma} \in \Pi(\widetilde{\mu}_1,\widetilde{\mu}_2)\) and its shadow \(\gamma\in \Pi(\mu_1,\mu_2)\), one has
\[
    W_2(\gamma,\widetilde{\gamma})^2 =\sum_{i=1}^2 W_2(\mu_i,\widetilde{\mu}_i)^2.
\]
Furthermore, the relative entropy of the shadow is bounded by that of the original plan:
\[
    \entropy\left(\gamma \,|\,\mu_1 \otimes \mu_2\right) \leq \entropy\left(\widetilde{\gamma} \,|\, \widetilde\mu_1\otimes\widetilde\mu_2\right).
\]
\end{lemma}

We denote by \(\gamma_{\varepsilon,R}\) the shadow of \(\widetilde\gamma_{\varepsilon,R}\). By \Cref{prop: bound entropy shadow} and \eqref{eq: normal-perturbation-W2}, \(W_2(\gamma_{\varepsilon,R},\widetilde\gamma_{\varepsilon,R})\leq\sqrt{2\varepsilon}\,R\), so that \(\gamma_{\varepsilon,R}\) still converges to \(\overline\gamma\). However, to remain a recovery sequence, it must also have the same limit energy. It turns out that passing to the shadow adds at most an energy term.

\begin{proposition}[Entropy of the shadow]
\label{prop: control of the shadow entropy}
For every \(R\geq1\) and \(\varepsilon>0\) with \(\sqrt\varepsilon R<r_0\),
\begin{equation}
\label{eq: shadow-entropy-stability}
 \entropy\left(\gamma_{\varepsilon,R}\,|\,\gibbs\right)
\leq\entropy(\widetilde\gamma_{\varepsilon,R}\,|\,\gibbs)
 +\frac1\varepsilon\int_\Omega E\,
 \dd(\gamma_{\varepsilon,R}-\widetilde\gamma_{\varepsilon,R}).
\end{equation}
\end{proposition}

\begin{proof}
In this proof, we write \(\gamma=\gamma_{\varepsilon,R}\), \(\widetilde\gamma=\widetilde\gamma_{\varepsilon,R}\) and \(\widetilde\mu_i=\widetilde\mu_{i,\varepsilon,R}\). Recall that \(\mathbb P_\varepsilon=f_\varepsilon\,\mu^\otimes\), where \(f_\varepsilon\defeq Z_\varepsilon^{-1}e^{-E/\varepsilon}\) (\Cref{prop: EOT-Gibbs}).
By \Cref{lemma: convergence entropy}, \(\entropy(\widetilde\gamma\,|\,\gibbs)<\infty\), hence \(\entropy(\widetilde\gamma\,|\,\mu^\otimes)\le\entropy(\widetilde\gamma\,|\,\gibbs)-\log Z_\varepsilon<\infty\) since \(E\geq0\).
Hence the marginal entropies \(\entropy(\widetilde\mu_i\,|\,\mu_i)\) are finite by projection, and so are all the
quantities below. By a change of reference measure, \(\entropy(\cdot\,|\,\gibbs)=\entropy(\cdot\,|\,\mu^\otimes)-\int_\Omega\log f_\varepsilon\,\dd(\cdot)\) on probability measures. Together with \Cref{prop: bound entropy shadow}, this gives
\begin{align*}
\entropy\left(\gamma\,|\,\gibbs\right)
 &=\entropy\left(\gamma\,|\,\mu^\otimes\right)
 -\int_\Omega\log f_\varepsilon\,\dd\gamma
 \\
 &\leq\entropy\left(\widetilde\gamma\,|\,\widetilde\mu_1\otimes\widetilde\mu_2\right)
 -\int_\Omega\log f_\varepsilon\,\dd\gamma\\
 &=\entropy\left(\widetilde\gamma\,|\,\gibbs\right)
 +\int_\Omega\log f_\varepsilon\,
 \dd(\widetilde\gamma-\gamma)
 \\
 &\qquad+\entropy\left(\widetilde\gamma\,|\,\widetilde\mu_1\otimes\widetilde\mu_2\right) -\entropy\left(\widetilde\gamma\,|\,\mu^\otimes\right).
\end{align*}
Since \(\widetilde\mu_i\ll\mu_i\) (\Cref{subsec: normal-recovery}), there exist
Borel functions \(a_i:M_i\to[0,\infty)\) such that
\(\widetilde\mu_i=a_i\,\mu_i\), so that
\(\widetilde\mu_1\otimes\widetilde\mu_2=a_1(x_1)a_2(x_2)\,\mu^\otimes\).
Consequently,
\begin{align*}
 \entropy\left(\widetilde\gamma\,|\,\widetilde\mu_1\otimes\widetilde\mu_2\right)-\entropy\left(\widetilde\gamma\,|\,\mu^\otimes\right)
 &=-\int_\Omega\log\bigl(a_1(x_1)a_2(x_2)\bigr)\dd\widetilde\gamma(x_1,x_2)
 \\
 &=-\sum_{i=1}^2\int_{M_i}\log a_i\,\dd\widetilde\mu_i
 =-\sum_{i=1}^2\entropy\left(\widetilde\mu_i\,|\,\mu_i\right)\leq0,
\end{align*}
where the last inequality follows from the nonnegativity of the relative
entropy between probability measures. Finally, since \(\log f_\varepsilon=-E/\varepsilon-\log Z_\varepsilon\) and both plans have total mass one,
\[
 \entropy\left(\gamma\,|\,\gibbs \right)
\leq\entropy\left(\widetilde\gamma\,|\,\gibbs\right)
 -\int_\Omega\log f_\varepsilon\,
 \dd(\gamma-\widetilde\gamma)
=\entropy\left(\widetilde\gamma\,|\,\gibbs\right)
 +\frac1\varepsilon\int_\Omega E\,
 \dd(\gamma-\widetilde\gamma).
\]
\end{proof}

The additional term in \eqref{eq: shadow-entropy-stability} tends to \(0\) as soon as \(W_2(\widetilde\mu_{i,\varepsilon,R},\mu_i)=o(\sqrt\varepsilon)\) (\Cref{prop: shadow-energy} below). The following estimate is therefore the cornerstone of the construction of the recovery sequence.
\begin{theorem}[Marginal estimate]
\label{thm: Wasserstein between the marginals}
For each fixed \(R\geq1\), the marginals of the normal perturbation
\eqref{eq: recovery sequence with wrong marginals 1} satisfy
\begin{equation}
\label{eq: marginal-cancellation}
W_2(\widetilde\mu_{i,\varepsilon,R},\mu_i)=o_R(\sqrt\varepsilon)
 \qquad (i=1,2),\quad\text{as }\varepsilon\to0,
\end{equation}
where the subscript indicates that the rate of convergence may depend on \(R\).
\end{theorem}
The proof is given in \Cref{subsection: proof of the main result}.

We now verify that the additional term in \eqref{eq: shadow-entropy-stability} vanishes in the limit.

\begin{proposition}[Energy of the correction]
\label{prop: shadow-energy}
For every fixed \(R\geq1\),
\[
 \frac1\varepsilon\left|\int_\Omega E\,
 \dd(\gamma_{\varepsilon,R}-\widetilde\gamma_{\varepsilon,R})\right|
 \longrightarrow0
 \qquad\text{as }\varepsilon\to0.
\]
\end{proposition}

\begin{proof}
Let \(T_{\varepsilon,R}\defeq(T_1,T_2)\), where \(T_i\) is the optimal transport map from \(\widetilde\mu_{i,\varepsilon,R}\) to \(\mu_i\). By definition of the shadow,
\(\gamma_{\varepsilon,R}
 =(T_{\varepsilon,R})_\#\widetilde\gamma_{\varepsilon,R}\).
Set
\[
 \Delta_{\varepsilon,R}^2
 \defeq\int_\Omega \dist(\x,T_{\varepsilon,R}(\x))^2
 \,\dd\widetilde\gamma_{\varepsilon,R}(\x)
 =\sum_{i=1}^2W_2^2(\widetilde\mu_{i,\varepsilon,R},\mu_i).
\]
The pushforward identity gives
\[
 \frac1\varepsilon\int_\Omega E\,
 \dd(\gamma_{\varepsilon,R}-\widetilde\gamma_{\varepsilon,R})
 =\frac1\varepsilon\int_\Omega
 \bigl(E(T_{\varepsilon,R}(\x))-E(\x)\bigr)
 \,\dd\widetilde\gamma_{\varepsilon,R}(\x).
\]

We identify \(\Omega\) with \(\mathbb T^{2d}=\R^{2d}/\mathbb Z^{2d}\), denote by \([y]\in\Omega\) the class of \(y\in\R^{2d}\), and by \(\widetilde E\defeq E([\,\cdot\,])\) the \(\mathbb Z^{2d}\)-periodic lift of \(E\), which belongs to \(C^{2,1}(\R^{2d})\) by \ref{H2}. We set
\[
 C_E\defeq\sup_{y\in\R^{2d}}\|D^2\widetilde E(y)\|<\infty,
\]
where \(\|\cdot\|\) denotes the operator norm.

We first estimate the gradient by the distance to \(\Sigma\). Let \(y\in\R^{2d}\). Since \(\Sigma\) is compact, there exists \(\bar y\in\R^{2d}\) such that \([\bar y]\in\Sigma\) and \(|y-\bar y|=\dist([y],\Sigma)\). Since \(\nabla\widetilde E(\bar y)=0\), the fundamental theorem of calculus gives
\begin{equation}
\label{eq: gradient-distance}
 |\nabla\widetilde E(y)|
 =\left|\int_0^1 D^2\widetilde E\bigl(\bar y+t(y-\bar y)\bigr)(y-\bar y)\dd t\right|
 \leq C_E\dist([y],\Sigma).
\end{equation}
Now, for each \(\x\in\Omega\), choose \(\widetilde\x\in\R^{2d}\) with \([\widetilde\x]=\x\), and a shortest displacement \(v_{\x}\in\mathbb R^{2d}\) such that
\[
 [\widetilde\x+v_{\x}]=T_{\varepsilon,R}(\x),
 \qquad |v_{\x}|=\dist(\x,T_{\varepsilon,R}(\x)).
\]
By the fundamental theorem of calculus,
\[
 E(T_{\varepsilon,R}(\x))-E(\x)
 =\widetilde E(\widetilde\x+v_{\x})-\widetilde E(\widetilde\x)
 =\int_0^1\bigl\langle\nabla\widetilde E(\widetilde\x+\tau v_{\x}),
                         v_{\x}\bigr\rangle\,\dd\tau.
\]
For every \(\tau\in[0,1]\), the triangle inequality gives \(\dist([\widetilde\x+\tau v_{\x}],\Sigma)\leq\dist(\x,\Sigma)+\tau|v_{\x}|\). Together with \eqref{eq: gradient-distance}, we obtain
\begin{equation}
\label{eq: energy-difference-gradient}
 |E(T_{\varepsilon,R}(\x))-E(\x)|
 \leq C_E \dist(\x,T_{\varepsilon,R}(\x))
              \dist(\x,\Sigma)
 +C_E \dist(\x,T_{\varepsilon,R}(\x))^2.
\end{equation}
Integrating \eqref{eq: energy-difference-gradient} with respect to \(\widetilde\gamma_{\varepsilon,R}\) and applying the Cauchy--Schwarz inequality gives
\begin{align*}
 \frac1\varepsilon\left|\int_\Omega E\,
 \dd(\gamma_{\varepsilon,R}-\widetilde\gamma_{\varepsilon,R})\right|
 &\leq\frac{C_E\Delta_{\varepsilon,R}}{\varepsilon}
 \left(\int_\Omega \dist(\x,\Sigma)^2
 \,\dd\widetilde\gamma_{\varepsilon,R}(\x)\right)^{1/2}
 \\
 &\qquad+\frac{C_E\Delta_{\varepsilon,R}^2}{\varepsilon}.
\end{align*}
Recalling the definition \eqref{eq: recovery sequence with wrong marginals 1} of \(\widetilde\gamma_{\varepsilon,R}\) and using \eqref{eq: normal-distance}, we have
\[
 \int_\Omega \dist(\x,\Sigma)^2
 \,\dd\widetilde\gamma_{\varepsilon,R}(\x)
 =\varepsilon\int_\Sigma
 \int_{\ball_R}|s|^2
 K_R(z,s)\dd s\,\overline\gamma(\dd z)
 \leq M_2\,\varepsilon,
\]
by \eqref{eq: gaussian-moments}.
Hence
\begin{equation}
\label{eq: shadow-energy-diagonal}
 \frac1\varepsilon\left|\int_\Omega E\,
 \dd(\gamma_{\varepsilon,R}-\widetilde\gamma_{\varepsilon,R})\right|
 \leq C\left(\frac{\Delta_{\varepsilon,R}}{\sqrt\varepsilon}
            +\frac{\Delta_{\varepsilon,R}^2}{\varepsilon}\right).
\end{equation}
For fixed \(R\), \Cref{thm: Wasserstein between the marginals}
gives \(\Delta_{\varepsilon,R}=o_R(\sqrt\varepsilon)\), so that
both terms on the right-hand side tend to zero.
\end{proof}

\subsection{Proof of Theorem~A}
\label{subsec: selection-value-asymptotics}

\begin{proof}[Proof of {\hyperref[thm: A]{Theorem~A}}]
\textbf{Step 1. The \(\Gamma\)-\(\liminf\) inequality \eqref{eq: gamma-liminf}.} Let \(\gamma_\varepsilon\rightharpoonup\overline\gamma\). If \(\overline\gamma\notin\Pi(\mu_1,\mu_2)\), then, since \(\Pi(\mu_1,\mu_2)\) is closed, \(\gamma_\varepsilon\notin\Pi(\mu_1,\mu_2)\) for \(\varepsilon\) small enough, so that \(\mathcal G_\varepsilon(\gamma_\varepsilon)=+\infty\). Otherwise, the joint lower semicontinuity of the relative entropy, together with the convergence \(\mathbb P_\varepsilon\rightharpoonup\mathbb P_0\) given by \Cref{thm: convergence-Gibbs-measure}, yields
\[
\liminf_{\varepsilon\to0}\mathcal G_\varepsilon(\gamma_\varepsilon)
\geq\liminf_{\varepsilon\to0} \entropy\left( \gamma_\varepsilon \,\middle\vert\, \mathbb{P}_\varepsilon \right) \geq \entropy\left( \overline \gamma \,\middle\vert\, \mathbb{P}_0 \right)=\mathcal G_0(\overline\gamma).
\]

\textbf{Step 2. The \(\Gamma\)-\(\limsup\) inequality \eqref{eq: gamma-limsup}.} Let \(\overline\gamma\in\mathcal P(\Omega)\) with \(\mathcal G_0(\overline\gamma)<\infty\) (otherwise there is nothing to prove). For \(R\geq1\) and \(\varepsilon>0\) with \(\sqrt\varepsilon R<r_0\), let \(\widetilde\gamma_{\varepsilon,R}\) be the normal perturbation of \(\overline\gamma\) defined in \eqref{eq: recovery sequence with wrong marginals 1}, and \(\gamma_{\varepsilon,R}\in\Pi(\mu_1,\mu_2)\) its shadow. We first fix \(R\). By \Cref{lemma: convergence entropy}, \Cref{prop: control of the shadow entropy} and \Cref{prop: shadow-energy},
\[
 \limsup_{\varepsilon\to0}\mathcal G_\varepsilon(\gamma_{\varepsilon,R})
 =\limsup_{\varepsilon\to0}\entropy\left(\gamma_{\varepsilon,R}\,|\,\gibbs\right)
 \leq\mathcal G_0(\overline\gamma)+b_R(\overline\gamma),
\]
while \Cref{prop: bound entropy shadow} and \eqref{eq: normal-perturbation-W2} give
\[
W_2(\gamma_{\varepsilon,R},\overline\gamma)
\leq W_2(\gamma_{\varepsilon,R},\widetilde\gamma_{\varepsilon,R})+W_2(\widetilde\gamma_{\varepsilon,R},\overline\gamma)
\leq(\sqrt2+1)\sqrt\varepsilon R
\longrightarrow0
\qquad(\varepsilon\to0),
\]
so that \(\gamma_{\varepsilon,R}\rightharpoonup\overline\gamma\), since \(W_2\) metrizes the weak convergence on the compact set \(\Omega\). Hence \((\gamma_{\varepsilon,R})_{\varepsilon}\) is a recovery sequence for \(\overline\gamma\) up to the error \(b_R(\overline\gamma)\), which tends to \(0\) as \(R\to\infty\) by \Cref{lemma: convergence entropy}. As is standard in \(\Gamma\)-convergence, this suffices to obtain a recovery sequence by a diagonal argument (see, e.g., \cite{dalmaso1993introduction}), which we recall for completeness. Choose \(R_j\uparrow\infty\) and, inductively, \(\varepsilon_j\downarrow0\) such that, for every \(0<\varepsilon\leq\varepsilon_j\), we have \(\sqrt\varepsilon R_j<r_0\),
\(W_2(\gamma_{\varepsilon,R_j},\overline\gamma)\leq1/j\), and
\(\mathcal G_\varepsilon(\gamma_{\varepsilon,R_j})\leq\mathcal G_0(\overline\gamma)+b_{R_j}(\overline\gamma)+1/j\).
Then \(\gamma_\varepsilon\defeq\gamma_{\varepsilon,R_j}\) for \(\varepsilon_{j+1}<\varepsilon\leq\varepsilon_j\) defines a recovery sequence for \(\overline\gamma\).

\textbf{Step 3. Conclusion.}
By \eqref{eq: G-vs-F}, \eqref{eq: gamma-expansion} follows from
the \(\Gamma\)-convergence of \(\mathcal G_\varepsilon\) to \(\mathcal G_0\) established in Steps~1 and~2, since adding to \(\mathcal G_\varepsilon\) constants converging as \(\varepsilon\to0\) does not affect \(\Gamma\)-convergence.
The infimum of \(\mathcal F^{(1)}\) is finite. Indeed, since relative entropies are nonnegative, \(\mathcal F^{(1)}\geq-\log Z_0-\frac k2\log(2\pi)\), and it suffices to show that every optimal transport plan \(\overline\gamma\) satisfies \(\entropy(\overline\gamma\,|\,\mathbb P_0)<\infty\). By \eqref{eq: optimality-support}, such a plan satisfies \(\overline\gamma(\Omega\setminus\Sigma)=0\). For every Borel set \(A\subseteq\Sigma\), since \(\proj_{M_1}\) is \(1\)-Lipschitz, we have
\[
\overline\gamma(A)
\leq\mu_1\bigl(\proj_{M_1}(A)\bigr)
\leq\beta\,\mathcal H^d\bigl(\proj_{M_1}(A)\bigr)
\leq\beta\,\mathcal H^d(A),
\]
so that \(\overline\gamma\leq\beta\,\mathcal H^d_{|\Sigma}\). Since the density of \(\mathbb P_0\) with respect to \(\mathcal H^d_{|\Sigma}\) is bounded from below by a positive constant, by \ref{H4} and the continuity of \(\hess^\perp E\) on the compact set \(\Sigma\), the density of \(\overline\gamma\) with respect to \(\mathbb P_0\) is bounded, hence \(\entropy(\overline\gamma\,|\,\mathbb P_0)<\infty\).
Compactness and convergence of minima under \(\Gamma\)-convergence
then give \eqref{eq: energy-expansion}.

For (iii), note that, by \Cref{prop: EOT-Gibbs}, \(\mathcal F_\varepsilon\) and \(\mathcal G_\varepsilon\) have the same minimizers, and so do \(\mathcal F^{(1)}\) and \(\mathcal G_0\). Since \(\mathcal P(\Omega)\) is weakly compact and the functionals \(\mathcal G_\varepsilon\), \(\mathcal G_0\) are lower semicontinuous, strictly convex on their domains and not identically \(+\infty\), they admit unique minimizers \(\gamma_\varepsilon\) and \(\gamma_0\). By the fundamental theorem of \(\Gamma\)-convergence (see, e.g., \cite{dalmaso1993introduction}), every cluster point of \((\gamma_\varepsilon)_{\varepsilon>0}\) as \(\varepsilon\to0\) minimizes \(\mathcal G_0\), hence \(\gamma_\varepsilon\rightharpoonup\gamma_0\). Finally, since \(\mathcal G_0(\gamma_0)<\infty\), we have \(\gamma_0\ll\mathbb P_0\), so that \(\spt(\gamma_0)\subseteq\Sigma\) and \(\gamma_0\) is an optimal transport plan by \eqref{eq: optimality-support}, hence the unique minimizer of \(\entropy(\cdot\,|\,\mathbb P_0)\) among all optimal transport plans.
\end{proof}

\subsection[Proof of the marginal estimate]{Proof of \Cref{thm: Wasserstein between the marginals}}
\label{subsection: proof of the main result}
Throughout this section, \(R\geq1\) is fixed, and we write \(K=K_R\), \(\widetilde\gamma_\varepsilon=\widetilde\gamma_{\varepsilon,R}\) and \(\widetilde\mu_{i,\varepsilon}=\widetilde\mu_{i,\varepsilon,R}\). Recall that, by \eqref{eq: recovery sequence with wrong marginals 1}, the kernel \(K\) does not depend on \(\varepsilon\).

Let us first explain where the difficulty lies. If \(\Sigma\) is flat and \(\hess^\perp E\) is constant along \(\Sigma\), the normal directions and the kernel \(K(z,\cdot)\) do not depend on \(z\), so that \(\widetilde\gamma_\varepsilon\) is the convolution of \(\overline\gamma\) with a fixed truncated Gaussian of variance of order \(\varepsilon\), supported on the normal space. Each marginal \(\widetilde\mu_{i,\varepsilon}\) is then the convolution of \(\mu_i\) with the projection \(\eta_\varepsilon\) of this Gaussian, which is a centered kernel of variance of order \(\varepsilon\). For a general probability measure \(\mu\), \(W_2(\mu,\mu*\eta_\varepsilon)=O(\sqrt\varepsilon)\), which is not enough for our purpose, since we need \(o(\sqrt\varepsilon)\). This improves to \(O(\varepsilon^{(1+\theta)/2})\) if \(\mu\) has a \(C^{1,\theta}\) density bounded from below, with \(\theta\in(0,1]\), the first-order contributions cancelling out by symmetry. In general, however, \(\Sigma\) is not flat, and several branches of \(\Sigma\), that is, pieces of \(\Sigma\) on which the projection onto \(M_i\) is a diffeomorphism, may lie above the same point of \(M_i\). The first issue is mild, since \(\Sigma\) is locally almost flat. The second one is more serious: \(\widetilde\mu_{i,\varepsilon}\) is the sum of the contributions of the branches, each obtained by regularizing the partial marginal of \(\overline\gamma\) carried by the branch, and, although these partial marginals sum up to \(\mu_i\), their densities need not inherit the regularity of \(\mu_i\): they are merely \(L^1\). We therefore need to handle \(L^1\) densities regularized along a non-flat manifold. To explain the strategy, we first consider the flat case with one-dimensional marginals, and show in \Cref{subsubsec: 1D} that \(W_2(\mu,\mu*\eta_\varepsilon)=o(\sqrt\varepsilon)\) as soon as \(\mu\) has an \(L^1\) density. The argument relies again on the symmetry of the Gaussian kernel.

\subsubsection{A model case: one-dimensional marginals and flat contact set}
\label{subsubsec: 1D}
To illustrate the proof of \Cref{thm: Wasserstein between the marginals}, we first consider one-dimensional marginals and a flat contact set. As explained above, the marginals of the normal perturbation are then convolutions of the original marginals with a fixed truncated Gaussian kernel, and we estimate the Wasserstein distance between a measure \(\mu\) on \(\R\) and its convolution with such a kernel, assuming only that \(\mu\) is absolutely continuous, that is, has an \(L^1\) density. A similar approach was introduced in one dimension in \cite{colombo2025first}. Here we adapt and clarify it before extending it to higher dimensions.
\begin{proposition}[One-dimensional flat case]
\label{prop: 1D}
Let \(\mu\in\mathcal P(\R)\) be absolutely continuous with respect to the Lebesgue measure, with finite second moment. For \(R>0\) and \(\varepsilon>0\), let
\[
\eta_\varepsilon(x)\defeq\frac{1}{\sqrt\varepsilon\,c_R}\,e^{-x^2/(2\varepsilon)}\,\mathbf1_{\{|x|\leq\sqrt\varepsilon R\}},
\qquad
c_R\defeq\int_{-R}^Re^{-z^2/2}\dd z.
\]
Then \(\widetilde\mu_\varepsilon\defeq\mu*\eta_\varepsilon\) satisfies
\[
W_2(\mu,\widetilde\mu_\varepsilon)=\sqrt R\,o(\sqrt\varepsilon)\qquad(\varepsilon\to0),
\]
where the \(o(\sqrt\varepsilon)\) depends only on \(\mu\), and not on \(R\).
\end{proposition}
\begin{proof}
Throughout the proof, we still denote by \(\mu\) the density of \(\mu\). For \(t\in[0,1]\), let
\[
\mu_{t,\varepsilon}
\defeq\mu*(x\mapsto t\sqrt\varepsilon x)_\#
\left[\frac{1}{c_R}e^{-x^2/2}
\mathbf1_{\{|x|\le R\}}\dd x\right],
\]
where convolution and push-forward are understood as operations on
measures, so that \(\mu_{0,\varepsilon}=\mu\), \(\mu_{1,\varepsilon}=\widetilde\mu_\varepsilon\), and, for every \(t\in[0,1]\),
\[
\mu_{t,\varepsilon}(x)
=\frac{1}{c_R}\int_{-R}^R
\mu(x-t\sqrt\varepsilon z)e^{-z^2/2}\dd z.
\]

\textbf{Step 1. Computing the flux.}
For \(\phi\in C_c^\infty(\mathbb R)\) and \(t\in[0,1]\),
\begin{align*}
\frac{\dd}{\dd t}\int_\R\phi\dd\mu_{t,\varepsilon}
&=\frac{\dd}{\dd t}\frac{1}{c_R}\int_{-R}^R\int_\R
\phi(y+t\sqrt\varepsilon z)\mu(y)e^{-z^2/2}\dd y\dd z\\
&=\frac{1}{c_R}\int_{-R}^R\int_\R
\phi'(y+t\sqrt\varepsilon z)\sqrt\varepsilon z\,
\mu(y)e^{-z^2/2}\dd y \dd z
=\int_\R\phi'(x)\flow(x)\dd x,
\end{align*}
where, by the change of variables \(y=x-t\sqrt\varepsilon z\),
\[
\flow(x)\defeq\frac{1}{c_R}\int_{-R}^R
\sqrt\varepsilon z\,\mu(x-t\sqrt\varepsilon z)
e^{-z^2/2}\dd z.
\]
Differentiation under the integral sign is justified since \(\phi'\) is bounded, \(\mu\in L^1\) and \(\int_{-R}^R |z|e^{-z^2/2}\dd z<\infty\).
Thus \(\partial_t\mu_{t,\varepsilon}+\partial_x\flow=0\) in the sense of
distributions. Note that \(\mathfrak m_{0,\varepsilon}=0\) by
symmetry of the Gaussian kernel.

Let us estimate the total variation \(\|\flow\|_{\mathrm{TV}}=\int_\R|\flow|\dd x\). The flux depends linearly on \(\mu\): for \(f\in L^1(\R)\), let
\[
\flow[f](x)\defeq\frac{1}{c_R}\int_{-R}^R
\sqrt\varepsilon z\,f(x-t\sqrt\varepsilon z)
e^{-z^2/2}\dd z,
\qquad\text{so that }\flow=\flow[\mu].
\]
We will use the following bounds on the moments of the truncated Gaussian, which are uniform in \(R\) and follow from an integration by parts and the Cauchy--Schwarz inequality:
\begin{align}
\label{eq: 1D-second-moment}
\frac{1}{c_R}\int_{-R}^R z^2e^{-z^2/2}\dd z
&=1-\frac{2Re^{-R^2/2}}{c_R}\le1,
\\
\label{eq: 1D-first-moment}
\frac{1}{c_R}\int_{-R}^R |z|e^{-z^2/2}\dd z
&\le\left(\frac{1}{c_R}\int_{-R}^R z^2e^{-z^2/2}\dd z\right)^{1/2}
\le1.
\end{align}
By Fubini's theorem, the translation invariance of the Lebesgue measure and \eqref{eq: 1D-first-moment}, we get
\[
\|\flow[f]\|_{L^1}\le\sqrt\varepsilon\|f\|_{L^1}
\qquad\text{for every }f\in L^1(\R).
\]
If moreover \(f\in C_c^1(\R)\), we write, with \(x_s\defeq x-s\sqrt\varepsilon z\),
\[
f(x-t\sqrt\varepsilon z)
=f(x)-\int_0^t\sqrt\varepsilon z f'(x_s)\dd s.
\]
Since the kernel is centered, that is,
\(\int_{-R}^R z e^{-z^2/2}\dd z=0\),
the term containing \(f(x)\) cancels, so that, by \eqref{eq: 1D-second-moment},
\[
\|\flow[f]\|_{L^1}
=\int_\R\left|\frac{\varepsilon}{c_R}\int_{-R}^R
z^2e^{-z^2/2}\int_0^tf'(x_s)\dd s\dd z\right|\dd x
\le t\varepsilon\|f'\|_{L^1}
\le\varepsilon\|f'\|_{L^1}.
\]
Now let \((\mu^k)_{k\in\N}\) be a sequence in \(C_c^1(\R)\)
such that \(\mu^k\to\mu\) in \(L^1\). Writing \(\flow=\flow[\mu-\mu^k]+\flow[\mu^k]\), we obtain, for every \(k\in\N\) and \(t\in[0,1]\),
\[
\frac{\|\flow\|_{\mathrm{TV}}}{\sqrt\varepsilon}
\le\|\mu-\mu^k\|_{L^1}
+\sqrt\varepsilon\|(\mu^k)'\|_{L^1}.
\]
Define, for \(r>0\),
\[
\omega(r)\defeq\inf_{k\in\N}\left\{
\|\mu-\mu^k\|_{L^1}+r\|(\mu^k)'\|_{L^1}\right\}.
\]
For each fixed \(k\),
\(0\le\limsup_{r\downarrow0}\omega(r)\le\|\mu-\mu^k\|_{L^1}\), and
letting \(k\to\infty\) shows that \(\omega(r)\to0\) as \(r\to0\). Therefore
\begin{equation}
\label{eq: 1D-flux-bound}
\sup_{t\in[0,1]}\frac{\|\flow\|_{\mathrm{TV}}}{\sqrt\varepsilon}
\le\omega(\sqrt\varepsilon)\longrightarrow0
\qquad(\varepsilon\to0).
\end{equation}

\textbf{Step 2. Estimate of \(W_2(\mu,\widetilde\mu_\varepsilon)\).}
By Step~1, the curve \((\mu_{t,\varepsilon})_{t\in[0,1]}\) of probability measures with finite second moments solves, together with \((\flow)_{t\in[0,1]}\), the continuity equation, and it is weakly continuous, with \(\mu_{0,\varepsilon}=\mu\) and \(\mu_{1,\varepsilon}=\widetilde\mu_\varepsilon\). Since \(|z|\le R\) in the definition of \(\flow\), we have
\[
|\flow(x)|
\le\frac{\sqrt\varepsilon R}{c_R}
\int_{-R}^R\mu(x-t\sqrt\varepsilon z)
e^{-z^2/2}\dd z
=\sqrt\varepsilon R\,\mu_{t,\varepsilon}(x)
\qquad\text{for a.e. }x\in\R.
\]
Thus \(\flow\ll\mu_{t,\varepsilon}\), and the velocity field
\[
v_{t,\varepsilon}
\defeq\frac{\dd \flow}{\dd\mu_{t,\varepsilon}}
\quad\text{satisfies}\quad |v_{t,\varepsilon}|\le\sqrt\varepsilon R
\quad\mu_{t,\varepsilon}\text{-a.e.}
\]
By the Benamou--Brenier formula (see, e.g., \cite{santambrogioOptimalTransportApplied2015}), with \((\mu_{t,\varepsilon},v_{t,\varepsilon})_{t\in[0,1]}\) as a competitor, and by \eqref{eq: 1D-flux-bound}, we get
\begin{align*}
W_2^2(\mu,\widetilde\mu_\varepsilon)
&\le\int_0^1\int_\R |v_{t,\varepsilon}|^2
\,\dd\mu_{t,\varepsilon}\dd t=\int_0^1\int_\R |v_{t,\varepsilon}|\,\dd|\flow|\dd t
\\
&\le\sqrt\varepsilon R\sup_{t\in[0,1]}
\|\flow\|_{\mathrm{TV}}
\le \varepsilon R\,\omega(\sqrt\varepsilon).
\end{align*}
Hence \(W_2(\mu,\widetilde\mu_\varepsilon)\le\sqrt R\,\sqrt{\omega(\sqrt\varepsilon)}\,\sqrt\varepsilon\), which concludes the proof, since \(\omega\) depends only on \(\mu\).
\end{proof}
\subsubsection{General statement}
Before turning to the proof, we introduce some notation and a technical lemma. We fix \(i\in\{1,2\}\) and write \(M\defeq M_i=\T^d=\R^d/\mathbb Z^d\). Every open ball of \(\T^d\) with radius less than \(1/2\) is the isometric image of a Euclidean ball under the canonical projection \(\R^d\to\T^d\). All the points involved in the computations below lie in such balls, which we identify with the corresponding Euclidean balls, so that sums, segments and convexity are understood in \(\R^d\). In this section and in \Cref{sec: extensions}, \(C\) denotes a positive constant that does not depend on \(\varepsilon\), whose value may change from line to line, and we write \(a\lesssim b\) when \(a\leq Cb\).
 
By \Cref{prop: twist-geometry}, \(\dim\Sigma=d\) and the projection \(\proj_M^\Sigma:\Sigma\to M\) is a local \(C^{1,1}\)-diffeomorphism.
More precisely, there exist a finite index set \(J\), an open covering
\((U_j)_{j\in J}\) of \(\Sigma\), and open sets
\(V_j\subseteq M\), which we may assume convex and contained in balls of radius \(1/4\), such that, for every \(j\in J\),
\[
    \proj_M^{U_j}:U_j\longrightarrow V_j
\]
is a \(C^{1,1}\)-diffeomorphism. We denote its inverse by
\[
    \alpha_j
    \defeq
    \left(\proj_M^{U_j}\right)^{-1}
    :V_j\longrightarrow U_j.
\]
As in \Cref{lemma: determinant of Phi}, the Jacobian \(\mathcal J_F\) of a map \(F\) between \(d\)-dimensional manifolds (here, open subsets of \(\R^d\) or of \(\Sigma\)) denotes the absolute value of the determinant of \(DF\) in orthonormal bases.
Let \((\chi_j)_{j\in J}\) be a partition of unity on
\(\Sigma\) subordinate to the covering \((U_j)_{j\in J}\), that is, each \(\chi_j:\Sigma\to[0,1]\) is continuous with compact support in \(U_j\) (recall that \(\Sigma\) is compact), and \(\sum_{j\in J}\chi_j=1\) on \(\Sigma\).
 
\begin{lemma}
\label{lemma: C^1 diffeo}
Let \(V\subseteq \R^d\) be an open set and \(f:V\to\R^d\) be \(1\)-Lipschitz. There exists \(C_d\geq2\), depending only on \(d\), such that, for \(0<\tau\leq1/C_d\), the map \(\psi_\tau(x)\defeq x+\tau f(x)\) is a bi-Lipschitz homeomorphism from \(V\) onto the open set \(\psi_\tau(V)\), and \(\kappa_\tau\defeq\psi_\tau^{-1}\) satisfies
\(\|\mathcal J_{\psi_\tau}-1\|_{L^\infty(V)}+\|\mathcal J_{\kappa_\tau}-1\|_{L^\infty(\psi_\tau(V))}
\leq
C_d\tau\).
\end{lemma}
 
\begin{proof}
Since \(f\) is \(1\)-Lipschitz and \(\tau<1\), we have \((1-\tau)|x-y|\leq|\psi_\tau(x)-\psi_\tau(y)|\leq(1+\tau)|x-y|\) for \(x,y\in V\). Thus \(\psi_\tau\) is a bi-Lipschitz homeomorphism onto its image, which is open by the invariance of domain theorem. By Rademacher's theorem, \(f\) is differentiable a.e.\ on \(V\), with \(\|Df\|\leq1\), and \(D\psi_\tau=\id+\tau Df\) there. Since the determinant is Lipschitz on the set of matrices \(B\) with \(\|B\|\leq1\), there exists \(C_d\) such that \(|\det(\id+B)-1|\leq C_d\|B\|\) for such \(B\). Applying this to \(B=\tau Df(x)\) gives \(|\mathcal J_{\psi_\tau}-1|\leq C_d\tau\) a.e.\ on \(V\). Finally, since the Lipschitz maps \(\psi_\tau\) and \(\kappa_\tau\) map negligible sets to negligible sets, for a.e.\ \(y\in\psi_\tau(V)\), \(\kappa_\tau\) is differentiable at \(y\) and \(\psi_\tau\) is differentiable at \(\kappa_\tau(y)\), so that \(\mathcal J_{\kappa_\tau}(y)=1/\mathcal J_{\psi_\tau}(\kappa_\tau(y))\) by the chain rule. Hence the same bound holds for \(\mathcal J_{\kappa_\tau}\), up to enlarging \(C_d\).
\end{proof}

\begin{proof}[Proof of \Cref{thm: Wasserstein between the marginals}]
Let \(\varepsilon>0\) be small enough so that \(\sqrt\varepsilon R<1/4\). All the displacements considered below are bounded by \(\sqrt\varepsilon R\), so that, starting from \(V_j\), they stay in a ball of radius \(1/2\), where the identification with \(\R^d\) holds. Recall that \(\overline \gamma = \overline \rho\, \mathcal H^d_{|\Sigma}\). We define the \(d\times k\) matrix (recall that \(k=d\))
\[A(z)
\defeq \Big(\proj_{M}\big(n_1(z)\big),\ldots,\proj_M\big(n_k(z)\big)\Big). \]
Equivalently, \(A(z)s
=
\sum_{p=1}^{k}
s_p \proj_{M}\big(n_p(z)\big)\). Therefore,
\[\proj_M\Phi(z,s)
=\proj_M(z+ n(z)s)= \proj_M(z) + A(z)s .\]
Since \(K(z,\cdot)\) is even, we have, for every \(z\in \Sigma\),
\begin{equation}
    \label{eq: symmetry of the Gaussian}
\int_{\ball_R} s K(z,s)\dd s=0.
\end{equation}

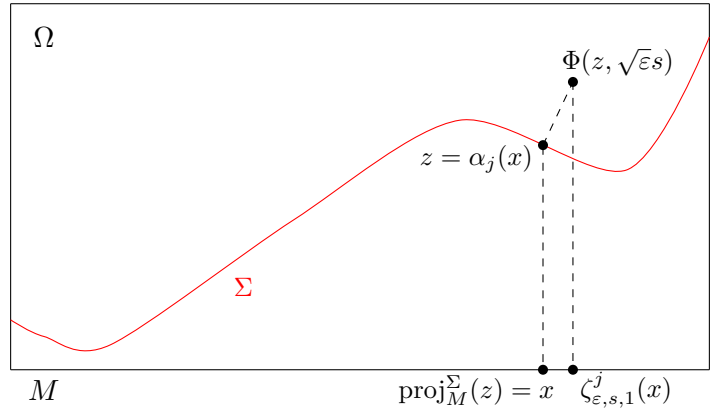
\begin{figure}[htbp]
    \centering
\begin{tikzpicture}[scale=2.2]
  \draw[-] (-0.7,0) -- (3.5,0);
  \draw[-] (-0.7,0) -- (-0.7,2.2);
  \draw[-] (-0.7,2.2) -- (3.5,2.2);
  \draw[-] (3.5,0) -- (3.5,2.2);
\draw[thin, red, smooth]
    plot coordinates {(-0.7,0.3) (-0.5,0.2) (-0.1,0.15)
    (1,0.9) (2,1.5) (3,1.2) (3.5,2)};
\draw[dashed] (2.5,1.35) -- (2.68,1.73);
\draw[dashed] (2.68,1.73) -- (2.68,0);
\draw[dashed] (2.5,1.34) -- (2.5,0);
\node at (2.95,1.85){\small \(\Phi(z,\sqrt{\varepsilon}s)\)};
\fill (2.68,1.73) circle (0.03);
\node at (2.1,1.27){\small \(z=\alpha_j(x)\)};
\fill (2.5,1.35) circle (0.03);
\node at (2.1,-0.12){\small \(\proj_M^\Sigma(z)=x\)};
\fill (2.5,0) circle (0.03);
\node at (3,-0.12){\small \(\zeta_{\varepsilon,s,1}^j(x)\)};
\fill (2.68,0) circle (0.03);
\node[red] at (0.7,0.5){\small \(\Sigma\)};
\node at (-0.5,2){\(\Omega\)};
\node at (-0.5,-0.12){\(M\)};
\end{tikzpicture}
 \caption{Normal perturbation of \(z\in\Sigma\) and its projection onto \(M\).}
    \label{fig: normal-perturbation}
\end{figure}
 
\textbf{Step 1. Construction of the interpolation.} As in \Cref{subsubsec: 1D}, we interpolate between \(\mu\) and \(\widetilde\mu_\varepsilon\) by letting the variance of the normal perturbation grow from \(0\) to \(\varepsilon\): for \(t\in[0,1]\), we let \(\mu_{t,\varepsilon}\defeq\widetilde\mu_{t^2\varepsilon}\) be the marginal of the normal perturbation \(\widetilde\gamma_{t^2\varepsilon}\), with the convention \(\widetilde\gamma_0\defeq\overline\gamma\). In other words, by \eqref{eq: recovery sequence with wrong marginals 1}, for every \(\phi\in C^1(M)\),
\begin{equation}
\label{eq: interpolation-global}
\int_M\phi\dd\mu_{t,\varepsilon}
=\int_\Sigma\int_{\ball_R}\phi\bigl(\proj_M(\Phi(z,t\sqrt\varepsilon s))\bigr)K(z,s)\dd s\dd\overline\gamma(z),
\end{equation}
so that \(\mu_{0,\varepsilon}=\mu\) and \(\mu_{1,\varepsilon}=\widetilde\mu_\varepsilon\).
 
To compute \(\mu_{t,\varepsilon}\), we use the charts \(\alpha_j\). For \(j\in J\) and \(x\in V_j\), we write
\[
A_j(x)\defeq A(\alpha_j(x)),
\qquad
K_j(x,s)\defeq K(\alpha_j(x),s).
\]
Inserting \(\sum_{j\in J}\chi_j=1\) in \eqref{eq: interpolation-global}, then changing variables \(z=\alpha_j(x)\) in each term by the area formula, and using \(\proj_M(\Phi(\alpha_j(x),t\sqrt\varepsilon s))=x+t\sqrt\varepsilon A_j(x)s\), we get
\begin{align*}
\int_M\phi\dd\mu_{t,\varepsilon}
&=\sum_{j\in J}\int_{U_j}\int_{\ball_R}\chi_j(z)\,\phi\bigl(\proj_M(\Phi(z,t\sqrt\varepsilon s))\bigr)K(z,s)\dd s\,\overline\rho(z)\dd\mathcal H^d(z)
\\
&=\sum_{j\in J}\int_{V_j}\int_{\ball_R}\phi\bigl(x+t\sqrt\varepsilon A_j(x)s\bigr)K_j(x,s)\dd s
\\
&\hspace{4cm}\times\chi_j(\alpha_j(x))\,\overline\rho(\alpha_j(x))\,\mathcal J_{\alpha_j}(x)\dd\vol_M(x)
\\
&=\sum_{j\in J}\int_{V_j}\int_{\ball_R}\phi\bigl(\zeta^j_{\varepsilon,s,t}(x)\bigr)\,\rho_j(x)K_j(x,s)\dd s\dd\vol_M(x),
\end{align*}
where, for \(s\in\ball_R\), \(t\in[0,1]\) and \(x\in V_j\), we have set
\[
\zeta_{\varepsilon,s,t}^j(x) \defeq x + t \sqrt{\varepsilon }A_j(x)s,
\qquad
\rho_j(x)\defeq\chi_j(\alpha_j(x))\,\overline\rho(\alpha_j(x))\,\mathcal J_{\alpha_j}(x).
\]
Here \(\zeta^j_{\varepsilon,s,t}(x)\) is the projection onto \(M\) of the point \(\Phi(\alpha_j(x),t\sqrt\varepsilon s)\), obtained by moving the point \(\alpha_j(x)\in\Sigma\) lying above \(x\) in the normal direction (see \Cref{fig: normal-perturbation}), and \(\rho_j\) is the density, with respect to \(\vol_M\), of the partial marginal \((\proj_M)_\#(\chi_j\overline\gamma)\), concentrated on \(V_j\). Hence
\begin{equation}
\label{eq: interpolation-charts}
\mu_{t,\varepsilon} = \sum_{j\in J} \mu_{t,\varepsilon}^j,
\qquad
\mu_{t,\varepsilon}^j\defeq\bigl((s,x)\mapsto\zeta^j_{\varepsilon,s,t}(x)\bigr)_\#\bigl(\rho_j(x)K_j(x,s)\dd s\dd\vol_M(x)\bigr).
\end{equation}
Equivalently, \(\mu_{t,\varepsilon}^j=\int_{\ball_R}(\zeta_{\varepsilon,s,t}^j)_{\#}(\rho_jK_j(\cdot,s)\vol_M)\dd s\), which is the form used below.
 
\textbf{Step 2. Construction of the flux.} For every \(\phi\in C^1(M)\), by \eqref{eq: interpolation-charts},
\begin{align*}
\frac{\dd}{\dd t}\int_M \phi\dd\mu_{t,\varepsilon}
&=\frac{\dd}{\dd t}\sum_{j\in J}\int_{V_j} \int_{\ball_R}\phi(\zeta_{\varepsilon,s,t}^j(x))\rho_j(x)K_j(x,s) \dd s \dd \vol_M(x) 
\\
&=\sum_{j\in J}\int_{V_j} \int_{\ball_R} \big\langle \nabla \phi(\zeta_{\varepsilon,s,t}^j(x)), \partial_t\zeta_{\varepsilon,s,t}^j (x)\big\rangle \rho_j(x)K_j(x,s) \dd s \dd \vol_M(x)
\\
&=\sum_{j\in J}\int_{V_j} \int_{\ball_R} \sqrt{\varepsilon}\big\langle \nabla \phi(\zeta_{\varepsilon,s,t}^j(x)), A_j(x)s  \big\rangle \rho_j(x)K_j(x,s) \dd s \dd \vol_M(x)
\\
&=
\int_M
\left\langle \nabla \phi,\dd\flow\right\rangle,
\end{align*}
where the differentiation under the integral sign is justified since \(\nabla\phi\) is bounded and \(|\partial_t\zeta^j_{\varepsilon,s,t}(x)|=\sqrt\varepsilon|A_j(x)s|\leq\sqrt\varepsilon R\), and where the flux \(\flow\in \mathcal M(M;TM)\) is the vector measure defined by
\[\flow \defeq \sqrt{\varepsilon}\sum_{j\in J} \int_{\ball_R} (\zeta_{\varepsilon,s,t}^j)_{\#}\left(\rho_jK_j(\cdot,s)A_js\, \vol_M\right) \dd s. \]
Integrating in time against a test function in
\(C_c^1((0,1))\) yields
\[
\partial_t\mu_{t,\varepsilon}
+
\operatorname{div}\flow
=0
\qquad\text{in }\mathcal D'((0,1)\times M).
\]

\textbf{Step 3. Gaussian symmetry.} As in \Cref{subsubsec: 1D}, the flux depends linearly on the densities \(\rho_j\). For \(j\in J\), \(s\in\ball_R\) and \(f\in L^1(V_j)\), let
\[
F_{j,s}[f](x) \defeq
\begin{cases}
f(x)K_j(x,s)A_j(x)s,
& x\in V_j,
\\[1mm]
0,
& x\notin V_j,
\end{cases}
\]
so that, by definition of the flux,
\[
\frac{\flow}{\sqrt{\varepsilon}}
= \sum_{j\in J}
\int_{\ball_R}
(\zeta_{\varepsilon,s,t}^j)_{\#}
\big(
F_{j,s}[\rho_j]\vol_M
\big)
\dd s.
\]
By \eqref{eq: symmetry of the Gaussian}, \(\int_{\ball_R}F_{j,s}[\rho_j]\dd s=0\). Subtracting this in each term, we obtain
\begin{equation}
\label{eq: flow-decomposition}
\frac{\flow}{\sqrt{\varepsilon}}
=\sum_{j\in J}
\int_{\ball_R}
\left((\zeta_{\varepsilon,s,t}^j)_{\#}-\id\right)\left(
F_{j,s}[\rho_j]\vol_M
\right)
\dd s
=\sum_{j\in J}\mathbf{I}_{t,\varepsilon}^j[\rho_j],
\end{equation}
where, for \(f\in L^1(V_j)\),
\[
\mathbf{I}_{t,\varepsilon}^j[f]
\defeq
\int_{\ball_R}
\left((\zeta_{\varepsilon,s,t}^j)_{\#}-\id\right)\left(
F_{j,s}[f]\vol_M
\right)
\dd s.
\]
 
\textbf{Step 4. Estimate of the displacement term.} Fix \(j\in J\). We estimate the total variation of \(\mathbf{I}_{t,\varepsilon}^j[f]\) in two ways.
 
\emph{(a) For \(f\in L^1(V_j)\).} Since the push-forward operation does not increase the total variation and \(|A(z)s|\leq|n(z)s|=|s|\), \eqref{eq: gaussian-moments} gives
\begin{align*}
\|\mathbf{I}_{t,\varepsilon}^j[f]\|_{\mathrm{TV}}
&\leq 2\int_{\ball_R}\|F_{j,s}[f]\|_{L^1(V_j)}\dd s
\\
&\leq 2\int_{V_j}|f(x)|\int_{\ball_R}|s|K_j(x,s)\dd s\dd\vol_M(x)
\leq 2M_1\|f\|_{L^1(V_j)}.
\end{align*}
 
\emph{(b) For \(f\in C_c^1(V_j)\) supported in a fixed compact set \(S\subseteq V_j\).} Write \(F_s\defeq F_{j,s}[f]\). Fix \(t\in[0,1]\) and \(s\in\ball_R\), and write \(\zeta\defeq\zeta_{\varepsilon,s,t}^j\), \(\kappa\defeq\zeta^{-1}\) and \(W\defeq\zeta(V_j)\). Since \(|\zeta(x)-x|\leq\sqrt\varepsilon R\),
we may choose \(\varepsilon>0\) sufficiently small, depending only on \(S\) (and not on \(f\), \(s\) and \(t\)), so that
\[
S\Subset W
\qquad\text{and}\qquad
\zeta(S)\Subset V_j.
\]
By the area formula, the density of \(\zeta_{\#}
(F_s\vol_M)\) is
\(F_s\circ\kappa\, \mathcal J_{\kappa}\mathbf{1}_{W}\), and, since \(\spt(F_s)\subseteq S\), the above inclusions show that \(F_s\circ\kappa\,\mathbf 1_{W}\) and \(F_s\) are both supported in \(W\cap V_j\). Hence
\begin{align*}
\big\|
\zeta_{\#}
(F_s\vol_M)
-
F_s\vol_M
\big\|_{\mathrm{TV}}
&\leq\|\mathcal J_{\kappa}-1\|_{L^\infty(W)}
\int_{W}
|F_s\circ\kappa|
\dd \vol_M
\\
&\quad+\int_{W\cap V_j }
|F_s\circ\kappa-F_s|\dd\vol_M.
\end{align*}
By \Cref{lemma: C^1 diffeo}, applied with \(\tau\defeq t\sqrt\varepsilon|s|L_j\) and \(f\defeq A_js/(|s|L_j)\), where \(L_j\) is a Lipschitz constant of \(A_j\) on the convex set \(V_j\) (recall that \(n\) is Lipschitz and \(\alpha_j\) is \(C^{1,1}\)), we have, for \(\varepsilon\) small enough, \(\|\mathcal J_{\kappa}-1\|_{L^\infty(W)}+\|\mathcal J_{\zeta}-1\|_{L^\infty(V_j)}\leq C\sqrt{\varepsilon}|s|\), and in particular \(1/2\leq\mathcal J_\kappa,\mathcal J_\zeta\leq2\) a.e. Since, by the area formula (change of variables \(y=\kappa(x)\)),
\[
\int_W|F_s\circ\kappa|\dd\vol_M
\leq2\int_W|F_s\circ\kappa|\,\mathcal J_\kappa\dd\vol_M
=2\|F_s\|_{L^1(V_j)},
\]
the first term is at most \(C\sqrt\varepsilon|s|\,\|F_s\|_{L^1(V_j)}\). For the second term, recall that \(F_s\) is defined on \(M\) and vanishes outside \(V_j\). By the area formula (change of variables \(x=\zeta(y)\)),
\[
\int_{W\cap V_j}|F_s\circ\kappa-F_s|\dd\vol_M
\leq2\int_{V_j}|F_s\circ\zeta-F_s|\dd\vol_M
\leq4\sqrt\varepsilon|s|\,\|\nabla F_s\|_{L^1(V_j)}.
\]
To prove the last inequality, assume first that \(F_s\) is \(C^1\). Since \(u\mapsto\zeta^j_{\varepsilon,s,ut}(y)=y+ut\sqrt\varepsilon A_j(y)s\) parametrizes the segment from \(y\) to \(\zeta(y)\) and \(|A_j(y)s|\leq|s|\), we have
\[
|F_s(\zeta(y))-F_s(y)|\leq\sqrt\varepsilon|s|\int_0^1\bigl|\nabla F_s\bigl(\zeta^j_{\varepsilon,s,ut}(y)\bigr)\bigr|\dd u,
\]
and we conclude by the area formula (change of variables \(x=\zeta^j_{\varepsilon,s,ut}(y)\)), since the inverse of \(\zeta^j_{\varepsilon,s,ut}\) has Jacobian at most \(2\) by \Cref{lemma: C^1 diffeo} (applied with \(ut\) in place of \(t\)). In general, \(F_s\) is only Lipschitz on \(M\), since \(A_j\) and \(K_j(\cdot,s)\) are Lipschitz on \(V_j\) (see below) and \(f\) has compact support in \(V_j\). We then apply the above to mollifications of \(F_s\), which converge to \(F_s\) in \(W^{1,1}(M)\), and pass to the limit, using that \(\int_{V_j}|g\circ\zeta|\dd\vol_M\leq2\|g\|_{L^1(M)}\) for \(g\in L^1(M)\), since \(\mathcal J_\kappa\leq2\). Altogether,
\[
\big\|
\zeta_{\#}
(F_s\vol_M)
-
F_s\vol_M
\big\|_{\mathrm{TV}}
\leq
C\sqrt{\varepsilon}|s|
\left(
\|F_s\|_{L^1(V_j)}
+
\|\nabla F_s\|_{L^1(V_j)}
\right).
\]
It remains to bound the \(L^1\) norms of \(F_s\) and \(\nabla F_s\). By \eqref{eq: normal-ellipticity} and since \(Z_R\geq c\), we have \(K_j(x,s)\leq c^{-1}e^{-\lambda|s|^2/2}\). Moreover, \(x\mapsto\Hn{\alpha_j(x)}\) is Lipschitz (since \(E\in C^{2,1}\), \(n\) is Lipschitz and \(\alpha_j\) is \(C^{1,1}\)), hence so are \(x\mapsto Z_R(\alpha_j(x))\) and \(K_j(\cdot,s)\), and differentiating a.e.\
\(K_j(x,s)=\mathbf1_{\ball_R}(s)\,e^{-\Hn{\alpha_j(x)}(s,s)/2}/Z_R(\alpha_j(x))\) with respect to \(x\) gives \(|\nabla_xK_j(x,s)|\leq C(1+|s|^2)e^{-\lambda|s|^2/2}\) for a.e.\ \(x\in V_j\). Since \(F_s=f\,K_j(\cdot,s)\,A_js\), with \(|A_js|\leq|s|\) and \(A_j\) Lipschitz, the product rule yields
\[
\|F_s\|_{L^1(V_j)}
\leq
C e^{-\lambda|s|^2/2}
|s|\|f\|_{L^1(V_j)},
\]
\[
\|\nabla F_s\|_{L^1(V_j)}
\leq
C e^{-\lambda|s|^2/2}
\left[
|s|\|\nabla f\|_{L^1(V_j)}
+
(|s|+|s|^3)\|f\|_{L^1(V_j)}
\right].
\]
Integrating with respect to \(s\in\ball_R\), we obtain, uniformly in \(t\in[0,1]\),
\[
\|\mathbf{I}_{t,\varepsilon}^j[f]\|_{\mathrm{TV}}
\leq
C\sqrt{\varepsilon} (M_2+M_4)
\left(
\|f\|_{L^1(V_j)}
+
\|\nabla f\|_{L^1(V_j)}
\right),
\]
where \(M_2\) and \(M_4\) are defined in \eqref{eq: gaussian-moments}.
 
\emph{Conclusion.} Since \(\chi_j\) has compact support in \(U_j\), \(\rho_j\) vanishes outside the compact set \(\proj_M(\spt\chi_j)\subseteq V_j\). By mollification, there exist a compact set \(S_j\subseteq V_j\) and a sequence \((\rho_j^{(k)})_{k\in\N}\) in \(C_c^1(V_j)\), all supported in \(S_j\), such that
\(\|\rho_j-\rho_j^{(k)}\|_{L^1(V_j)}\to0\) as \(k\to\infty\). Applying (b) with \(S=S_j\), there exists \(\varepsilon_0>0\), independent of \(k\), such that, for \(0<\varepsilon\leq\varepsilon_0\) and every \(k\in\N\), writing \(\mathbf{I}_{t,\varepsilon}^j[\rho_j]=\mathbf{I}_{t,\varepsilon}^j[\rho_j-\rho_j^{(k)}]+\mathbf{I}_{t,\varepsilon}^j[\rho_j^{(k)}]\), (a) and (b) give
\[
\sup_{t\in[0,1]}\|\mathbf{I}_{t,\varepsilon}^j[\rho_j]\|_{\mathrm{TV}}
\leq 2M_1\|\rho_j-\rho_j^{(k)}\|_{L^1(V_j)}
+
C\sqrt{\varepsilon}
\left(
\|\rho_j^{(k)}\|_{L^1(V_j)}
+
\|\nabla \rho_j^{(k)}\|_{L^1(V_j)}
\right).
\]
For \(r>0\), set
\[
\omega_j(r)
\defeq
\inf_{k\geq1}
\left\{
2M_1\|\rho_j-\rho_j^{(k)}\|_{L^1(V_j)}
+
rC
\left(
\|\rho_j^{(k)}\|_{L^1(V_j)}
+
\|\nabla \rho_j^{(k)}\|_{L^1(V_j)}
\right)
\right\},
\]
and \(\omega(r)\defeq\sum_{j\in J}\omega_j(r)\).
Since \(\rho_j^{(k)}\to\rho_j\) in \(L^1(V_j)\), we have
\(\omega(r)\to0\) as \(r\to 0\): one first fixes \(k\) sufficiently large so that the first
term is arbitrarily small, and then lets \(r\to0\).
Combining this with \eqref{eq: flow-decomposition} yields, for \(0<\varepsilon\leq\varepsilon_0\),
\[
\sup_{t\in[0,1]}
\frac{\|\flow\|_{\mathrm{TV}}}
{\sqrt{\varepsilon}}
\leq \sum_{j\in J} \sup_{t\in[0,1]} \|\mathbf{I}_{t,\varepsilon}^j[\rho_j]\|_{\mathrm{TV}} \leq
\omega(\sqrt{\varepsilon})
=o(1).
\]
 
\textbf{Step 5. Benamou--Brenier estimate.} Recall from Step~2 and \eqref{eq: interpolation-charts} that
\begin{align*}
\flow &= \sqrt{\varepsilon}\sum_{j\in J} \int_{\ball_R} (\zeta_{\varepsilon,s,t}^j)_{\#}\left(\rho_jK_j(\cdot,s)A_js\, \vol_M\right) \dd s,
\\
\mu_{t,\varepsilon}&=\sum_{j\in J}\int_{\ball_R}(\zeta_{\varepsilon,s,t}^j)_{\#}(\rho_jK_j(\cdot,s)\vol_M)\dd s.
\end{align*}
Since \(|A_j(x)s|\leq|n(\alpha_j(x))s|=|s|\leq R\), we get \(|\flow|\leq\sqrt\varepsilon R\,\mu_{t,\varepsilon}\). Thus \(\flow\ll\mu_{t,\varepsilon}\), and the velocity field
\[
v_{t,\varepsilon}
\defeq\frac{\dd \flow}{\dd\mu_{t,\varepsilon}}\in L^\infty(\mu_{t,\varepsilon};TM)
\quad\text{satisfies}\quad |v_{t,\varepsilon}|\le\sqrt\varepsilon R
\quad\mu_{t,\varepsilon}\text{-a.e.}
\]
Since \((\mu_{t,\varepsilon},v_{t,\varepsilon})\) satisfies the
continuity equation, with \(\mu_{0,\varepsilon}=\mu\) and \(\mu_{1,\varepsilon}=\widetilde\mu_\varepsilon\), the Benamou--Brenier formula (see, e.g., \cite{santambrogioOptimalTransportApplied2015}) yields
\begin{align*}
W_2^2(\mu,\widetilde\mu_\varepsilon)
&\leq\int_0^1\int_M |v_{t,\varepsilon}|^2
\,\dd\mu_{t,\varepsilon}\dd t=\int_0^1\int_M |v_{t,\varepsilon}|\,\dd|\flow|\dd t
\\
&\le\sqrt\varepsilon R\sup_{t\in[0,1]}
\|\flow\|_{\mathrm{TV}}.
\end{align*}
 
By Step~4, \(\sup_{t\in[0,1]}\|\flow\|_{\mathrm{TV}}\leq\sqrt\varepsilon\,\omega(\sqrt\varepsilon)\) for \(0<\varepsilon\leq\varepsilon_0\), hence
\[
W_2^2(\mu,\widetilde\mu_\varepsilon)
\leq
\varepsilon R\,\omega(\sqrt{\varepsilon}).
\]
Since \(\omega(r)\to0\) as \(r\to0\), this gives \(W_2(\mu,\widetilde\mu_\varepsilon)=o_R(\sqrt{\varepsilon})\).
\end{proof}

\section{Extension to generalized framework and examples}
\label{sec: extensions}
In this section, we extend \hyperref[thm: A]{Theorem~A} to compactly supported marginals on \(\R^d\) (\Cref{subsec: euclidean-extension}), where we also prove \Cref{thm: convergence-Gibbs-measure}, and to \(N\geq2\) marginals (\Cref{subsec: multimarginal-extension}). We illustrate these extensions with the quadratic cost, the non-twisted cost of \cite[Example~3.1]{mccann2012rectifiability}, and a repulsive harmonic cost with three marginals on the circle.

\subsection{Compactly supported marginals on \texorpdfstring{\(\R^d\)}{Rd}}
\label{subsec: euclidean-extension}
In accordance with \Cref{sec: framework}, where \(M_1\) and \(M_2\) are compact, we take for \(M_i\) compact convex subsets of \(\R^d\) containing the supports of the marginals in their interior, and we replace \ref{H1}--\ref{H4} by the following assumptions.
\begin{enumerate}[(H1')]
\item\label{H1'} For \(i=1,2\), \(M_i\subset\R^d\) is a compact convex set, and \(\Omega=M_1\times M_2\). The cost \(c\) is continuous on \(\Omega\), and we fix continuous Kantorovich potentials \(\overline\varphi_1,\overline\varphi_2\). The cost is \(C^2\) on \(\interior\Omega\) and non-degenerate on an open neighborhood \(U\subset\interior\Omega\) of \(\Sigma\cap\interior\Omega\): \(D^2_{x_1x_2}c(z)\) is invertible for every \(z\in U\).
\item\label{H2'} The duality gap \(E=c-\overline\varphi_1\oplus\overline\varphi_2\geq0\) belongs to \(C^{2,1}(\interior\Omega)\); in particular, \(\hess E\) is bounded on \(\interior\Omega\).
\item\label{H3'} The set \(\Sigma\cap\interior\Omega\) is a nonempty \(C^{1,1}\) submanifold of \(\interior\Omega\) without boundary, of dimension \(d\). As before, \(k\defeq d\) denotes its codimension.
\item\label{H4'} For \(i=1,2\), \(\mu_i=\rho_i\dd x\), where \(\rho_i\) vanishes outside a compact set \(K_i\subset\interior M_i\), is Lipschitz on \(K_i\), and satisfies \(0<\alpha\leq\rho_i\leq\beta\) on \(K_i\). Moreover, setting \(\mathcal K\defeq K_1\times K_2\), we have \(\mathcal H^d(\Sigma\cap\partial\mathcal K)=0\).
\end{enumerate}
We set \(\rho(x_1,x_2)\defeq\rho_1(x_1)\rho_2(x_2)\), so that \(\mu^\otimes=\mu_1\otimes\mu_2=\rho\dd x=\rho\mathbf 1_{\mathcal K}\dd x\). Compared with the torus, the new features are that \(\Sigma\) may extend beyond \(\mathcal K\) and cross its boundary, and that \(\mu^\otimes\) is discontinuous across \(\partial\mathcal K\).

\subsubsection{Convergence of the Gibbs measure and the geometric Laplace's method}
\label{subsec: Convergence of the Gibbs measure and the geometric Laplace's method}

This section is devoted to the proof of the convergence of the Gibbs measure in the case of compactly supported marginals on \(\R^d\). The result slightly extends the setting considered in \cite{hwang1980laplace}, since the metric and tubular neighborhoods of the contact set need not coincide. The same proof, with \(\mathcal K=\omega=\Omega\) (so that \(\mathbf 1_{\mathcal K}\equiv1\), \(\Sigma\cap\omega=\Sigma\), and \Cref{lemma: tubular vs regular neighborhood} follows from \Cref{subsec:structure}), gives \Cref{thm: convergence-Gibbs-measure} on the torus.

\begin{theorem}[Convergence of the Gibbs measure]
\label{thm: euclidean-Gibbs}
Under assumptions \ref{H1'}--\ref{H4'},
    \begin{equation}
        \gibbs \xrightharpoonup[]{\ast}\mathbb{P}_0,
    \end{equation}
where \[\mathbb P_0 =\frac{1}{Z_0}
 (\det \hess^\perp E)^{-1/2}
 \rho\mathbf{1}_{\mathcal K}\mathcal H^d_{|\Sigma} \quad  \text{and} \quad Z_0=\int_{\Sigma\cap \mathcal K}\frac{\rho_1(z_1)\rho_2(z_2)}
 {\sqrt{\det \hess^\perp E(z)}}\,\dd\mathcal H^d(z). \]
Moreover, \(Z_\varepsilon=(2\pi\varepsilon)^{k/2}(Z_0+o(1))\).
\end{theorem}

Before turning to the proof, we first establish a preliminary result. We fix an open set \(\omega\) such that \(\mathcal K\Subset\omega\Subset\interior\Omega\). Since \(\Sigma\cap\overline\omega\) is a compact subset of the \(C^{1,1}\) submanifold \(\Sigma\cap\interior\Omega\), the constructions of \Cref{subsec:structure} and the estimates of \Cref{prop: twist-geometry,lemma: quadratic-detachment} hold on \(\Sigma\cap\overline\omega\), with uniform constants. For \(r>0\) small enough, we denote by
\[
\mathcal T_r\defeq\Phi((\Sigma\cap\omega)\times\ball_r)=\{z+v:z\in\Sigma\cap\omega,\ v\in N_z\Sigma,\ |v|<r\}
\]
the tubular neighborhood of radius \(r\) of \(\Sigma\cap\omega\), not to be confused with the \(r\)-neighborhood \(\Sigma_r=\{\x\in\Omega:\dist(\x,\Sigma)<r\}\) of \(\Sigma\).

\begin{lemma}
\label{lemma: tubular vs regular neighborhood}
Under assumptions \ref{H1'}--\ref{H4'}, there exists \(r_0>0\) such that, for every \(0<r\leq r_0\),
\[
\mathcal K\cap\Sigma_r\subset\mathcal T_r.
\]
Consequently,
\[
x\in\mathcal K\setminus\mathcal T_r
\quad\Longrightarrow\quad
\dist(x,\Sigma)\geq r.
\]
\end{lemma}

\begin{proof}
Set \(\eta\defeq\dist(\mathcal K,\R^{2d}\setminus\omega)>0\). Since \(\Sigma\) is a \(C^{1,1}\) submanifold, it has locally positive reach. By compactness of \(\Sigma\cap\overline\omega\), there exists \(r_0\in(0,\eta)\) such that the nearest-point projection \(\proj_\Sigma\) is uniquely defined on \(\mathcal K\cap\Sigma_{r_0}\). Let \(0<r\leq r_0\) and \(\x\in\mathcal K\cap\Sigma_r\). Then \(|\x-\proj_\Sigma(\x)|=\dist(\x,\Sigma)<r\). Since \(\x\in\mathcal K\) and \(r_0<\eta\), it follows that \(\proj_\Sigma(\x)\in\omega\). Moreover, \(\x-\proj_\Sigma(\x) \in N_{\proj_\Sigma(\x)}\Sigma\). Hence \(\x=\proj_\Sigma(\x)+\big(\x-\proj_\Sigma(\x)\big)\in\mathcal T_r\).
\end{proof}

The following figure illustrates the geometric reason for enlarging the base of the tubular neighborhood from \(\Sigma\cap\mathcal K\) to \(\Sigma\cap\omega\). This enlargement ensures that every point of \(\mathcal K\) lying outside the tubular neighborhood is uniformly separated from the whole contact set:
\[
x\in\mathcal K\setminus\mathcal T_r
\quad\Longrightarrow\quad
\dist(x,\Sigma)\geq r.
\]
This is what allows us to decompose the integrals in the proof below.

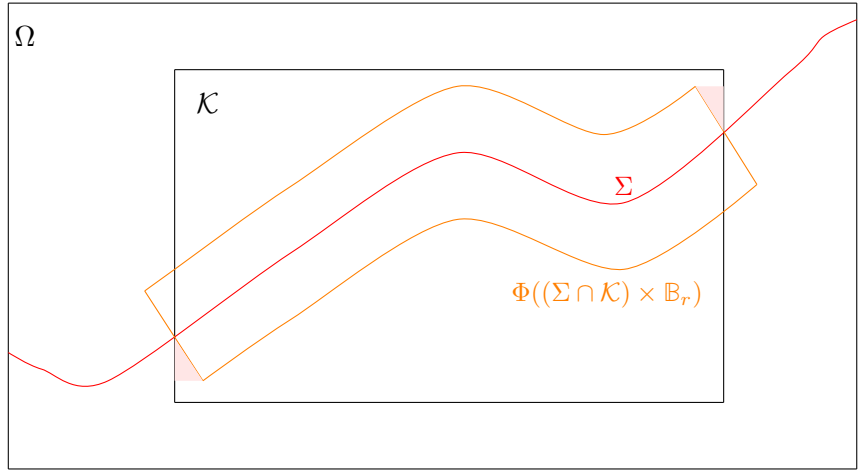
\begin{center}
\begin{tikzpicture}[scale=2.2]
  \draw[-] (0.3,0) -- (3.6,0) {};
  \draw[-] (0.3,0) -- (0.3,2) {};
  \draw[-] (0.3,2) -- (3.6,2) {};
  \draw[-] (3.6,2) -- (3.6,0) {};

  \draw[-] (-0.7,-0.4) -- (4.4,-0.4) {};
  \draw[-] (-0.7,-0.4) -- (-0.7,2.4) {};
  \draw[-] (-0.7,2.4) -- (4.4,2.4) {};
  \draw[-] (4.4,-0.4) -- (4.4,2.4) {};
  
\draw[thin, red , smooth]
    plot coordinates {(-0.7, 0.3) (-0.5,0.2) (-0.1,0.13)  (1,0.9) (2,1.5) (3,1.2) (4,1.99) (4.2,2.2) (4.4,2.3)};

\draw[thin, orange , smooth]
    plot coordinates { (0.12,0.67) (0.3,0.8)   (1,1.3) (2,1.9) (2.88,1.61) (3.43,1.9) };

\draw[thin, orange , smooth]
    plot coordinates { (0.47,0.13)  (1,0.5) (2,1.1) (3,0.8) (3.8,1.31) };

\draw[thin, orange , smooth]
    plot coordinates { (0.47,0.13) (0.12,0.67) };
\draw[thin, orange , smooth]
    plot coordinates { (3.8,1.31) (3.43,1.9) };

\fill[red!10] (0.3,0.13) -- (0.47,0.13) -- (0.3,0.385) -- cycle;
\fill[red!10] (3.43,1.9) -- (3.6,1.9) -- (3.6,1.63) -- cycle;

  \node[thin, red] at (3,1.3) {\small \(\Sigma\)};
  \node[thin, orange] at (2.9,0.65) {\small \(\Phi((\Sigma\cap\mathcal{K})\times\ball_{r})\)};
  \node[] at (0.5,1.8) {\(\mathcal{K}\)};
  \node[] at (-0.6,2.2) {\(\Omega\)};
\end{tikzpicture}
\end{center}

If the tubular neighborhood were constructed only around \(\Sigma\cap\mathcal K\), the shaded regions could contain points of \(\mathcal K\) that remain close to portions of \(\Sigma\) lying outside \(\mathcal K\). Thus, the distance to \(\Sigma\) would not be controlled solely by lying outside \(\Phi((\Sigma\cap\mathcal K)\times\ball_r)\).

\begin{proof}
Let \(r_0\) be as in \Cref{lemma: tubular vs regular neighborhood}. We fix \(r\in(0,r_0]\) such that the bounds of \Cref{lemma: quadratic-detachment} hold on \(\mathcal T_r\), and \(R>0\) so that \(R\sqrt\varepsilon<r\) for \(\varepsilon\) small enough. Let \(\phi\in C(\Omega)\) and introduce the unnormalized Gibbs measure
\[\dd \unnormalizedgibbs= (2\pi\varepsilon)^{-k/2} \exp\left(-\frac{E}{\varepsilon}\right)\dd \mu^\otimes,\]
so that \(\gibbs=\unnormalizedgibbs/\unnormalizedgibbs(\Omega)\). We decompose \(\Omega\) as
\begin{align*}
\int_\Omega \phi\dd\unnormalizedgibbs
&=
\int_{\Sigma_{R\sqrt{\varepsilon}}}\phi\dd\unnormalizedgibbs
+
\int_{\Sigma_{r}\setminus \Sigma_{R\sqrt{\varepsilon}}}
\phi\dd\unnormalizedgibbs
+
\int_{\Omega\setminus\Sigma_{r}}
\phi\dd\unnormalizedgibbs
\\
&=\int_{\mathcal T_{R\sqrt{\varepsilon}}}\phi\dd\unnormalizedgibbs
+
\int_{\mathcal T_{r}\setminus \mathcal T_{R\sqrt{\varepsilon}}}
\phi\dd\unnormalizedgibbs
+
\int_{\Omega\setminus \Sigma_{r}}
\phi\dd\unnormalizedgibbs
\\
&\defeq
I_{1,\varepsilon,R}(\phi)
+
I_{2,\varepsilon,R}(\phi)
+
I_{3,\varepsilon}(\phi),
\end{align*}
where, for the second equality, we used that \(\unnormalizedgibbs\) is supported on \(\mathcal K\) and that \(\mathcal K\cap\Sigma_\rho=\mathcal K\cap\mathcal T_\rho\) for \(0<\rho\leq r\), by \Cref{lemma: tubular vs regular neighborhood} and since \(\mathcal T_\rho\subset\Sigma_\rho\). We analyze these three terms separately.

\textbf{Step 1. Estimate on \(\mathcal T_{R\sqrt{\varepsilon}}\).} Using the tubular coordinates induced by \(\Phi\), we write \(E(z,s)\) for \(E(\Phi(z,s))\) and define
\[
a(z,s)
\defeq
\phi\bigl(\Phi(z,s)\bigr)
\rho(\Phi(z,s))
\mathcal J_{\Phi}(z,s)\mathbf{1}_{\mathcal K}(\Phi(z,s)),
\]
for \(\mathcal H^{d}\otimes \mathcal L^k\)-a.e. \((z,s)\in(\Sigma\cap\omega)\times \ball_{r}\). Note that \(|a|\leq\beta^2\|\phi\|_\infty C_J\), where \(C_J\) bounds \(\mathcal J_\Phi\) on \((\Sigma\cap\overline\omega)\times\ball_r\). By the area formula \eqref{eq: tube-area-formula},
\[
I_{1,\varepsilon,R}(\phi)
=
(2\pi\varepsilon)^{-k/2}
\int_{\Sigma\cap\omega}
\int_{\ball_{R\sqrt\varepsilon}}
e^{-E(z,s)/\varepsilon}
a(z,s)\dd s \dd \mathcal H^d_{|\Sigma}(z).
\]
Since \(E(z)=0\) and \(\nabla E(z)=0\) for \(z\in\Sigma\), and since \(E\in C^{2,1}\), we have
\[ E(z,s) = \frac12 \Hn{z}(s,s) + \mathcal R(z,s),
\qquad
|\mathcal R(z,s)|\leq C|s|^3 \quad \text{for all } (z,s)\in (\Sigma\cap \overline\omega)\times \ball_r.
\]
Performing the change of variables \(s=\sqrt{\varepsilon}\,\eta\), \(\dd s=\varepsilon^{k/2}\dd \eta\), and then relabeling the normal variable, we obtain
\begin{equation}
\label{eq: I1-after-rescaling}
I_{1,\varepsilon,R}(\phi)
=(2\pi)^{-k/2}
\int_{\Sigma\cap\omega}
\int_{\ball_R}
e^{
-\frac{1}{2}
\Hn{z}(s,s)
-\frac{1}{\varepsilon}\mathcal R(z,\sqrt{\varepsilon }s)}
a(z,\sqrt{\varepsilon }s)\dd s\dd \mathcal H^d_{|\Sigma}(z).
\end{equation}
As \(\varepsilon\to0\), we have \(|\mathcal R(z,\sqrt\varepsilon s)|/\varepsilon\leq C\sqrt\varepsilon|s|^3\to0\) and \(\mathcal J_\Phi(z,\sqrt\varepsilon s)\to\mathcal J_\Phi(z,0)=1\) by \Cref{lemma: determinant of Phi}. Moreover, \(\Phi(z,\sqrt\varepsilon s)\to z\) and \(\mathbf 1_{\mathcal K}\) is locally constant on \(\R^{2d}\setminus\partial\mathcal K\), so that, \(\rho\) being continuous on \(\mathcal K\), \((\rho\mathbf 1_{\mathcal K})(\Phi(z,\sqrt\varepsilon s))\to(\rho\mathbf 1_{\mathcal K})(z)\) for every \(z\in\Sigma\setminus\partial\mathcal K\); since \(\mathcal H^d(\Sigma\cap\partial\mathcal K)=0\), we get \(a(z,\sqrt\varepsilon s)\to\phi(z)\rho(z)\mathbf 1_{\mathcal K}(z)\) for \(\mathcal H^d\otimes\mathcal L^k\)-a.e.\ \((z,s)\in(\Sigma\cap\omega)\times\ball_R\). Since the integrand of \eqref{eq: I1-after-rescaling} is bounded by \(e^{C\sqrt\varepsilon R^3}\beta^2\|\phi\|_\infty C_J\leq2\beta^2\|\phi\|_\infty C_J\) for \(\varepsilon\) small enough, and \(\mathcal H^d(\Sigma\cap\omega)\) is finite, the dominated convergence theorem gives
\[
\lim_{\varepsilon\to 0}I_{1,\varepsilon,R}(\phi)
=(2\pi)^{-k/2}\int_{\Sigma\cap\mathcal K}\phi(z)\rho(z)\int_{\ball_R} e^{-\Hn{z}(s,s)/2}\dd s\dd\mathcal H^d_{|\Sigma}(z).
\]
Next, since \((2\pi)^{-k/2}\int_{\R^k}e^{-\Hn{z}(s,s)/2}\dd s=\bigl(\det\Hn{z}\bigr)^{-1/2}\), the dominated convergence theorem gives
\[
\lim_{R\to +\infty} \lim_{\varepsilon \to 0}I_{1 ,\varepsilon , R}(\phi)=\int_{\Sigma\cap\mathcal K}\phi(z)\bigl(\det \Hn{z}\bigr)^{-1/2}\rho(z)\dd \mathcal H^d_{|\Sigma}(z).
\]

\textbf{Step 2. Estimate on \(\mathcal T_{r}\setminus  \mathcal T_{R\sqrt{\varepsilon}}\).} As in Step~1, the area formula and the change of variables \(s=\sqrt\varepsilon\,\eta\) give
\[
\big|I_{2,\varepsilon,R}(\phi)\big|
\leq(2\pi)^{-k/2}\beta^2\|\phi\|_\infty C_J
\int_{\Sigma\cap\omega}\int_{R\leq|s|<r/\sqrt\varepsilon}
e^{-E(z,\sqrt\varepsilon s)/\varepsilon}\dd s\dd\mathcal H^d_{|\Sigma}(z).
\]
By \Cref{lemma: quadratic-detachment} and \eqref{eq: normal-distance}, \(E(z,\sqrt\varepsilon s)\geq\underline{C}\dist\bigl(\Phi(z,\sqrt\varepsilon s),\Sigma\bigr)^2=\underline C\varepsilon|s|^2\) for \(|s|<r/\sqrt\varepsilon\). Hence, uniformly in \(\varepsilon\),
\[
\big|I_{2,\varepsilon,R}(\phi)\big|
\leq(2\pi)^{-k/2}\beta^2\|\phi\|_\infty C_J\,\mathcal H^d(\Sigma\cap\omega)\int_{\R^k\setminus\ball_R}e^{-\underline C|s|^2}\dd s
\xrightarrow[R\to+\infty]{}0,
\]
as the tail of an integrable Gaussian. This proves that the contribution of \(\mathcal T_r\setminus\mathcal T_{R\sqrt\varepsilon}\) is negligible.

\textbf{Step 3. Estimate on \(\Omega\setminus \Sigma_{r}\).}
If \(\mathcal K\setminus \Sigma_{r}= \emptyset\) there is nothing to do. Otherwise, \(\mathcal K\setminus\Sigma_r\) is a compact set which does not meet \(\Sigma\), so that \(E>0\) on it, and we denote
\[
\delta \defeq
\min_{\mathcal K\setminus\Sigma_r}E>0.
\]
Since \(\unnormalizedgibbs\) is supported on \(\mathcal K\) and \(\mu^\otimes(\mathcal K)=1\),
\[
\big|I_{3,\varepsilon}(\phi)\big|
\leq (2\pi\varepsilon)^{-k/2} \|\phi\|_{\infty}
e^{-\delta/\varepsilon}\xrightarrow[\varepsilon \to 0]{} 0.
\]

\textbf{Conclusion.} Set \(L(\phi)\defeq\int_{\Sigma\cap\mathcal K}\phi(z)\bigl(\det \Hn{z}\bigr)^{-1/2}\rho(z)\dd\mathcal H^d_{|\Sigma}(z)\). Since \(\int_\Omega\phi\dd\unnormalizedgibbs\) does not depend on \(R\), Steps~1--3 give, for every \(R>0\),
\[
\limsup_{\varepsilon\to0}\left|\int_\Omega\phi\dd\unnormalizedgibbs-L(\phi)\right|
\leq\left|\lim_{\varepsilon\to0}I_{1,\varepsilon,R}(\phi)-L(\phi)\right|+\sup_{\varepsilon}\big|I_{2,\varepsilon,R}(\phi)\big|,
\]
and the right-hand side tends to \(0\) as \(R\to+\infty\). In other words,
\[
\unnormalizedgibbs
\xrightharpoonup[]{\ast}
\bigl(\det \hess^\perp E\bigr)^{-1/2}
\rho\mathbf 1_{\mathcal K}
\mathcal H^d_{|\Sigma}.
\]
Taking \(\phi\equiv1\), we get \(\unnormalizedgibbs(\Omega)=(2\pi\varepsilon)^{-k/2}Z_\varepsilon\to Z_0\), which is the asymptotics of \(Z_\varepsilon\). Moreover, \(Z_0>0\): an optimal plan is concentrated on \(\Sigma\cap\mathcal K\), so that \(\mu_1\) is concentrated on \(\proj_{M_1}(\Sigma\cap\mathcal K)\), and since \(\mu_1\) is absolutely continuous and \(\proj_{M_1}\) is \(1\)-Lipschitz, \(0<\mathcal L^d(\proj_{M_1}(\Sigma\cap\mathcal K))\leq\mathcal H^d(\Sigma\cap\mathcal K)\). Thus
\[\gibbs = \frac{\unnormalizedgibbs}{\unnormalizedgibbs(\Omega)}
\xrightharpoonup[]{\ast}
\frac1{Z_0}
\bigl(\det \hess^\perp E\bigr)^{-1/2}
\rho\mathbf1_{\mathcal K}
\mathcal H^d_{|\Sigma},
\]
which proves the claim.
\end{proof}

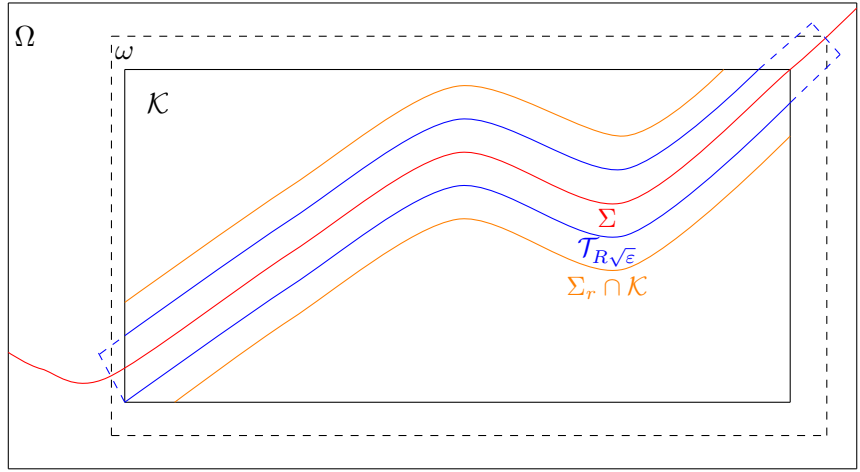
\begin{center}
\begin{tikzpicture}[scale=2.2]
  \draw[-] (0,0) -- (4,0) {};
  \draw[-] (0,0) -- (0,2) {};
  \draw[-] (0,2) -- (4,2) {};
  \draw[-] (4,2) -- (4,0) {};
  \draw[dashed] (-0.08,-0.2) -- (4.22,-0.2) {};
  \draw[dashed] (-0.08,-0.2) -- (-0.08,2.2) {};
  \draw[dashed] (-0.08,2.2) -- (4.22,2.2) {};
  \draw[dashed] (4.22,-0.2) -- (4.22,2.2) {};
  \draw[-] (-0.7,-0.4) -- (4.4,-0.4) {};
  \draw[-] (-0.7,-0.4) -- (-0.7,2.4) {};
  \draw[-] (-0.7,2.4) -- (4.4,2.4) {};
  \draw[-] (4.4,-0.4) -- (4.4,2.4) {};
\draw[thin, red , smooth]
    plot coordinates {(-0.7, 0.3) (-0.5,0.2) (-0.1,0.15)  (1,0.9) (2,1.5) (3,1.2) (4,2)  (4.4,2.37)};
\draw[thin, blue , smooth]
    plot coordinates {(0,0.4)   (1,1.1) (2,1.7) (3,1.4) (3.81,2) };    
\draw[thin, blue , smooth]
    plot coordinates { (0,0)  (1,0.7) (2,1.3) (3,1) (4,1.8) };
\draw[thin, blue , smooth, dashed]
    plot coordinates {(-0.15,0.29) (0,0.4)};
\draw[thin, blue , smooth, dashed]
    plot coordinates { (3.81,2) (4.13,2.28)};
\draw[thin, blue , smooth, dashed]
    plot coordinates { (4,1.8) (4.3,2.09)};
\draw[thin, blue, smooth, dashed]
plot coordinates {(0,0) (-0.16,0.3)};
\draw[thin, blue, smooth, dashed]
plot coordinates {(4.3,2.09) (4.13,2.28)};
\draw[thin, orange , smooth]
    plot coordinates {(0,0.6)   (1,1.3) (2,1.9) (3,1.6) (3.6,2) };
\draw[thin, orange , smooth]
    plot coordinates { (0.3,0)  (1,0.5) (2,1.1) (3,0.8) (4,1.6) };
  \node[thin, red] at (2.9,1.11) {\small \(\Sigma\)};
  \node[thin, blue] at (2.9,0.9) {\small \(\mathcal T_{R\sqrt\varepsilon}\)};
  \node[thin, orange] at (2.9,0.69) {\small \(\Sigma_{r}\cap \mathcal K\)};
  \node[] at (0.2,1.8) {\(\mathcal{K}\)};
  \node[] at (0,2.1) {\(\omega\)};
  \node[] at (-0.6,2.2) {\(\Omega\)};
\end{tikzpicture}
\end{center}
The previous figure provides a schematic illustration of the different neighborhoods involved in the proof. The compact set \(\mathcal K\) is contained in the intermediate neighborhood \(\omega\), which is itself contained in \(\Omega\). The contact set \(\Sigma\) may extend beyond \(\mathcal K\), and the tubular neighborhoods considered in the proof are therefore constructed around the larger portion \(\Sigma\cap\omega\).

\subsubsection{First-order entropic selection}

\begin{theorem}[Compactly supported Euclidean marginals]
\label{thm: euclidean-selection}
Assume the hypotheses \ref{H1'}--\ref{H4'} are satisfied and define
\[\mathcal F^{(0)}(\gamma)\defeq\int_{\Omega}c\dd \gamma+\chi_{\Pi(\mu_1,\mu_2)}(\gamma),
\qquad
\mathcal F_{\varepsilon}(\gamma)=\int_{\Omega}c\dd \gamma+\varepsilon\entropy(\gamma|\mu_1\otimes\mu_2)+\chi_{\Pi(\mu_1,\mu_2)}(\gamma).\]
Then, with \(\mathbb P_0\) and \(Z_0\) given by \Cref{thm: euclidean-Gibbs},
\begin{equation}
\label{eq: gamma-expansion-R}
 \mathcal F_\varepsilon-\frac{k}{2}\varepsilon\log(1/\varepsilon)
 \stackrel{\Gamma}{=}
 \mathcal F^{(0)}+\varepsilon\Bigl(
 \entropy\left(\cdot\,|\,\mathbb P_0\right)+\chi_{\Pi(\mu_1,\mu_2)}(\cdot)
 -\log Z_0-\frac{k}{2}\log(2\pi)
 \Bigr)+o(\varepsilon).
\end{equation}
Moreover, the following asymptotics of the energy holds
\begin{equation}
    \label{eq: energy-expansion-R}
    \min\mathcal{F}_\varepsilon=\min \mathcal F^{(0)}+\dfrac{k}{2}\varepsilon\log(1/\varepsilon)+\varepsilon C_0+o(\varepsilon),
\end{equation}
where \(C_0\defeq  \min_{\gamma\in\Pi(\mu_1,\mu_2)}  \entropy(\gamma\,|\,\mathbb P_0)-\log Z_0-\frac{k}{2}\log(2\pi)\), and the minimizer of \(\mathcal F_\varepsilon\) converges weakly to the unique minimizer of \(\entropy(\cdot\,|\,\mathbb P_0)\) on \(\Pi(\mu_1,\mu_2)\).
\end{theorem}

\begin{proof}
We follow the proof of \hyperref[thm: A]{Theorem~A} given in \Cref{sec: selection}, and only indicate the changes. By compactness of \(\Sigma\cap\overline\omega\), the tubular coordinates, the normal ellipticity and the quadratic detachment hold on \(\Sigma\cap\omega\) with uniform constants (see \Cref{subsec: Convergence of the Gibbs measure and the geometric Laplace's method}), and \Cref{thm: euclidean-Gibbs} replaces \Cref{thm: convergence-Gibbs-measure}. In particular, the \(\Gamma\)-\(\liminf\) inequality (Step~1 of \Cref{subsec: selection-value-asymptotics}) is unchanged, and so is Step~3, since optimal plans are concentrated on \(\Sigma\cap\mathcal K\), where the density of \(\mathbb P_0\) with respect to \(\mathcal H^d_{|\Sigma}\) is bounded from below.

It remains to adapt the \(\Gamma\)-\(\limsup\) inequality (Step~2 of \Cref{subsec: selection-value-asymptotics}). We fix \(\overline\gamma\in\Pi(\mu_1,\mu_2)\) with \(\entropy(\overline\gamma\,|\,\mathbb P_0)<\infty\), which is concentrated on \(\Sigma\cap\mathcal K\subset\Sigma\cap\omega\), and \(R\geq1\), and we define the normal perturbation \(\widetilde\gamma_{\varepsilon,R}\), its marginals \(\widetilde\mu_{i,\varepsilon,R}\) and its shadow \(\gamma_{\varepsilon,R}\) as in \Cref{subsec: normal-recovery,subsec: correction}.

\textbf{Step 1. The marginals.} The marginal estimate of \Cref{thm: Wasserstein between the marginals} still holds: its proof only uses inverse charts of \(\proj^\Sigma_{M_i}\), which exist on \(\Sigma\cap\omega\) by \Cref{prop: twist-geometry} and finitely many of which cover the compact set \(\Sigma\cap\mathcal K\), and the integrability of the chart densities of \(\overline\gamma\). The energy estimate of \Cref{prop: shadow-energy} still holds as well: since \(\Omega\) is convex, the segments used in its proof stay in \(\interior\Omega\), and \eqref{eq: gradient-distance} holds with \(\Sigma\) replaced by \(\Sigma\cap\overline\omega\). Indeed, a nearest point \(\bar y\) of \(y\in\interior\Omega\) in the compact set \(\Sigma\cap\overline\omega\) lies in \(\interior\Omega\), so that \(\nabla E(\bar y)=0\), and the segment \([\bar y,y]\) lies in \(\interior\Omega\), whence \(|\nabla E(y)|\leq C_E\dist(y,\Sigma\cap\overline\omega)\) with \(C_E\defeq\sup_{\interior\Omega}\|D^2E\|<\infty\) by \ref{H2'}. Hence, for fixed \(R\),
\[
W_2(\widetilde\mu_{i,\varepsilon,R},\mu_i)=o_R(\sqrt\varepsilon)\quad(i=1,2)
\qquad\text{and}\qquad
\frac1\varepsilon\int E\,\dd(\gamma_{\varepsilon,R}-\widetilde\gamma_{\varepsilon,R})=o_R(1).
\]

\textbf{Step 2. The entropy.} The entropy upper bound is the point to adapt since, as we wrote it, \(\entropy(\widetilde\gamma_{\varepsilon,R}\,|\,\mathbb{P}_\varepsilon)=+\infty\) in this case: \(\widetilde\gamma_{\varepsilon,R}\) charges \(\Omega\setminus\mathcal K\), where \(\gibbs\) vanishes. We therefore consider the entropy with respect to the modified marginals:
\[
 I_{\varepsilon,R}\defeq
 \entropy\left(\widetilde\gamma_{\varepsilon,R}\,
 \middle|\,\widetilde\mu_{1,\varepsilon,R}\otimes\widetilde\mu_{2,\varepsilon,R}\right).
\]
Since \(\dd\gibbs=Z_\varepsilon^{-1}e^{-E/\varepsilon}\dd(\mu_1\otimes\mu_2)\), the shadow inequality (\Cref{prop: bound entropy shadow}) and Step~1 give, for fixed \(R\),
\[\begin{split}
 \entropy(\gamma_{\varepsilon,R}\,|\,\mathbb P_\varepsilon)&=\entropy(\gamma_{\varepsilon,R}\,|\,\mu_1\otimes\mu_2)+\log Z_\varepsilon+\frac1\varepsilon\int E\,\dd\gamma_{\varepsilon,R}\\
 &\leq I_{\varepsilon,R}+\log Z_\varepsilon+\frac1\varepsilon\int E\,\dd(\gamma_{\varepsilon,R}-\widetilde\gamma_{\varepsilon,R})+\frac1\varepsilon\int E\,\dd\widetilde\gamma_{\varepsilon,R}\\
 &\leq I_{\varepsilon,R}
 +\frac1\varepsilon\int E\,\dd\widetilde\gamma_{\varepsilon,R}
 +\log Z_\varepsilon+o_R(1).
\end{split}
\]
It remains to show that
\begin{equation}
\label{eq: euclidean-entropy-limsup}
 \limsup_{\varepsilon\to0}\left[
 I_{\varepsilon,R}
 +\frac1\varepsilon\int E\,\dd\widetilde\gamma_{\varepsilon,R}
 +\log Z_\varepsilon\right]
 \leq\entropy(\overline\gamma\,|\,\mathbb P_0)+b_R(\overline\gamma).
\end{equation}
Since \(|\Phi(z,\sqrt\varepsilon s)-z|\leq\sqrt\varepsilon R\) for \(s\in\ball_R\) and \(\overline\gamma\) is concentrated on \(\mathcal K\), for \(\sqrt\varepsilon R\leq1\) the measures \(\widetilde\gamma_{\varepsilon,R}\) and \(\widetilde\mu_{i,\varepsilon,R}\) are concentrated on the fixed compact sets \(\mathcal K^1\defeq\{\x:\dist(\x,\mathcal K)\leq1\}\) and \(K_i^1\defeq\{x:\dist(x,K_i)\leq1\}\). For an absolutely continuous probability measure \(\nu\) concentrated on a bounded set, we denote by \(\entropy(\nu)\defeq\entropy(\nu\,|\,\mathcal L)\) its entropy with respect to the Lebesgue measure. Since \(\widetilde\mu_{1,\varepsilon,R}\otimes\widetilde\mu_{2,\varepsilon,R}\) has the product of the marginal densities as density, we have
\[
I_{\varepsilon,R}=\entropy(\widetilde\gamma_{\varepsilon,R})-\entropy(\widetilde\mu_{1,\varepsilon,R})-\entropy(\widetilde\mu_{2,\varepsilon,R}),
\]
all these entropies being finite, since \(\widetilde\gamma_{\varepsilon,R}\) has finite entropy and all the measures are concentrated on fixed bounded sets. Hence
\begin{equation}
\label{eq: euclidean-entropy-decomposition}
I_{\varepsilon,R}
+\frac1\varepsilon\int E\,\dd\widetilde\gamma_{\varepsilon,R}
+\log Z_\varepsilon
=\Bigl(\entropy(\widetilde\gamma_{\varepsilon,R})+\frac1\varepsilon\int E\,\dd\widetilde\gamma_{\varepsilon,R}+\log Z_\varepsilon\Bigr)-\sum_{i=1}^2\entropy\left(\widetilde\mu_{i,\varepsilon,R}\right),
\end{equation}
and we pass to the limit in the two terms of the right-hand side in the next two steps.

\textbf{Step 3. The first term of \eqref{eq: euclidean-entropy-decomposition}.} Let \(\widetilde p_{\varepsilon,R}\) be the density \eqref{eq: normal-recovery-density} of \(\widetilde\gamma_{\varepsilon,R}\), and recall that \(K_R(z,s)=Z_R(z)^{-1}e^{-\Hn{z}(s,s)/2}\) on \(\ball_R\). Writing \(E(\Phi(z,\eta))=\frac12\Hn{z}(\eta,\eta)+\mathcal R(z,\eta)\) with \(|\mathcal R(z,\eta)|\leq C|\eta|^3\), the quadratic terms cancel, and, for \(\mathcal H^d\)-a.e.\ \(z\in\Sigma\cap\mathcal K\) and \(s\in\ball_R\),
\[
\log\Bigl(Z_\varepsilon e^{E/\varepsilon}\,\widetilde p_{\varepsilon,R}\Bigr)\bigl(\Phi(z,\sqrt\varepsilon s)\bigr)
=\log\overline\rho(z)+\log\frac{Z_\varepsilon}{\varepsilon^{k/2}Z_R(z)}+\frac{\mathcal R(z,\sqrt\varepsilon s)}{\varepsilon}-\log\mathcal J_\Phi(z,\sqrt\varepsilon s).
\]
We integrate this identity against \(K_R(z,s)\dd s\,\dd\overline\gamma(z)\), which gives the first term of \eqref{eq: euclidean-entropy-decomposition} by \eqref{eq: recovery sequence with wrong marginals 1}. The integral of the last two terms is bounded by \(C\sqrt\varepsilon(M_1+M_3)\to0\), by \eqref{eq: tube-jacobian-rescaled} and \eqref{eq: gaussian-moments}. Moreover, since \(Z_\varepsilon=(2\pi\varepsilon)^{k/2}(Z_0+o(1))\) by \Cref{thm: euclidean-Gibbs}, we have, uniformly in \(z\), with \(Z_\infty\) as in \Cref{lemma: convergence entropy},
\[
\log\frac{Z_\varepsilon}{\varepsilon^{k/2}Z_R(z)}=\log\Bigl(Z_0\sqrt{\det\Hn{z}}\Bigr)+\log\frac{Z_\infty(z)}{Z_R(z)}+o(1).
\]
Since \(\int_\Sigma|\log\overline\rho|\dd\overline\gamma<\infty\) (because \(\entropy(\overline\gamma\,|\,\mathbb P_0)<\infty\)) and \(\mathbb P_0\) has density \(Z_0^{-1}(\det\Hn{z})^{-1/2}\rho(z)\) with respect to \(\mathcal H^d_{|\Sigma}\) on \(\Sigma\cap\mathcal K\), we obtain
\begin{align*}
\lim_{\varepsilon\to0}\Bigl(\entropy(\widetilde\gamma_{\varepsilon,R})+\frac1\varepsilon\int E\,\dd\widetilde\gamma_{\varepsilon,R}+\log Z_\varepsilon\Bigr)
&=\int_\Sigma\log\frac{Z_0\,\overline\rho\sqrt{\det\Hn{z}}}{\rho}\dd\overline\gamma+\int_\Sigma\log\rho\dd\overline\gamma+b_R(\overline\gamma)\\
&=\entropy(\overline\gamma\,|\,\mathbb P_0)+\entropy(\mu_1)+\entropy(\mu_2)+b_R(\overline\gamma),
\end{align*}
where we used that \(\overline\gamma\) has marginals \(\mu_1,\mu_2\) and \(\rho=\rho_1\otimes\rho_2\), so that \(\int\log\rho\dd\overline\gamma=\entropy(\mu_1)+\entropy(\mu_2)\).

\textbf{Step 4. The marginal entropies in \eqref{eq: euclidean-entropy-decomposition}.} By \eqref{eq: normal-perturbation-W2}, \(\widetilde\mu_{i,\varepsilon,R}\rightharpoonup\mu_i\) as \(\varepsilon\to0\), and all these measures are concentrated on the compact set \(K_i^1\). Since \(\entropy(\cdot)=\entropy(\cdot\,|\,\mathcal L_{|K_i^1})\) on such measures, and the relative entropy with respect to a finite measure is weakly lower semicontinuous, we get
\[
\liminf_{\varepsilon\to0}\entropy\left(\widetilde\mu_{i,\varepsilon,R}\right)\geq\entropy(\mu_i),\qquad i=1,2.
\]

\textbf{Conclusion.} Plugging Steps~3 and~4 into \eqref{eq: euclidean-entropy-decomposition} gives \eqref{eq: euclidean-entropy-limsup}, hence
\[
\limsup_{\varepsilon\to0}\entropy(\gamma_{\varepsilon,R}\,|\,\gibbs)\leq\entropy(\overline\gamma\,|\,\mathbb P_0)+b_R(\overline\gamma).
\] The diagonal argument of Step~2 of \Cref{subsec: selection-value-asymptotics} then provides a recovery sequence, and the remaining conclusions follow as in the proof of \hyperref[thm: A]{Theorem~A}.
\end{proof}

The purpose of the following proposition is to show that, in the quadratic case, we recover the known asymptotic expansion of \(\mathrm{EOT}\) (see, for instance, Pal~\cite{pal2024difference}).

\begin{proposition}
\label{prop: quadratic-expansion}
Assume \ref{H1'}--\ref{H4'} hold. If moreover \(c(x_1,x_2)=\tfrac12|x_1-x_2|^2\), then, denoting by \(\gamma_\varepsilon\) the minimizer of \(\mathcal F_\varepsilon\), by \(\overline\gamma\) the optimal plan, and by \(\entropy(\mu_i)\defeq\entropy(\mu_i\,|\,\mathcal L^d)=\int\rho_i\log\rho_i\dd x_i\),
\begin{equation*}
\begin{aligned}
 \mathcal F_\varepsilon(\gamma_\varepsilon)
 =
 \mathcal F^{(0)}(\overline\gamma)+\frac{d}{2}\varepsilon\log(1/\varepsilon)-\frac{\varepsilon}{2}\Bigl(
\entropy(\mu_1)+\entropy(\mu_2)
 +d\log(2\pi)
 \Bigr)+o(\varepsilon).
\end{aligned}
\end{equation*}
\end{proposition}

\begin{proof}
    By Brenier's theorem, in the quadratic case the optimal plan \(\overline\gamma\) is unique and given by \(\overline\gamma=(\operatorname{id},T)_{\#}\mu_1\), where \(T\) is the Brenier map between \(\mu_1\) and \(\mu_2\). Consequently, \(\overline\gamma\) is supported on the graph \(\{(x_1,T(x_1)) \,|\, x_1\in\spt\mu_1\}\subseteq\Sigma\) and the tangent space of \(\Sigma\) is
\[
T_{(x_1,T(x_1))}\Sigma=\img\begin{pmatrix}\operatorname{Id}\\ DT(x_1)\end{pmatrix}.
\]
Hence, the Hessian of the duality gap is given by
    \begin{align*}     
    \hess E\big(x_1,T(x_1)\big) 
    &=\begin{pmatrix}
        \operatorname{Id}- D^2\overline\varphi_1(x_1) & -\operatorname{Id}
        \\ 
        -\operatorname{Id}  & \operatorname{Id}- D^2\overline\varphi_2(T(x_1))
    \end{pmatrix}
    \\
    &=\begin{pmatrix}
        DT(x_1) & -\operatorname{Id}
        \\ 
        -\operatorname{Id}  & DT(x_1)^{-1}
    \end{pmatrix},
\end{align*}
where we used that \(T=\id-\nabla\overline\varphi_1\) (from \(\nabla_{x_1}E=0\) on \(\Sigma\)) and that \(\id-D^2\overline\varphi_2(T(x_1))=DT(x_1)^{-1}\), obtained by differentiating the identity \(x_1=T(x_1)-\nabla\overline\varphi_2(T(x_1))\) (from \(\nabla_{x_2}E=0\) on \(\Sigma\)).
We define \(n(x_1)\), the matrix whose columns form an orthonormal basis of the normal space \(N_{(x_1,T(x_1))}\Sigma\), as 
\[n(x_1) \defeq \begin{pmatrix}
    -DT(x_1) \\
        \operatorname{Id}
        \end{pmatrix}(\operatorname{Id}+ DT(x_1)^2)^{-1/2}.\]
The reduced Hessian of \(E\) in the normal directions is therefore
\begin{align*}
\hess^\perp &E\big(x_1,T(x_1)\big)
\\
&= n(x_1)^\top \hess E\big(x_1,T(x_1)\big)n(x_1)
\\
&= (DT(x_1)^2 + \operatorname{Id})^{-1/2}\Big( DT(x_1)^3 + 2DT(x_1) + DT(x_1)^{-1}\Big) (DT(x_1)^2 + \operatorname{Id})^{-1/2}
\\
&=DT(x_1) + DT(x_1)^{-1}
\end{align*}
For \(\mu_1\)-almost every \(x_1\), with \(x_2=T(x_1)\)  the Monge-Ampère equation gives
\[
\det DT(x_1)=\frac{\rho_1(x_1)}{\rho_2(T(x_1))}=\frac{\rho_1(x_1)}{\rho_2(x_2)}
\]
and therefore
\[
\det\big(\hess^\perp E(x_1,x_2)\big)^{-1/2}
=\left(\frac{\rho_1(x_1)}{\rho_2(x_2)}\right)^{1/2}\det\big(\operatorname{Id}+DT(x_1)^2\big)^{-1/2}.
\]
Consequently, writing \(\overline\rho\) for the density of \(\overline\gamma\) with respect to \(\mathcal H^d_{|\Sigma}\), we obtain
\begin{align*}
\entropy(\overline\gamma\,|\,\mathbb P_0)
&=\int_{\Sigma}
\log\left(
\frac{\overline\rho(x_1,x_2)}{Z_0^{-1}\det\big(\operatorname{Id}+DT(x_1)^2\big)^{-1/2}\,
\rho_1(x_1)^{3/2}\rho_2(x_2)^{1/2}}
\right)\dd\overline\gamma(x_1,x_2)
\\
&=\entropy\left(\overline\gamma\,|\,\mathcal H^d_{|\Sigma}\right)
-\frac32\entropy(\mu_1)-\frac12\entropy(\mu_2)
\\
&\qquad 
+\frac12\int\log\det\big(\operatorname{Id}+DT(x_1)^2\big)\dd\mu_1(x_1)
+\log Z_0 .
\end{align*}
We now compute \(\entropy(\overline\gamma\,|\,\mathcal H^d_{|\Sigma})\) using the graph structure of \(\Sigma\). Since \(\Sigma\) is parametrized by \(x_1\mapsto(x_1,T(x_1))\), whose differential \(\begin{pmatrix}\operatorname{Id}\\ DT(x_1)\end{pmatrix}\) has Gram matrix \(\operatorname{Id}+DT(x_1)^2\), the area formula gives
\[
\dd\mathcal H^d_{|\Sigma}=\det\big(\operatorname{Id}+DT(x_1)^2\big)^{1/2}\dd x_1 .
\]
Hence
\[
\overline\rho(x_1,T(x_1))=\rho_1(x_1)\det\big(\operatorname{Id}+DT(x_1)^2\big)^{-1/2}.
\]
Therefore
\begin{align*}
\entropy\left(\overline\gamma\,|\,\mathcal H^d_{|\Sigma}\right)
&=\int\log\overline\rho\,\dd\overline\gamma
=\int\log\Big(\rho_1(x_1)\det\big(\operatorname{Id}+DT(x_1)^2\big)^{-1/2}\Big)\dd\mu_1(x_1)
\\
&=\entropy(\mu_1)-\frac12\int\log\det\big(\operatorname{Id}+DT(x_1)^2\big)\dd\mu_1(x_1).
\end{align*}
Plugging this into the previous identity, the \(\log\det\) terms cancel and we get
\begin{align*}
\entropy(\overline\gamma\,|\,\mathbb P_0)-\log Z_0
&=\entropy(\mu_1)-\frac32\entropy(\mu_1)-\frac12\entropy(\mu_2)
\\
&=-\frac12\big(\entropy(\mu_1)+\entropy(\mu_2)\big).
\end{align*}
Finally, we conclude by \Cref{thm: euclidean-selection}.
\end{proof}

\begin{remark}
Our result is consistent with the one of Pal~\cite{pal2024difference}, up to the choice of the reference measure. Pal regularizes the transport cost with \(\entropy(\gamma)\), that is, with the entropy relative to the Lebesgue measure, and obtains the limit \(+\tfrac12\big(\entropy(\mu_1)+\entropy(\mu_2)\big)\). Since the marginals of \(\gamma\) are fixed, 
\[
\entropy\left(\gamma\,|\,\mu_1\otimes\mu_2\right)=\entropy(\gamma)-\entropy(\mu_1)-\entropy(\mu_2),
\]
so that, up to normalizing constants, the two results differ exactly by \[-\entropy(\mu_1)-\entropy(\mu_2),\]
which gives the value obtained above.
\end{remark}

\subsubsection{The McCann--Pass--Warren cost}
\label{subsubsec: MPW}
We illustrate \Cref{thm: euclidean-selection} with the non-twisted cost of \cite[Example~3.1]{mccann2012rectifiability}. In this paragraph, the points of \(\R^2\times\R^2\) are denoted by \((x_1,x_2)\), with \(x_1=(u_1,v_1)\) and \(x_2=(u_2,v_2)\), and
\[
c(x_1,x_2)\defeq e^{u_1+u_2}\cos(v_1-v_2)+\frac{e^{2u_1}}2+\frac{e^{2u_2}}2\geq\frac12\left(e^{u_1}-e^{u_2}\right)^2\geq0.
\]
Equality holds if and only if \(u_1=u_2\) and \(\cos(v_1-v_2)=-1\), that is, if and only if \((x_1,x_2)\) belongs to one of the planes
\[
\Sigma_j\defeq\{(x_1,T_j(x_1)):x_1\in\R^2\},
\qquad
T_j(x_1)\defeq\bigl(u_1,v_1+(2j+1)\pi\bigr),
\qquad j\in\Z.
\]
Since \(D^2_{x_1x_2}c(x_1,x_2)=e^{u_1+u_2}\begin{psmallmatrix}\cos\theta&\sin\theta\\-\sin\theta&\cos\theta\end{psmallmatrix}\) with \(\theta=v_1-v_2\), the cost is non-degenerate everywhere, but it is not twisted, since infinitely many planes \(\Sigma_j\) lie above each \(x_1\).

\paragraph{Marginals and optimal plans.}
Let \(\mu_1=\mu_2\) be the uniform probability measure on \(A\defeq[0,1]\times[0,4\pi]\), that is, \(\rho_1=\rho_2=\frac1{4\pi}\mathbf 1_A\). The map
\[
S(x_1)\defeq\begin{cases}(u_1,v_1+\pi)&\text{if }v_1\in[0,\pi)\cup[2\pi,3\pi),\\(u_1,v_1-\pi)&\text{if }v_1\in[\pi,2\pi)\cup[3\pi,4\pi],\end{cases}
\]
preserves \(\mu_1\) and satisfies \(S(x_1)\in\{T_0(x_1),T_{-1}(x_1)\}\), so that the plan \((\id,S)_\#\mu_1\) belongs to \(\Pi(\mu_1,\mu_2)\) and is concentrated on \(\{c=0\}\). Hence it is optimal, \(\OT(\mu_1,\mu_2)=0\), and \(\overline\varphi_1=\overline\varphi_2=0\) are Kantorovich potentials, the dual value being \(0\). We take for \(M_1=M_2\) a closed ball containing a neighborhood of \(K_1=K_2\defeq A\). Then \(E=c\) and \(\Sigma=\Omega\cap\bigcup_j\Sigma_j\), so that every element of \(\Pi(\mu_1,\mu_2)\) concentrated on \(\bigcup_j\Sigma_j\) is optimal. The optimal plan is not unique: for instance, the map \(x_1\mapsto(u_1,v_1+\pi)\) if \(v_1<3\pi\), \(x_1\mapsto(u_1,v_1-3\pi)\) otherwise, also preserves \(\mu_1\) and induces another optimal plan. Indeed, above each \(x_1\in A\) lie exactly two points \(T_j(x_1)\in A\) (for almost every \(x_1\)), and the optimal plans are the couplings which split the mass of \(x_1\) between them.

\paragraph{Assumptions.}
Assumptions \ref{H1'}--\ref{H4'} hold: \(c\) is smooth and non-degenerate, \(E=c\), \(\Sigma\) is a union of affine planes, \(\rho_i\) is constant on \(K_i=A\), and \(\mathcal H^2(\Sigma\cap\partial\mathcal K)=0\), since \(\Sigma\cap\partial\mathcal K\) is contained in a finite union of lines. Hence \Cref{thm: euclidean-selection} applies.

\paragraph{The limit Gibbs measure.}
At \(z=(x_1,x_2)\in\Sigma\), we have \(u_1=u_2\) and \(v_1-v_2\in\pi+2\pi\Z\), whence \(D^2_{x_1x_1}E(z)=D^2_{x_2x_2}E(z)=e^{2u_1}I_2\) and \(D^2_{x_1x_2}E(z)=-e^{2u_1}I_2\). Denoting by \((e_1,e_2)\) the canonical basis of \(\R^2\), the vectors \((e_1,-e_1)/\sqrt2\) and \((e_2,-e_2)/\sqrt2\) form an orthonormal basis of \(N_z\Sigma\), in which \(\Hn{z}=2e^{2u_1}I_2\), so that \(\det\Hn{z}=4e^{4u_1}\). Since \(x_1\mapsto(x_1,T_j(x_1))\) has Jacobian \(2\), we have \(\mathcal H^2_{|\Sigma_j}=(\id,T_j)_\#(2\mathcal L^2)\). Moreover, for every \(a\in[0,1]\) and a.e.\ \(b\in\R\),
\begin{equation}
\label{eq: MPW-periodization}
\sum_{j\in\Z}\rho_1\bigl(a,b+(2j+1)\pi\bigr)=\sum_{j\in\Z}\rho_1\bigl(a,b-(2j+1)\pi\bigr)=\frac1{2\pi},
\end{equation}
since the interval \([0,4\pi]\) has length \(4\pi\), and thus contains exactly two points of \(b+\pi+2\pi\Z\) for a.e.\ \(b\). Therefore
\[
Z_0=\sum_{j\in\Z}\int_{\R^2}\frac{e^{-2u_1}}2\rho_1(x_1)\rho_2(T_j(x_1))\,2\dd x_1
=\frac1{2\pi}\int_{\R^2}e^{-2u_1}\rho_1(x_1)\dd x_1=\frac{1-e^{-2}}{4\pi},
\]
and
\[
\dd\mathbb P_0(x_1,x_2)=\frac{2\pi}{1-e^{-2}}\,e^{-2u_1}\rho_1(x_1)\rho_2(x_2)\dd\mathcal H^2_{|\Sigma}(x_1,x_2).
\]
By \eqref{eq: MPW-periodization}, the marginals of \(\mathbb P_0\) have densities \(\frac{2}{1-e^{-2}}e^{-2u_i}\rho_i(x_i)\), \(i=1,2\), so that \(\mathbb P_0\notin\Pi(\mu_1,\mu_2)\).

\paragraph{The selected plan.}
\begin{proposition}
\label{prop: MPW-selection}
As \(\varepsilon\to0\), the entropic optimal plans converge weakly to
\[
\dd\gamma_0(x_1,x_2)\defeq\pi\rho_1(x_1)\rho_2(x_2)\dd\mathcal H^2_{|\Sigma}(x_1,x_2),
\qquad\text{that is,}\qquad
\gamma_0=\frac1{16\pi}\mathcal H^2_{|\Sigma\cap(A\times A)}.
\]
\end{proposition}

\begin{proof}
By \eqref{eq: MPW-periodization} and \(\mathcal H^2_{|\Sigma_j}=(\id,T_j)_\#(2\mathcal L^2)\), the marginals of \(\gamma_0\) have densities
\[
\sum_{j\in\Z}2\pi\rho_1(x_1)\rho_2(T_j(x_1))=\rho_1(x_1)
\qquad\text{and}\qquad
\sum_{j\in\Z}2\pi\rho_1(T_j^{-1}(x_2))\rho_2(x_2)=\rho_2(x_2),
\]
so that \(\gamma_0\in\Pi(\mu_1,\mu_2)\). Moreover, \(\gamma_0\) and \(\mathbb P_0\) are equivalent, with
\[
\frac{\dd\gamma_0}{\dd\mathbb P_0}(x_1,x_2)=g(u_1)\defeq\frac{1-e^{-2}}2\,e^{2u_1},
\]
which only depends on \(x_1\). Hence, for every \(\gamma\in\Pi(\mu_1,\mu_2)\) such that \(\entropy(\gamma\,|\,\mathbb P_0)<\infty\),
\begin{align*}
\entropy(\gamma\,|\,\mathbb P_0)
&=\entropy(\gamma\,|\,\gamma_0)+\int\log g(u_1)\dd\gamma(x_1,x_2)
=\entropy(\gamma\,|\,\gamma_0)+\int\log g(u_1)\dd\mu_1(x_1)\\
&=\entropy(\gamma\,|\,\gamma_0)+\entropy(\gamma_0\,|\,\mathbb P_0).
\end{align*}
Thus \(\gamma_0\) is the unique minimizer of \(\entropy(\cdot\,|\,\mathbb P_0)\) on \(\Pi(\mu_1,\mu_2)\), and the conclusion follows from \Cref{thm: euclidean-selection}.
\end{proof}

Since \(\dd\mathbb P_0/\dd\mathcal H^2_{|\Sigma}\) is the product of a function of \(x_1\) and a function of \(x_2\), \(\gamma_0\) is equivalently the unique minimizer of \(\entropy(\cdot\,|\,\mathcal H^2_{|\Sigma})\) on \(\Pi(\mu_1,\mu_2)\). In other words, the entropic selection splits the mass of each \(x_1\in A\) equally between the two points \(T_j(x_1)\in A\), whereas the limit Gibbs measure \(\mathbb P_0\) is not a coupling of \(\mu_1\) and \(\mu_2\).

\subsection{The extension to multi-marginal framework}
\label{subsec: multimarginal-extension}
Let \(N\geq2\) and, for \(i=1,\ldots,N\), let \(M_i=\T^d\) and \(\mu_i=\rho_i\vol_{M_i}\), where \(\rho_i\) is Lipschitz and \(\alpha\leq\rho_i\leq\beta\). We set \(\Omega\defeq\prod_{i=1}^NM_i\), \(\rho(z)\defeq\prod_{i=1}^N\rho_i(z_i)\) and \(\mu^\otimes\defeq\bigotimes_{i=1}^N\mu_i=\rho\vol_\Omega\), and we denote by \(\Pi(\mu_1,\ldots,\mu_N)\) the set of probability measures on \(\Omega\) with marginals \(\mu_1,\ldots,\mu_N\). The cost \(c\) is continuous, and we fix continuous Kantorovich potentials \(\overline\varphi_1,\ldots,\overline\varphi_N\) such that the duality gap \(E\defeq c-\sum_{i=1}^N\overline\varphi_i(x_i)\geq0\) belongs to \(C^{2,1}(\Omega)\), and such that \(\Sigma\defeq\{E=0\}\) is a \(C^{1,1}\) submanifold of \(\Omega\) without boundary, of dimension \(d\), satisfying
\begin{equation}
\label{eq: multimarginal-rank}
\rank\hess E(z)=(N-1)d\qquad\text{for every }z\in\Sigma.
\end{equation}
The Gibbs measures \(\gibbs\), the constants \(Z_\varepsilon\) and the functionals \(\mathcal F_\varepsilon\), \(\mathcal F^{(0)}\), \(\mathcal F^{(1)}\) are defined as in \Cref{sec: intro,sec: framework}, with \(\Pi(\mu_1,\ldots,\mu_N)\) in place of \(\Pi(\mu_1,\mu_2)\), and with the codimension \(k\defeq(N-1)d\) of \(\Sigma\) in place of \(d\).

For two marginals, \eqref{eq: multimarginal-rank} follows from the non-degeneracy of \(c\) (\Cref{prop: twist-geometry}). For \(N\geq3\), it has to be imposed: since \(T_z\Sigma\subseteq\ker\hess E(z)\) and \(\dim T_z\Sigma=d\), it amounts to \(\ker\hess E(z)=T_z\Sigma\), that is, to the positive definiteness of \(\Hn{z}\). Without it, some normal directions would be degenerate and would contribute differently to the asymptotics.

\begin{theorem}[Multi-marginal \(\Gamma\)-expansion]
\label{thm: multimarginal-extension}
Under these assumptions, the conclusions of \Cref{thm: convergence-Gibbs-measure} and of \hyperref[thm: A]{Theorem~A} hold, with \(k\) in place of \(d\) in \eqref{eq: gamma-expansion}--\eqref{eq: energy-expansion}, and with
\[
\dd\mathbb P_0(z)=\frac1{Z_0}\left(\det\Hn{z}\right)^{-1/2}\rho(z)\dd\mathcal H^d_{|\Sigma}(z),
\qquad
Z_0\defeq\int_\Sigma\left(\det\Hn{z}\right)^{-1/2}\rho(z)\dd\mathcal H^d(z).
\]
\end{theorem}

\begin{proof}
The proofs of \Cref{sec: gibbs,sec: selection} and of \Cref{subsec: Convergence of the Gibbs measure and the geometric Laplace's method} use the two-marginal structure at three places only.

\textbf{(i) Normal ellipticity.} In \Cref{prop: twist-geometry}, \eqref{eq: normal-ellipticity} is deduced from the non-degeneracy of \(c\). Here, it follows from \eqref{eq: multimarginal-rank}, the positive semidefiniteness of \(\hess E\) on \(\Sigma\), and the compactness of \(\Sigma\).

\textbf{(ii) Local normal frames.} A \(C^{1,1}\) submanifold admits local Lipschitz orthonormal frames of its normal bundle; see the Gram--Schmidt construction in \cite[Section~2.3.1, arXiv version]{minguzzi2015convex}. We choose a finite cover \((U_a)_a\) of \(\Sigma\) by relatively compact patches with such frames \(n_a\), and a subordinate Lipschitz partition of unity \((\chi_a)_a\). For \(r>0\) sufficiently small, the maps
\[
\Phi_a(z,s)\defeq z+n_a(z)s,
\qquad z\in U_a,\quad s\in\ball_r,
\]
are bi-Lipschitz tubular parametrizations, by the positive-reach property used in \Cref{subsec:structure}. Since the frames are Lipschitz, the proof of \Cref{lemma: determinant of Phi} applies on each patch, and \eqref{eq: tube-jacobian-expansion} holds almost everywhere, with uniform bounds over the finite cover. The area formula then gives
\[
\int_{\Sigma_r}h\dd\vol_\Omega
=\sum_a\int_{U_a}\chi_a(z)\int_{\ball_r}
h\bigl(\Phi_a(z,s)\bigr)\mathcal J_{\Phi_a}(z,s)\dd s\dd\mathcal H^d(z)
\]
for every nonnegative Borel function \(h\). We apply the three-region Laplace argument to each integral in this sum. The weights \(\chi_a\) depend only on \(z\), and the resulting limits sum to the desired integral because \(\sum_a\chi_a=1\).

On overlaps, \(n_b(z)=n_a(z)O_{ab}(z)\) with \(O_{ab}(z)\in O(k)\). The coordinate change \((z,s)\mapsto(z,O_{ab}(z)s)\) has absolute Jacobian one almost everywhere, even though \(O_{ab}\) depends on \(z\). It also preserves the normal Gaussian measure and \(\ball_R\), that is 
\[
H_b=O_{ab}^{\top}H_aO_{ab},
\qquad
K_b(z,s)=K_a(z,O_{ab}s),
\] where \(H_a(z)\) is the matrix of the normal Hessian in the local orthonormal frame \(n_a(z)\). Thus the local normal perturbations agree on overlaps and define the recovery measure in \Cref{subsec: normal-recovery}. In the Euclidean case, we use a finite cover of the compact portion \(\Sigma\cap\overline\omega\) appearing in the Laplace proof.

With (i) and (ii), the proofs of \Cref{thm: convergence-Gibbs-measure} (\Cref{subsec: Convergence of the Gibbs measure and the geometric Laplace's method}), of \Cref{lemma: convergence entropy}, \Cref{prop: control of the shadow entropy}, \Cref{prop: shadow-energy}, and of Steps~1 and~3 of \Cref{subsec: selection-value-asymptotics} apply verbatim, with sums over \(i\in\{1,\ldots,N\}\); recall that Step~3 only uses that \(\proj_{M_1}\) is \(1\)-Lipschitz.

\textbf{(iii) Marginal estimate.} The proof of \Cref{thm: Wasserstein between the marginals} uses local inverses of \(\proj_{M_i}^\Sigma\), which, for \(N\geq3\), may have critical points (vertical branches). Let \(\overline\gamma\) be the plan fixed in Step~2 of \Cref{subsec: selection-value-asymptotics}. We fix \(i\) and \(R\geq1\), and write \(p\defeq\proj_{M_i}^\Sigma\), \(\mu\defeq\mu_i\) and \(\widetilde\mu_\varepsilon\defeq\widetilde\mu_{i,\varepsilon,R}\). Let \(C\defeq\{z\in\Sigma:\rank Dp(z)<d\}\). By Sard's theorem, \(p(C)\) is Lebesgue-negligible, so that \(\overline\gamma(C)\leq\mu(p(C))=0\). Given \(\eta>0\), we choose an open set \(V\supset p(C)\) such that \(\mu(V)<\eta\), and we split \(\overline\gamma=\overline\gamma^b+\overline\gamma^g\), with \(\overline\gamma^b\defeq\overline\gamma\mres p^{-1}(V)\). Accordingly, \(\mu=\mu^b+\mu^g\) and \(\widetilde\mu_\varepsilon=\widetilde\mu^b_\varepsilon+\widetilde\mu^g_\varepsilon\), where \(\mu^b\), \(\mu^g\) and \(\widetilde\mu^b_\varepsilon\), \(\widetilde\mu^g_\varepsilon\) are the \(i\)-th marginals of \(\overline\gamma^b\), \(\overline\gamma^g\) and of their normal perturbations, defined by \eqref{eq: recovery sequence with wrong marginals 1}. On the compact set \(\Sigma\setminus p^{-1}(V)\), \(p\) is a local diffeomorphism, and the proof of \Cref{thm: Wasserstein between the marginals}, which only uses that the chart densities are integrable, gives \(W_2^2(\mu^g,\widetilde\mu^g_\varepsilon)=o_\eta(\varepsilon)\), where \(W_2\) is extended to measures of equal mass. On the other hand, since \(|p(\Phi(z,\sqrt\varepsilon s))-p(z)|\leq\sqrt\varepsilon|s|\), the map \((z,s)\mapsto(p(z),p(\Phi(z,\sqrt\varepsilon s)))\) pushes \(K_R(z,s)\dd s\,\dd\overline\gamma^b(z)\) forward to a coupling of \(\mu^b\) and \(\widetilde\mu^b_\varepsilon\) with quadratic cost at most
\[
\varepsilon\int_\Sigma\int_{\ball_R}|s|^2K_R(z,s)\dd s\dd\overline\gamma^b(z)\leq M_2\,\varepsilon\,\overline\gamma^b(\Sigma)=M_2\,\varepsilon\,\mu(V)\leq M_2\,\varepsilon\,\eta,
\]
by \eqref{eq: gaussian-moments}. Gluing the two couplings, we get \(W_2^2(\mu,\widetilde\mu_\varepsilon)\leq o_\eta(\varepsilon)+M_2\varepsilon\eta\). Letting \(\varepsilon\to0\), then \(\eta\to0\), gives \(W_2(\mu,\widetilde\mu_\varepsilon)=o_R(\sqrt\varepsilon)\), which completes the proof.
\end{proof}

Combining this proof with \Cref{subsec: euclidean-extension}, \Cref{thm: multimarginal-extension} extends to \(N\) compactly supported marginals on \(\R^d\), under the analogue of \ref{H1'}--\ref{H4'} in which \(\mathcal K\defeq\prod_{i=1}^NK_i\), and the non-degeneracy of \(c\) is replaced by \eqref{eq: multimarginal-rank} on \(\Sigma\cap\interior\Omega\).

\subsubsection{Multi-marginal optimal transport with repulsive harmonic cost}
We apply \Cref{thm: multimarginal-extension} to the special case of
three marginals on the circle.
Let \(N=3\), \(d=1\), and \(M_i=\mathbb S^1\subset\mathbb R^2\) for
\(i=1,2,3\). We use the uniform probability
\(\mu=(2\pi)^{-1}\mathcal H^1_{|\mathbb S^1}\) and set \(\mu_i=\mu\).
Here \(d\) is the intrinsic dimension of each marginal space.
The contact set computed below is a disjoint union of two circles,
so \(\dim\Sigma=1=d\) and its codimension in \((\mathbb S^1)^3\) is
\(k=2\). 
We consider the harmonic repulsive cost
\[
    c(\x)
    \defeq
    -\sum_{1\leq i<j\leq N}
    |x_i-x_j|^2.
\]
The associated multi-marginal optimal transport problem is 
\begin{equation}
\label{eq: MOT harmonic repulsive cost}
    \inf_{\gamma\in\Pi(\mu_i)}
    \int_{(\mathbb S^1)^N} c(\x)\dd\gamma(\x).
\end{equation}
\begin{proposition}
Problem \eqref{eq: MOT harmonic repulsive cost} is equivalent to
\begin{equation}
\label{eq: MOT harmonic repulsive cost 2}
    \inf_{\gamma\in\Pi(\mu_1,\mu_2,\mu_3)}
    \int_{(\mathbb S^1)^N}
    \Big|
        \sum_{i=1}^N x_i
    \Big|^2
    \dd\gamma(\x).
\end{equation}
\end{proposition}

\begin{proof}
We start by expanding the cost in the following way:
\begin{align*}
\sum_{1\leq i<j\leq N}|x_i-x_j|^2
&=
\sum_{1\leq i<j\leq N}
\left(
|x_i|^2+|x_j|^2-2\langle x_i, x_j\rangle
\right)
\\
&=
(N-1)\sum_{i=1}^N|x_i|^2
-2\sum_{1\leq i<j\leq N}\langle x_i, x_j\rangle.
\end{align*}
On the other hand, \(|\sum_{i=1}^N x_i|^2
= \sum_{i=1}^N|x_i|^2 + 2\sum_{1\leq i<j\leq N}\langle x_i, x_j\rangle\). Therefore, \(\sum_{1\leq i<j\leq N}|x_i-x_j|^2
=
N\sum_{i=1}^N|x_i|^2-|
\sum_{i=1}^N x_i|^2\). Since \(x_i\in \mathbb S^1\), one has \(|x_i|=1\), and hence \(
\sum_{i=1}^N|x_i|^2=N\). Consequently, \(c(\x)=|
\sum_{i=1}^N x_i|^2-N^2\).
Thus, for every \(\gamma\in\Pi(\mu_i)\),
\[
\int_{(\mathbb S^1)^N}c(\x)\,\dd\gamma(\x)
=
\int_{(\mathbb S^1)^N}
\Big|
\sum_{i=1}^N x_i
\Big|^2
\dd\gamma(\x)
-
N^2.
\]
Since \(N^2\) is independent of \(\gamma\), minimizing the original
cost is equivalent to minimizing
\[
\int_{(\mathbb S^1)^N}
\Big|
\sum_{i=1}^N x_i
\Big|^2
\dd\gamma(\x).
\]
\end{proof}
We denote by \(\overline c\) the cost of \Cref{eq: MOT harmonic repulsive cost 2}.

\paragraph{Geometry of the contact set}
The dual problem associated with
\Cref{eq: MOT harmonic repulsive cost 2} is
\begin{equation}
\label{eq: dual MOT harmonic repulsive cost 2}
    \sup_{(\varphi_i)_{i=1}^3}
    \left\{
        \sum_{i=1}^3
        \int_{\mathbb S^1}\varphi_i(x_i)\,\dd\mu_i(x_i)
        \;\Big|\;
        \sum_{i=1}^3\varphi_i(x_i)
        \leq \overline c(\x)
        \quad
        \forall \x \in(\mathbb S^1)^3
    \right\}.
\end{equation}
\begin{proposition}
A family of optimal Kantorovich potentials for the dual problem \Cref{eq: dual MOT harmonic repulsive cost 2} is given by
\[
    \overline\varphi_1=\overline\varphi_2=\overline\varphi_3=0.
\]
\end{proposition}

\begin{proof}
There exists a coupling \(\gamma\in\Pi(\mu_1,\mu_2,\mu_3)\) concentrated on \(\{\x\in(\mathbb S^1)^3\: |\:x_1+x_2+x_3=0\}\), for instance \(\overline \gamma=(F)_\#\mu_1\), with \(F(x)=(x, R_{2\pi/3}x, R_{-2\pi/3}x)\). One can easily check that \(\overline \gamma\in\Pi(\mu_1,\mu_2,\mu_3)\) and \(\overline c =0\) on its support. Thus the value of \Cref{eq: MOT harmonic repulsive cost 2} is null and therefore the admissible dual family \((\overline\varphi_1,\overline\varphi_2,\overline\varphi_3)=(0,0,0)\) attains the primal value and is a family of optimal Kantorovich potentials.
\end{proof}

\begin{proposition}
The contact set is a compact manifold without boundary 
    \[\Sigma= \{(x , R_{2\pi/3}x,  R_{-2\pi/3}x) \:|\: x\in \mathbb S^1\} \sqcup \{(x , R_{-2\pi/3}x,  R_{2\pi/3}x) \:|\: x\in \mathbb S^1\},\]
where \(R_\theta\) denotes the rotation of angle \(\theta\).
\end{proposition}

\begin{proof}
By definition, \(\Sigma = \{E=0\} \). Here \(E= \overline c\) which is equivalent to \(x_1= - x_2 - x_3\). Moreover, we have \(
|x_2+ x_3|^2 = |x_2|^2 + |x_3|^2 + 2 \langle x_2,x_3\rangle = 2 + 2\langle x_2,x_3\rangle\), and thus, using \(|x_2+x_3|=|x_1|=1\),  
\(\langle x_2,x_3 \rangle = -\frac{1}{2}\). It implies that \(x_2 = R_{\pm 2\pi/3}x_3\) and thus
\begin{align*}
\Sigma 
&= \{(x , R_{2\pi/3}x,  R_{-2\pi/3}x) \:|\: x\in \mathbb S^1\} \sqcup \{(x , R_{-2\pi/3}x,  R_{2\pi/3}x) \:|\: x\in \mathbb S^1\}
\\
&\defeq \Sigma_+ \sqcup \Sigma_-.
\end{align*}
Finally, the maps \(
F_\pm: x\in \mathbb S^1\mapsto F_\pm(x)
=
(x,R_{\pm 2\pi/3}x,
R_{\mp 2\pi/3}x) \in (\mathbb S^1)^3\) are smooth embeddings. Therefore \(\Sigma_+\) and \(\Sigma_-\) are
diffeomorphic to \(\mathbb S^1\). Since \(\mathbb S^1\) is compact and without
boundary, \(\Sigma\) is a compact manifold without boundary.
\end{proof}

\paragraph{Entropic selection principle}

\begin{proposition}

The limiting Gibbs measure has the explicit expression

    \[\mathbb P_0=\frac{1}{2}(\gamma_+ + \gamma_-),\]

where \(\gamma_\pm=(F_\pm)_\#\mu\) and \(F_\pm\) are the two
parametrizations above. Moreover, \(\mathbb P_0\) is the selected
transport plan.
\end{proposition}

\begin{proof}
We verify the normal-rank hypothesis and apply the Gibbs formula
of \Cref{thm: multimarginal-extension}:
    \[\mathbb P_0 = \frac{1}{Z_0} (\det \hess^\perp E)^{-1/2} \prod_{i=1}^N \rho_i \mathcal H^1_{| \Sigma}.\]
We compute the normal Hessian and the normalization. For \(i=1,2,3\), introduce the geodesics
\[\nu_i(\tau ) =\exp_{x_i}(\tau v_i), \]
where \(v_i \in T_{x_i}\mathbb S^1\). We also introduce \(S(\tau) \defeq \nu_1(\tau) + \nu_2(\tau) + \nu_3(\tau)\), then we have
\[\frac{\dd }{\dd \tau }E\big( \nu_1(\tau) , \nu_2(\tau) , \nu_3(\tau)\big) = 2 \langle \dot S(\tau) , S(\tau) \rangle, \]
and 
\[\frac{\dd^2 }{\dd\tau^2}E\big( \nu_1(\tau) , \nu_2(\tau) , \nu_3(\tau)\big) = 2\langle \ddot S(\tau) , S(\tau)\rangle + 2|\dot S(\tau)|^2.\]
Thus, for \(\tau =0\), since \( S(0)= 0\), we obtain
\[\frac{\dd^2 }{\dd\tau^2}E\big( \nu_1(\tau) , \nu_2(\tau) , \nu_3(\tau)\big)\Big|_{\tau=0}= 2|v_1 + v_2 + v_3|^2,\]
and denoting \(V = (v_1, v_2,v_3)\) we have 
\[\hess E(\x)[V,V]=2|v_1 + v_2 + v_3|^2. \]
Let \(e_i \defeq Jx_i\), where
\[
J=
\begin{pmatrix}
0 & -1\\
1 & 0
\end{pmatrix},
\]
so that \(e_i\) is a unit tangent vector to \(\mathbb S^1\) at \(x_i\). Then
\[
E_1=(e_1,0,0),\qquad E_2=(0,e_2,0),\qquad E_3=(0,0,e_3)
\]
is an orthonormal basis of \(T_{\x}(\mathbb S^1)^3\). We obtain
\[
\hess E(\x)[E_i,E_i]=2,
\qquad
\hess E(\x)[E_i,E_j]
=
2\langle e_i,e_j\rangle
=
2\langle x_i,x_j\rangle
=
-1,
\]
for \(i\neq j\), since \(\x\in\Sigma\). Hence
\[
[\hess E(\x)]
=
\begin{pmatrix}
2 & -1 & -1\\
-1 & 2 & -1\\
-1 & -1 & 2
\end{pmatrix}.
\]
Moreover,
\[
T_{\x}\Sigma
=
\spn \{E_1+E_2+E_3\},
\]
and therefore
\[
N_{\x}\Sigma
=
\left\{
aE_1+bE_2+cE_3
\;|\;
a+b+c=0
\right\}.
\]
On this normal space, the above matrix acts as \(3\id\). Thus
\[
\hess^\perp E(\x)=3\id_{N_{\x}\Sigma},
\]
Thus \(\operatorname{rank}\hess E=2=(N-1)d\), as required
by \Cref{thm: multimarginal-extension}, and
\(\det\hess^\perp E=9\). The induced arclength must also
be included in the normalization: the maps \(F_\pm\) have speed
\(\sqrt3\) when the unit circle is parametrized by arclength. Hence
\[
 \mathcal H^1(\Sigma_\pm)=2\pi\sqrt3,
 \qquad Z_0=(2\pi)^{-3}\frac{4\pi\sqrt3}{3}
 =\frac{1}{2\sqrt3\,\pi^2}.
\]
Consequently,
\[
 \mathbb P_0=\frac{1}{4\pi\sqrt3}\mathcal H^1_{|\Sigma}
 =\frac12(\gamma_++\gamma_-).
\]
Since \(\entropy(\cdot \,|\,\mathbb P_0) \) is nonnegative and \(\mathbb P_0 \in \Pi(\mu_i)\), we conclude that the selected transport plan is \(\mathbb P_0\).
\end{proof}
\subsubsection{Repulsive harmonic cost in \texorpdfstring{\(\R^d\)}{Rd}}
\label{subsubsec: harmonic-Rd}
Let \(N\geq3\), \(d\geq1\), and consider the repulsive harmonic cost
\[
c(\x)\defeq-\sum_{1\leq i<j\leq N}|x_i-x_j|^2,
\qquad
\x=(x_1,\ldots,x_N)\in(\R^d)^N.
\]
For \(i=1,\ldots,N\), let \(\mu_i=\rho_i\dd x\), where \(\rho_i\) vanishes outside a compact set \(K_i\subset\R^d\) with \(\mathcal L^d(\partial K_i)=0\), and is bounded and continuous on \(K_i\). We take for \(M_i\) a closed ball containing a neighborhood of \(K_i\), and we set \(\Omega\defeq\prod_iM_i\), \(\rho(\x)\defeq\prod_i\rho_i(x_i)\) and \(b\defeq\sum_i\int x\dd\mu_i(x)\). We assume:
\begin{enumerate}
\item[(M)]\label{M-condition} There exists \(i_0\in\{1,\ldots,N\}\) such that \(\mu_{i_0}\) is the law of \(b-\sum_{i\neq i_0}X_i\), where the random variables \(X_i\), \(i\neq i_0\), are independent with laws \(\mu_i\).
\end{enumerate}
Under (M), the law \(\gamma_M\) of \((X_1,\ldots,X_N)\), where \(X_{i_0}\defeq b-\sum_{i\neq i_0}X_i\), belongs to \(\Pi(\mu_1,\ldots,\mu_N)\) and is concentrated on the affine subspace
\[
\Sigma_b\defeq\Bigl\{\x\in(\R^d)^N:\sum_{i=1}^Nx_i=b\Bigr\},
\]
of dimension \((N-1)d\) and codimension \(k=d\). Since
\[
c(\x)=\Bigl|\sum_{i=1}^Nx_i-b\Bigr|^2+\sum_{i=1}^N\overline\varphi_i(x_i),
\qquad
\overline\varphi_i(x)\defeq-N|x|^2+2\langle b,x\rangle-\frac{|b|^2}N,
\]
the functions \(\overline\varphi_i\) are Kantorovich potentials, \(\gamma_M\) is optimal, the duality gap is \(E(\x)=|\sum_ix_i-b|^2\), and \(\Sigma=\Sigma_b\cap\Omega\). The normal space to \(\Sigma_b\) is spanned by the orthonormal vectors \((e_p,\ldots,e_p)/\sqrt N\), \(p=1,\ldots,d\), so that
\[
\Phi(z,s)\defeq z+\frac{(s,\ldots,s)}{\sqrt N},
\qquad z\in\Sigma_b,\ s\in\R^d,
\]
defines global normal coordinates: \(\Phi\) is an isometry from \(\Sigma_b\times\R^d\) onto \((\R^d)^N\), \(\mathcal J_\Phi\equiv1\), \(\dist(\Phi(z,s),\Sigma_b)=|s|\) and \(E(\Phi(z,s))=N|s|^2\). In particular, \(\Hn{z}=2N\,I_d\), and the remainder \(\mathcal R\) of the Taylor expansion of \(E\) vanishes.

\begin{theorem}[Repulsive harmonic cost in \(\R^d\)]
\label{thm: harmonic-Rd}
Assume (M). Then the conclusions of \Cref{thm: convergence-Gibbs-measure} and of \hyperref[thm: A]{Theorem~A} hold, with \(\Pi(\mu_1,\ldots,\mu_N)\) in place of \(\Pi(\mu_1,\mu_2)\), \(k=d\), and
\[
\mathbb P_0\defeq\frac{\rho\,\mathcal H^{(N-1)d}_{|\Sigma_b}}{\int_{\Sigma_b}\rho\dd\mathcal H^{(N-1)d}},
\qquad
Z_0\defeq(2N)^{-d/2}\int_{\Sigma_b}\rho\dd\mathcal H^{(N-1)d}.
\]
Moreover, the selected plan is \(\gamma_M\), and
\[
\min\mathcal F_\varepsilon=\OT(\mu_1,\ldots,\mu_N)+\frac d2\varepsilon\log(1/\varepsilon)-\varepsilon\Bigl(\entropy(\mu_{i_0})+\frac d2\log\pi\Bigr)+o(\varepsilon),
\]
where \(\entropy(\mu_{i_0})\defeq\int\rho_{i_0}\log\rho_{i_0}\dd x\).
\end{theorem}

\begin{proof}
We follow the proofs of \Cref{thm: euclidean-Gibbs,thm: euclidean-selection}, with \(\Sigma_b\) in place of the \(d\)-dimensional contact set: these proofs do not use the dimension of the contact set, and they are simpler here, since the normal coordinates are global and \(E\) is exactly quadratic in them.

\textbf{Step 1. Gibbs measures and \(\Gamma\)-\(\liminf\).} The proof of \Cref{thm: euclidean-Gibbs} applies with \(\mathcal H^{(N-1)d}\) in place of \(\mathcal H^d\), \(\mathcal R\equiv0\), \(\mathcal J_\Phi\equiv1\) and \(\mathcal T_r=\Sigma_r\). Its assumption \(\mathcal H^{(N-1)d}(\Sigma_b\cap\partial\mathcal K)=0\), where \(\mathcal K=\prod_iK_i\), follows from \(\mathcal L^d(\partial K_i)=0\): for \(j\neq i\), parametrizing \(\Sigma_b\) by the coordinates \((x_l)_{l\neq j}\), the set \(\Sigma_b\cap\{x_i\in\partial K_i\}\) is the image of a negligible set by a Lipschitz map. Moreover, \(Z_0>0\), since \(\rho>0\) \(\gamma_M\)-a.e.\ and \(\gamma_M\ll\mathcal H^{(N-1)d}_{|\Sigma_b}\) (see Step~2). Since \(\det\Hn{z}=(2N)^d\) is constant, this gives \(\gibbs\rightharpoonup\mathbb P_0\) and \(Z_\varepsilon=(2\pi\varepsilon)^{d/2}(Z_0+o(1))\), and the \(\Gamma\)-\(\liminf\) inequality follows as in Step~1 of \Cref{subsec: selection-value-asymptotics}.

\textbf{Step 2. A finite-entropy optimal plan.} Parametrizing \(\Sigma_b\) by \((x_i)_{i\neq i_0}\), whose surface Jacobian is \(N^{d/2}\), we get \(\gamma_M=N^{-d/2}\prod_{i\neq i_0}\rho_i(x_i)\,\mathcal H^{(N-1)d}_{|\Sigma_b}\). Hence
\begin{equation}
\label{eq: harmonic-Rd-density}
\frac{\dd\gamma_M}{\dd\mathbb P_0}(\x)=\frac{C_M}{\rho_{i_0}(x_{i_0})},
\qquad
C_M\defeq N^{-d/2}\int_{\Sigma_b}\rho\dd\mathcal H^{(N-1)d},
\end{equation}
and \(\entropy(\gamma_M\,|\,\mathbb P_0)=\log C_M-\entropy(\mu_{i_0})<\infty\), since \(\rho_{i_0}\) is bounded with compact support. This replaces Step~3 of \Cref{subsec: selection-value-asymptotics}, where the bound \(\overline\gamma\leq\beta\mathcal H^d_{|\Sigma}\) used \(\dim\Sigma=d\).

\textbf{Step 3. The marginals.} Fix \(\overline\gamma\in\Pi(\mu_1,\ldots,\mu_N)\) with \(\entropy(\overline\gamma\,|\,\mathbb P_0)<\infty\), and \(R\geq1\). Since \(\Hn{z}=2N\,I_d\) does not depend on \(z\), neither does the kernel \(K_R(z,s)=K_R(s)\propto e^{-N|s|^2}\mathbf 1_{\ball_R}(s)\). Hence, by \eqref{eq: recovery sequence with wrong marginals 1}, \(\widetilde\gamma_{\varepsilon,R}\) is the law of \(Z+\sqrt{\varepsilon/N}\,(S,\ldots,S)\), where \(Z\sim\overline\gamma\) and \(S\sim K_R\) are independent, and its \(i\)-th marginal is
\[
\widetilde\mu_{i,\varepsilon,R}=\mu_i*\eta_\varepsilon,
\qquad
\eta_\varepsilon\defeq\text{law of }\sqrt{\varepsilon/N}\,S.
\]
Writing \(\sqrt{\varepsilon/N}\,S=\sqrt{\varepsilon/(2N^2)}\,S'\), where \(S'=\sqrt{2N}\,S\) has density proportional to \(e^{-|s'|^2/2}\) on \(\ball_{\sqrt{2N}R}\), we are in the setting of \Cref{prop: 1D}, with \(\varepsilon/(2N^2)\) and \(\sqrt{2N}R\) in place of \(\varepsilon\) and \(R\). Its proof applies verbatim in \(\R^d\), and gives \(W_2(\widetilde\mu_{i,\varepsilon,R},\mu_i)=o_R(\sqrt\varepsilon)\).

\textbf{Step 4. Energy and entropy.} The rest of the proof of \Cref{thm: euclidean-selection} is unchanged. The energy estimate of \Cref{prop: shadow-energy} holds, since \(|\nabla E(\mathbf y)|=2N\dist(\mathbf y,\Sigma_b)\) for every \(\mathbf y\in(\R^d)^N\). In Step~3 of the proof of \Cref{thm: euclidean-selection}, the remainder and Jacobian terms vanish, and \(\int|\log\overline\rho|\dd\overline\gamma<\infty\) follows from \(\entropy(\overline\gamma\,|\,\mathbb P_0)<\infty\) and \(\int|\log\rho|\dd\overline\gamma\leq\sum_i\int\rho_i|\log\rho_i|\dd x<\infty\).

\textbf{Step 5. Selection and value.} By \eqref{eq: harmonic-Rd-density}, \(\dd\gamma_M/\dd\mathbb P_0\) only depends on \(x_{i_0}\). Hence, for every \(\gamma\in\Pi(\mu_1,\ldots,\mu_N)\) with \(\entropy(\gamma\,|\,\mathbb P_0)<\infty\), as in the proof of \Cref{prop: MPW-selection},
\[
\entropy(\gamma\,|\,\mathbb P_0)=\entropy(\gamma\,|\,\gamma_M)+\int\log\frac{C_M}{\rho_{i_0}}\dd\mu_{i_0}=\entropy(\gamma\,|\,\gamma_M)+\entropy(\gamma_M\,|\,\mathbb P_0),
\]
so that \(\gamma_M\) is the unique minimizer of \(\entropy(\cdot\,|\,\mathbb P_0)\) on \(\Pi(\mu_1,\ldots,\mu_N)\), that is, the selected plan. Finally, since \(C_M=2^{d/2}Z_0\),
\[
\min\mathcal F^{(1)}=\entropy(\gamma_M\,|\,\mathbb P_0)-\log Z_0-\frac d2\log(2\pi)
=-\entropy(\mu_{i_0})-\frac d2\log\pi,
\]
and the expansion of \(\min\mathcal F_\varepsilon\) follows from \hyperref[thm: A]{Theorem~A}.
\end{proof}

\begin{remark}
Condition (M) only involves the marginals: it means that \(\mu_{i_0}\) is the image of \(\mu_1*\cdots*\mu_{i_0-1}*\mu_{i_0+1}*\cdots*\mu_N\) by \(x\mapsto b-x\); in particular, \(\rho_{i_0}\) is then automatically continuous. Since a convex combination of plans concentrated on \(\Sigma_b\) with bounded densities is again such a plan, all the conclusions of \Cref{thm: harmonic-Rd}, except the explicit selected plan and value, remain valid when \((\mu_1,\ldots,\mu_N)\) is a convex combination of \(N\)-tuples of marginals satisfying (M) with the same \(b\). For \(N=2\), (M) means that \(\mu_2\) is the image of \(\mu_1\) by \(x\mapsto b-x\), and the cost is equivalent to \(|x_1+x_2|^2\): the expansion above is then consistent with \Cref{prop: quadratic-expansion}, after the change of variables \(x_2\mapsto-x_2\) and the rescaling \(\varepsilon\mapsto\varepsilon/2\).
\end{remark}

\noindent{\textbf{Acknowledgments.}} L.N., P.P. and L.T. acknowledge  the support of the FMJH Program PGMO and the ANR project GOTA (ANR-23-CE46-0001). B.P. is pleased to acknowledge the support of Natural Sciences and Engineering Research
Council of Canada Discovery Grant number 04864-2024. The authors used AI-assisted tools to help check proofs and ensure consistency of notation, and take full responsibility for the final manuscript.

\printbibliography
\end{document}